\documentclass[10pt]{amsart}
\usepackage{enumitem}

\numberwithin{equation}{section}

\usepackage{amscd}
\usepackage{overpic}
\usepackage{amsmath,amsfonts,amsthm,epsfig,latexsym,graphicx,amssymb}
\usepackage[english]{babel}
\usepackage{color}
\usepackage{verbatim}

\newtheorem{theorem}{Theorem}[section]
\newtheorem{lemma}[theorem]{Lemma}
\newtheorem{claim}[theorem]{Claim}

\newtheorem{proposition}[theorem]{Proposition}
\newtheorem{corollary}[theorem]{Corollary}

\newtheorem{mtheo}{Theorem}

\theoremstyle{definition}
\newtheorem{definition}[theorem]{Definition}
\newtheorem{remark}[theorem]{Remark}

\newtheorem{question}{Question}

\DeclareMathOperator{\interior}{int}
\DeclareMathOperator{\Diff}{Diff}
\DeclareMathOperator{\id}{id}
\DeclareMathOperator{\supp}{supp}
\newcommand{\eqdef}{\stackrel{\scriptscriptstyle\rm def}{=}}

\def\bN{\mathbb{N}}
\def\bZ{\mathbb{Z}}
\def\bR{\mathbb{R}}
\usepackage{xcolor}
\definecolor{forestgreen}{rgb}{0.0, 0.27, 0.13}

\def\cs{\mathrm{cs}}
\def\cu{\mathrm{cu}}
\def\u{\mathrm{u}}
\def\s{\mathrm{s}}
\def\ss{\mathrm{ss}}
\def\uu{\mathrm{uu}}
\def\c{\mathrm{c}}

\def\A{A}

\DeclareMathOperator{\Leb}{Leb}
\DeclareMathOperator{\diam}{diam}

\newcommand{\llbracket}{[\![}
\newcommand{\rrbracket}{]\!]}

\def\R{I\kern -0.37 em R}
\def\N{I\kern -0.37 em N}
\def\Z{I\kern -0.37 em Z}

\def\supess_#1{\mathop{\rm supess}\limits_{#1}}
\def\infess_#1{\mathop{\rm infess}\limits_{#1}}

\def\EE{{\mathbb E}}

 \def\RR{{\mathbb R}}

 \def\ZZ{{\mathbb Z}}

\def\Si{\Sigma}
\def\La{\Lambda}
\def\De{\Delta}

\def\Ga{\Gamma}

   \def\cM{{\mathcal M}}  
     
\def\cC{{\mathcal C}}   \def\cO{{\mathcal O}} \def\cU{{\mathcal U}}
\def\cD{{\mathcal D}}    
    \def\cW{{\mathcal W}}
\def\cF{{\mathcal F}}    

\def\fB{\mathfrak{B}}
\def\fF{\mathfrak{F}}
\def\ff{\mathfrak{f}}

\def\fR{\mathfrak{R}}

\title[Heterodimensional cycles of measures]{\bf Heterodimensional cycles\\ of hyperbolic ergodic measures.} 

\author[Ch.~Bonatti]{Ch. Bonatti}
\address{Institut de Math\'{e}matiques de Bourgogne, UMR 5584 du CNRS, Universit\'{e} de
Bourgogne, 21000, Dijon, France}
\email{bonatti@u-bourgogne.fr}

\author[L.~J.~D\'iaz]{L. J. D\'\i az}
\address{Departamento de Matem\'atica PUC-Rio, Marqu\^es de S\~ao Vicente 225, G\'avea, Rio de Janeiro 22451-900, Brazil}
\email{lodiaz@mat.puc-rio.br}
\author[K.~Gelfert]{K.~Gelfert}
\address{Instituto de Matem\'atica Universidade Federal do Rio de Janeiro, Av. Athos da Silveira Ramos 149, Cidade Universit\'aria - Ilha do Fund\~ao, Rio de Janeiro 21945-909,  Brazil}\email{gelfert@im.ufrj.br}

\thanks{The IMB receives support from the EIPHI Graduate School (contract ANR-17-EURE-0002).
This research has been supported [in part] by 
CAPES -- Finance Code 001, by 
CNPq-grants  310069/2020-3 and 
305327/2022-4,  
CNPq Projeto Universal 430154/2018-6  and 
404943/2023-3, 
E-16/2014 INCT/FAPERJ and  
E-26/200.371/2023 CNE/FAPERJ (all Brazil), 
and the French-Brazilian Network in Mathematics (GDRI-RFBM)} 
\begin{document}

\begin{abstract}
We introduce the concept of a heterodimensional cycle of hyperbolic ergodic measures and a special type of them that we call rich.
Within a partially hyperbolic context, we prove that if two measures are related by a rich heterodimensional cycle, then the entire segment of probability measures linking them lies within the closure of measures supported on periodic orbits. 
Motivated by the occurrence of robust heterodimensional cycles of hyperbolic basic sets, we study robust rich heterodimensional cycles of measures providing a framework for this phenomenon for diffeomorphisms.
 In the setting of skew products, we construct an open set of maps having uncountably many measures related by rich heterodimensional cycles.
\end{abstract}


\maketitle
\tableofcontents

\section{Introduction}

The study of individual orbits for chaotic dynamics, by its very nature, does not allow to capture the global complexity of the system. Often, it is more meaningful to concentrate on the ``statistical'' behavior and to study the measures invariant under the dynamics, especially the ergodic ones. Note that each ergodic measure encodes information about a ``cluster'' of orbits (its generic points) and in this way, the study of certain types of orbits turns naturally into a study of the space of measures. We focus on ``heteroclinic'' relations between measures that mimic the corresponding relation between orbits. To be more precise, given an ergodic measure $\mu$ for some dynamics $f$, a point $x$ is (forward) \emph{$\mu$-generic} if it belongs to the \emph{stable set} (of $\mu$)
\[
	W^\s(\mu,f) 
	\eqdef \Big\{y\colon \frac1n(\delta_y+\delta_{f(y)}+\ldots+\delta_{f^{n-1}(y)})\to\mu\Big\}.
\] 
Analogously, for an invertible map $f$, we can define the \emph{unstable set} $W^\u(\mu,f)$ considering backward orbits. Two ergodic measures $\mu,\nu$ are \emph{heteroclinically related} if their associated sets $W^\s(\cdot)$ and $W^\u(\cdot)$  intersect cyclically, that is,
\begin{equation}\label{cycleofmeasures}
	W^\s(\mu,f) \cap W^\u(\nu,f) \ne\emptyset,\quad
	W^\u(\mu,f) \cap W^\s(\nu,f) \ne\emptyset.
\end{equation}
We aim to understand which are the implication of such sort of intersections. This configuration is somewhat reminiscent of the homoclinic/heteroclinic intersections between saddles (hyperbolic periodic points) which is at the core of the understanding of chaotic dynamics and, in particular,  hyperbolic dynamics. By means of Smale's spectral decomposition theorem \cite{Sma:67}, the orbital behavior is understood for basic%
\footnote{A set is \emph{basic} if it is compact, invariant, topologically transitive, locally maximal, and hyperbolic.}
 sets. Here the dynamics is organized around periodic orbits and the homoclinic/heteroclinic intersection of their invariant manifolds. Bowen's results \cite{Bow:08} provide a corresponding symbolic description. From the point of view of measures, the latter results have a counterpart that we proceed to discuss. 

For a continuous map $f$ on a compact metric space,  the space $\cM(f)$ of all invariant Borel probability measures equipped with the weak$\ast$ topology is a Choquet simplex and its extreme points are the ergodic measures, $\cM_{\rm erg}(f)$. For $f$ having the periodic specification property (in the differentiable setting, it represents essential features of hyperbolicity and is satisfied for basic sets) Sigmund \cite{Sig:74} proved that \emph{periodic measures} (those supported on periodic orbits) are dense in $\cM(f)$. Therefore, this space is the Poulsen simplex and has a rich topological structure \cite{LinOlsSte:78}. To give an oversimplifying idea how to obtain the results in \cite{Sig:74} in a hyperbolic context, note that any pair of measures $\mu,\nu$ presents very rich intersections as in \eqref{cycleofmeasures} (provided by homoclinic intersections of invariant manifolds of the hyperbolic dynamics). The symbolic codification then enables to  ``almost freely concatenate'' (this is essentially the meaning of specification) orbit segments of points contained in these intersections. This allows  to ``almost freely interpolate'' between two extreme points $\mu,\nu\in\cM_{\rm erg}(f)$ approximating the segment
\begin{equation}
\label{e.convexhull}
	[\mu,\nu]
	\eqdef \{s\mu+(1-s)\nu\colon s\in[0,1]\}
	\subset\cM(f).
\end{equation}
 in the weak$\ast$ topology. 

The above-described structure is only true for uniformly hyperbolic dynamics. We aim to go beyond uniform hyperbolicity and place Sigmund's results into a broader context. 
To make this more rigorous, in what follows, we consider a diffeomorphism $f$ on some compact manifold. Recall that an ergodic measure has Lyapunov exponents. An ergodic measure whose exponents are all nonzero is \emph{hyperbolic} and we denote by $\cM_{\rm hyp}(f)$ the set of all such measures. Two ergodic measures have \emph{different type of hyperbolicities} if their numbers of negative exponents (counted with their multiplicity) are different. There exist systems that simultaneously present ``transitive pieces of different types of hyperbolicity" and ``non-hyperbolic pieces'' which may intersect or ``(dynamically) overlap'' in the sense of presenting heteroclinic relations. For those systems, given two ergodic measures $\mu,\nu$, we aim to understand how the segment $[\mu,\nu]$ is placed in $\cM (f)$ and, in particular, if, how, and where it is accumulated by ergodic or periodic measures. The analysis of this question heavily relies on whether or not the measures share the same type of hyperbolicity. We are interested in the case when the measures form a \emph{cycle} as in \eqref{cycleofmeasures}. This cycle is \emph{equidimensional} if the measures have the same type of hyperbolicity and {\em{heterodimensional}} otherwise. This terminology extends the particular case of cycles of basic sets (see \cite{NewPal:73,Dia:95} and the discussed below).

For ergodic measures $\mu, \nu$ with the \emph{same} type of hyperbolicity, under mild ``domination'' hypotheses, \cite{BocBonGel:18} provides a complete geometric answer to the previous questions: the following facts are equivalent:
\begin{itemize}[leftmargin=0.9cm ]
\item[(st1)] there is $t\in(0,1)$ so that $t\mu+(1-t)\nu$ is the limit of ergodic measure,
\item[(st2)] the segment $[\mu,\nu]$ is contained in the closure of the set of ergodic measures,
 \item[(st3)] the measures $\mu$ and $\nu$ are \emph{homoclinically related} in the sense that there are generic points of $\mu$ and $\nu$ whose Pesin stable/unstable manifolds intersect transversely. In the above terminology, they form an equidimensional cycle and, additionally, have transverse intersections.
\end{itemize}
The configuration in (st3) was also investigated in \cite{RodRodTahUre:11} where the term \emph{ergodic homoclinic class} was coined. By \cite{AbdBonCro:11}, for a $C^1$ generic diffeomorphism, for periodic measures $\mu,\nu$ supported in a homoclinic class and of the same type, (st2) holds (relative to the convex closure of all periodic measures in this class). In \cite{GorPes:17}, a further structure such as general path-connectedness is studied.%
\footnote{By \cite{AbdBonCro:11}, $C^1$ generically, two periodic orbits of the same type of hyperbolicity and in the same homoclinic class are homoclinically related. This is the hypothesis used in \cite{GorPes:17}.}

In this paper, we investigate the case of ergodic measure of \emph{different} types of hyperbolicity and aim for results in the spirit of (st1)--(st3). In the seminal works \cite{New:70,AbrSma:70}, nonhyperbolic dynamical systems are described, exhibiting diverse forms of nonhyperbolicity. Indeed, as nonhyperbolicity may occur in many ways, here we focus on the simplest setting which is also the closest to the hyperbolic one. We assume that there is a \emph{partially hyperbolic splitting} of the form $E^\ss\oplus E^\c \oplus E^\uu$, where $E^\c$ is one-dimensional and $E^\ss$ and $E^\uu$ are uniformly contracting and expanding, respectively. Given any ergodic measure $\mu$, the \emph{central direction} $E^\c$ gives rise to a Lyapunov exponent in this direction. Any nonhyperbolicity comes from this direction. The pioneering works \cite{Shu:71,Man:78} provide robustly nonhyperbolic and transitive diffeomorphisms having this sort of splitting. \cite{Pal:05,Bon:11} points out the relevance of cycles in the generation of  nonhyperbolicity. According to \cite{BonDia:08}, in dimensions equal or higher than three, the occurrence of  nonhyperbolicity can be explained by a semi-local mechanism called robust heterodimensional cycle which involves blender-horseshoes and that will be explored below. 

To summarize our results, let us first revisit the setting of heterodimensional cycles between sets. Recall that every hyperbolic set  $\La$ of a diffeomorphism $f$ always has a well-defined continuation $\La_g$ for diffeomorphisms $g$ close to $f$. Analogously to what was said above, we talk about \emph{types of hyperbolicity} of basic sets. Two basic sets $\La$ and $\Ga$ with different types of hyperbolicity form a \emph{heterodimensional cycle} if their invariant sets intersect cyclically (analogously to \eqref{cycleofmeasures}). A special case occurs when $\La$ and $\Ga$ are orbits of periodic points, say $P$ and $Q$. In this case, considering their associated periodic measures, one gets a \emph{heterodimensional cycle of measures}. We point out that, without further assumptions about the ``interaction" of the measures $\mu, \nu$  in such a cycle, even in the periodic case, few can be said about the segment $[\mu,\nu]$. For instance, no measure in the ``open segment'' $(\mu,\nu)$ may be accumulated or the accumulation points may contain only one point in the open segment (see the examples in Section \ref{appnonrich}). To exclude these exceptions, we introduce the concept of \emph{rich heterodimensional cycle of ergodic measures}, (Definition~\ref{defrichcycmeas}).  The key feature of such cycles is that the strong un-/stable manifolds of some generic points of the measures are involved. Due to the hypothesis of partial hyperbolicity, such manifolds are indeed well-defined. In Figure \ref{figcycles} (right figure), we depict the simplest examples of a rich heterodimensional cycle, where the measures involved in the cycle are periodic ones.

\begin{figure}[h] 
 \begin{overpic}[scale=.7]{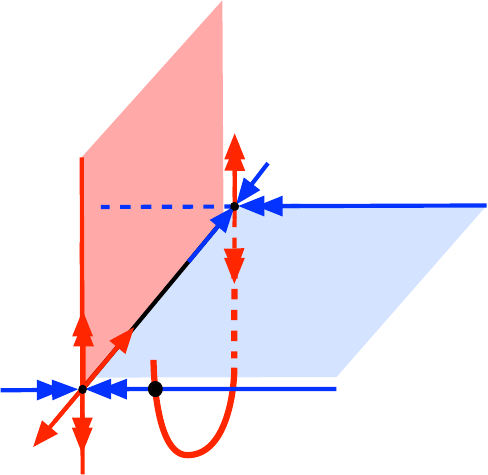}
 	\put(7,23){$Q$}
 	\put(37,59){$P$}
\end{overpic}
\hspace{1cm}
 \begin{overpic}[scale=.7]{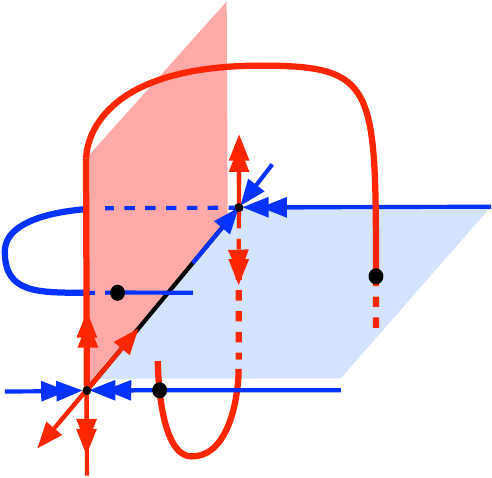}
 	\put(7,23){$Q$}
 	\put(37,59){$P$}
	\put(73,80){\textcolor{red}{$W^\uu(Q)$}}
	\put(-10,56){\textcolor{blue}{$W^\ss(P)$}}
\end{overpic}
\caption{``Sparse'' heterodimensional cycle (left figure) and rich heterodimensional cycle (right figure) between 
 measures $\delta_P$ and $\delta_Q$ associated to the saddles $P$ and $Q$}
\label{figcycles}
\end{figure}
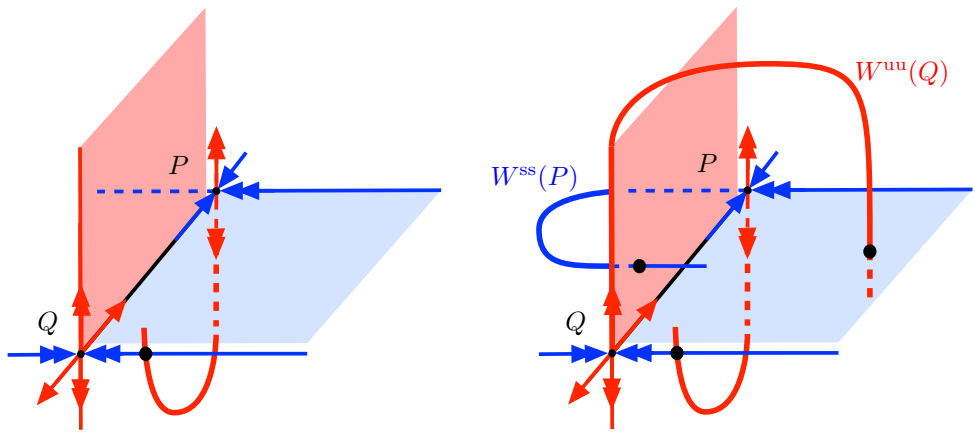

\begin{mtheo}\label{mmt.cycle} 
Let $f\in \Diff (M)$. Consider measures $\mu^\pm\in \cM_{\rm hyp}(f)$ related by a rich heterodimensional cycle. Then for every $\nu \in [\mu^-, \mu^+]$ there is a sequence of hyperbolic periodic measures $(\mu_{\nu,n})_n$ converging to $\nu$ in the weak$\ast$ topology, as $n\to +\infty$. 
\end{mtheo}

As pointed out above, an important issue concerning cycles, either of sets or of measures, is ``robustness''. In what is below, the term \emph{robust} refers to the fact that the involved objects persist after small perturbations and depend continuously. A cycle associated to $\La$ and $\Ga$ is \emph{$C^r$-robust} if for every $g$ which is $C^r$-close to $f$ the continuations $\La_g$ and $\Ga_g$ of $\Lambda$ and $\Gamma$ form a heterodimensional cycle. Heterodimensional cycles associated to a pair of  basic sets%
\footnote{The partially hyperbolic assumption implies that these cycles have co-index one, very few is known in the higher co-index case.}
generate $C^r$-robust cycles by arbitrarily small perturbations (see \cite{BonDia:08} in the $C^1$-case 
and \cite{LiTur:} for higher differentiability). In the discussion of robustness, the fact that a hyperbolic set has a well-defined continuous continuation is crucial.

To achieve ``robust'' heterodimensional cycles of measures, it is essential to first define measure continuations. While there is not yet an adequate framework for this definition, there are some feasible cases. Besides the obvious case of continuations of saddle periodic measures, one can also consider SRB measures of hyperbolic attractors/repellers. The latter cannot be related by cycles, but they can be used as constructing plugs. We define \emph{saddle-SRB measures} associated to \emph{saddle-attractors and -repellers}. To give a rough description, consider a hyperbolic attractor on a surface and its SRB measure. Embedding this attractor in a normally expanding surface one gets a saddle-attractor and the corresponding saddle-SRB measure. Every saddle-attractor $\Lambda^-$ of a $C^2$ diffeomorphism supports a saddle-SRB measure $\mu^-$. Moreover, the set and the measure have well-defined continuous continuations for every $C^2$-close diffeomorphism $g$, denoted by $\Lambda^-_g$ and $\mu^-_g$, respectively. Analogously, for saddle-repellers and their saddle-SRB measures. We present a setting having \emph{robust rich heterodimensional cycles of measures}. We postpone further details to Section \ref{sec:SRB}. 

\begin{mtheo}[Robust cycles of measures]\label{mmt.SRBrobust} 
Let $M$ be a three-dimensional compact manifold. There is an open set $\mathcal U$ of the space of $C^2$-diffeomorphisms on $M$ consisting of diffeomorphisms such that every $f\in \mathcal U$ has a saddle-attractor basic set $\La^-_f$ and a saddle-repeller basic set $\La^+_f$ whose associated saddle-SRB measures $\mu^\pm_f$ have a rich heterodimensional cycle. As a consequence, for every $f\in\mathcal U$ every measure in the segment $[\mu^-_f,\mu^+_f]$ is accumulated by ergodic measures. 
\end{mtheo}

Let us slightly change the perspective of robustness and present another class of examples. Here we introduce skew product maps $F$ which have two uncountable sets of measures $\Leb^-(F)$ and $\Leb^+(F)$ such that every measure $\mu^-\in\Leb^-(F)$ and $\mu^+\in\Leb^+(F)$ are related by a rich heterodimensional cycles. Moreover, every nearby $G$ has also this property. However, it is unclear if some of these pairs of measures $\mu^\pm_F$  have some continuations to corresponding measures of $G$. The heart of this construction are blender-horseshoes with different types of hyperbolicity in step skew products and with ``overlaps''. We postpone further details to Section \ref{s.robustblenders}. 

\begin{mtheo}\label{mmt.skew} 
	Assume that $F$ is a step skew product with overlapping blenders and $C^2$ fiber maps. Then the convex hull of $\Leb^-(F)\cup\Leb^+(F)$ is contained in the closure of the set of periodic measures. Moreover, this property holds for every step skew product map sufficiently $C^2$-close to $F$.
\end{mtheo}

Finally, although the above result deals with toy models of blender-horseshoes involved in a heterodimensional cycle, the approximation of convex combinations of measures of the same type of hyperbolicity in \cite{BonZha:19} sheds some light on the study of measures arising in robust cycles.
In a $C^1$-open setting, similar to the one in this paper, in \cite{BonZha:19} it is shown that the closures of the sets of hyperbolic ergodic measures of both possible types of hyperbolicity (positive or negative center Lyapunov exponent) is convex and their intersection is the closure of the set of non-hyperbolic ergodic measures. There are also similar global results for some classes of robustly transitive diffeomorphisms, where the approximation of measures of the segments is done also in entropy,  see \cite{DiaGelSan:20,YanZha:20}. 
 
Having in mind the robust cycles associated with blender-horseshoes in \cite{BonDia:08}, 
to show the existence of rich heterodimensional cycles of measures supported on those blender-horseshoes and to state a version of Theorem~\ref{mmt.skew} in a broader context of diffeomorphisms, not only for skew products, remains a challenge.  We close this introduction with some questions posed in the partially hyperbolic context of this paper. 

\begin{question} \label{q1}
In the context of a diffeomorphism with robust heterodimensional cycles, what convexity properties can we expect for the closure of the set of ergodic measures supported on the cycle? Specifically, given two 
such measures $\mu$ and $\nu$, under what conditions the segment $[\mu, \nu]$ is contained in the closure of ergodic/periodic measures?
\end{question}

This question is very general. In Section~\ref{sec:acip} it will be reframed in more practical terms. Roughly, it asks about the existence of measures supported on blender-horseshoes having absolutely continuous-like properties similar to those of saddle-SRB measures.

Versions of Question \ref{q1} considering the dependence on the dynamics become relevant, as they lead to  \emph{robust cycles of measures}. In the following question, determining the necessary degree of differentiability of the diffeomorphisms is also a part of the question.

\begin{question}\label{q2}
Consider a $C^r$-open set $\cU$ consisting of diffeomorphisms $f$ with two hyperbolic sets $\La^+_f$ and $\La^-_f$ with different types of hyperbolicity forming a robust heterodimensional cycle. Are there ergodic measures $\mu^+_f,\mu^-_f$ supported on $\La^+_f$ and $\La^-_f$, respectively, depending continuously on $f$ and are also related by a heterodimensional cycle? Additionally, can we choose such a cycle to be rich?
\end{question}

This paper is organized as follows. In Section \ref{sec2}, we define rich heterodimensional cycles between measures. Theorem \ref{mmt.cycle} is proven in Section \ref{sec3}.  In Section \ref{sec:SRB}, we define saddle-SRB measures, provide the construction of Plykin saddle-attractors, and prove Theorem \ref{mmt.SRBrobust}. In Section \ref{s.robustblenders}, we investigate robust rich heterodimensional cycles that arise from blenders in skew products and prove Theorem \ref{mmt.skew}. Section \ref{appnonrich} presents further examples that illustrate that the hypothesis of a \emph{rich} heterodimensional cycle is indeed necessary to have convexity properties between the measures that give rise to the cycle. Appendix \ref{appA} contains an auxiliary result about convex sums used in the proof of Theorem~\ref{mmt.skew}.

\section{Heterodimensional cycle between measures}\label{sec2}

Throughout this paper, $M$ denotes a closed Riemannian manifold with metric $d$. Denote by $\Diff(M)$   the space $C^1$ diffeomorphisms of $M$ and endow this space with the uniform topology. Consider the space of Borel probability measures on $M$, denoted by $\cM$. Given $f\in \Diff(M)$, we consider the compact and convex subset $\cM(f)\subset\cM$ of all $f$-invariant Borel probability measures. We denote by $\cM_{\rm erg}(f)$ its subset of ergodic measures. An ergodic measure is \emph{periodic} if it is supported on a periodic orbit. The set of periodic measures of $f$ is denoted by $\cM_{\rm per}(f)$. We choose on $\cM$ some distance $D$ defining the weak$\ast$ topology.

Given $f\in \Diff(M)$ and an ergodic probability measure $\mu$, the Oseledets theorem asserts that there are an $f$-invariant set $\Lambda$ with $\mu(\Lambda)=1$, an $Df$-invariant splitting $E_{\mu,1}\oplus \cdots\oplus E_{\mu,k}$ defined over $\Lambda$, and numbers $\chi_1<\ldots<\chi_k$, called \emph{Lyapunov exponents} of $\mu$, such that $\lim_{n\to\pm\infty}\frac1n\log\,\lVert Df^n|_{E_{\mu,i}(x)}\rVert=\chi_i$, for every $i=1,\ldots,k$ and $x\in\Lambda$. 

An ergodic measure $\mu$ is  \emph{hyperbolic} if all of its Lyapunov exponents are different from zero. In this case, the \emph{stable bundle} $E^\s_\mu$ associated to $\mu$ is the sum of the Oseledets spaces $E_{\mu,i}$ corresponding to the negative Lyapunov exponents. Similarly, the \emph{unstable bundle} $E^\u_\mu$ is defined considering the positive exponents. The dimension of $E^\s_\mu$ is the \emph{$\s$-index} of $\mu$. Associated to each $x\in\Lambda$ there exist the \emph{Pesin stable manifold} tangent to $E^\s_\mu(x)$ and the \emph{Pesin unstable manifold} tangent to $E^\u_\mu(x)$, \cite{Pes:77}. Note that for every $y\in W^\s(x)$, 
\[
	\limsup_{n\to\infty}\frac1n\log d(f^n(y),f^n(x))<0,
\]
analogously for $W^\u(x)$ considering negative iterations. Note that if $f\in\Diff^{1+\varepsilon}(M)$ then the Pesin un-/stable manifolds are nontrivial. If, however, $f$ is only a $C^1$ diffeomorphism, then these manifolds may consist only of single points (see, \cite{Pug:84,BonCroShi:13}).

We denote the subset of hyperbolic ergodic measures by $\cM_{\rm hyp}(f)$. We call an ergodic measure \emph{uniformly hyperbolic} if it is hyperbolic and its support is a uniformly hyperbolic set. Note that in this case, the stable Oseledets bundle of $\mu$ coincides with the stable hyperbolic bundle over $\supp(\mu)$; analogously for the unstable bundle.

Given $x\in M$ and $n\in\bN$, denote by $\delta_y$ the Dirac measure supported at $y$ and introduce the following ``Birkhoff averages'' 
\begin{equation}\label{defBn}
	\fB_n(x,f)
	\eqdef \frac1n\sum_{k=0}^{n-1}\delta_{f^k(x)},\quad
	\fB_{-n}(x,f)
	\eqdef \frac1n\sum_{k=-n+1}^{0}\delta_{f^k(x)}.
\end{equation}
We simply write $\fB_n(x)$ and $\fB_{-n}(x)$ if the map $f$ is clear from the context.
Let $\fB_\infty(x)$ denote the set of limit points of the sequence $(\fB_n(x))_n$ (in the weak$\ast$ topology), analogously define $\fB_{-\infty}(x)$ to be the set of limit points of the sequence $(\fB_{-n}(x))_{n\in\bN}$. As we assume that $M$ is compact, these sets are always nonempty. It is straightforward to check that for every $x$, $\fB_{\pm\infty}(x)$ are closed and connected subsets of $\cM(f)$. If $\mu$ is ergodic, then $\fB_\infty(x)=\{\mu\}$ for $\mu$-almost every $x$; any point with this property is called a \emph{generic point} of $\mu$. 

We now define a cycle between hyperbolic measures (compare also Figure \ref{figcycles} (left figure)). This definition mimics the one of cycles between hyperbolic sets.

\begin{definition}[Cycle between measures]\label{defcycmeas}
Two measures $\mu,\nu \in  \cM_{\rm hyp}(f)$ \emph{are related by a cycle} (or \emph{there is a cycle between $\mu$ and $\nu$}) if there are $\mu$-generic points $x_1,x_2$ and $\nu$-generic points $y_1, y_2$  such that
\begin{enumerate}
\item[(i)] $W^\s(x_1)\cap W^\u(y_1)\neq \emptyset$,
\item[(ii)] $W^\u(x_2)\cap W^\s(y_2)\neq \emptyset$. 
\end{enumerate}
The cycle is called \emph{heterodimensional} if the measures have different $\s$-indices and \emph{equidimensional} otherwise. 
\end{definition}

\begin{figure}[h] 
 \begin{overpic}[scale=.5]{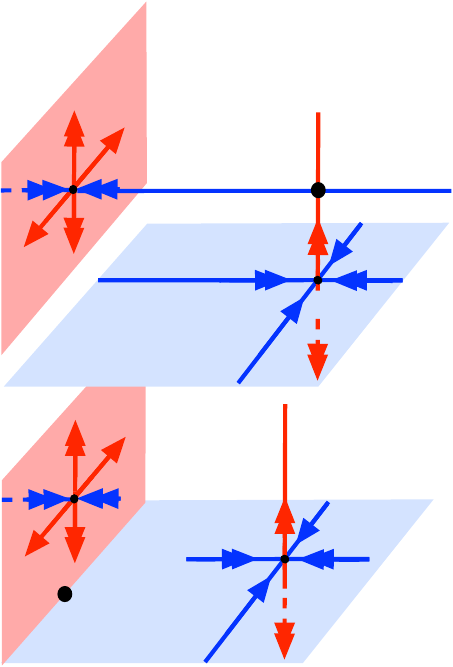}
	\put(5,75){$y_2$}
	\put(5,29){$y_1$}
	\put(50,53){$x_2$}
	\put(45,11){$x_1$}
	\put(45,85){\textcolor{red}{$W^\u(x_2)$}}
	\put(62,75){\textcolor{blue}{$W^\s(y_2)$}}
	\put(62,18){\textcolor{blue}{$W^\s(x_1)$}}
	\put(22,33){\textcolor{red}{$W^\u(y_1)$}}
	\put(-8,20){(i)}
	\put(-8,60){(ii)}
\end{overpic}
\hspace{2cm}
 \begin{overpic}[scale=.5]{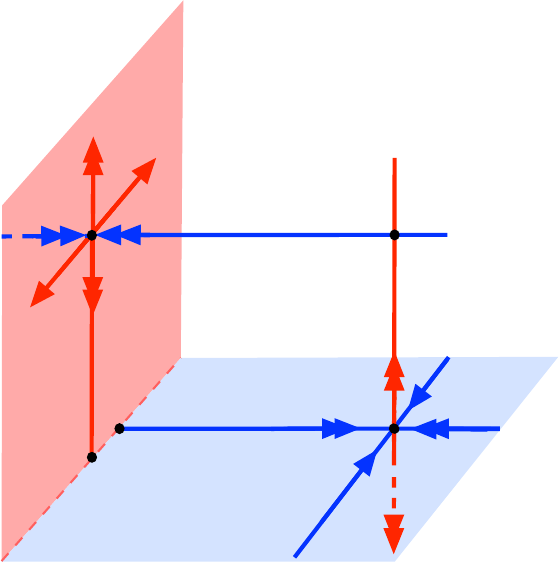}
	\put(72,16){$x^-$}
	\put(19,50){$x^+$}
	\put(65,75){\textcolor{red}{$W^\uu(x^-)$}}
	\put(80,59){\textcolor{blue}{$W^\ss(x^+)$}}
\end{overpic}
\caption{Heterodimensional cycle between measures satisfying conditions (i) and (ii) in Definition \ref{defcycmeas} (left figure) and rich heterodimensional cycle (right figure)}
\label{figcyclesbis}
\end{figure}
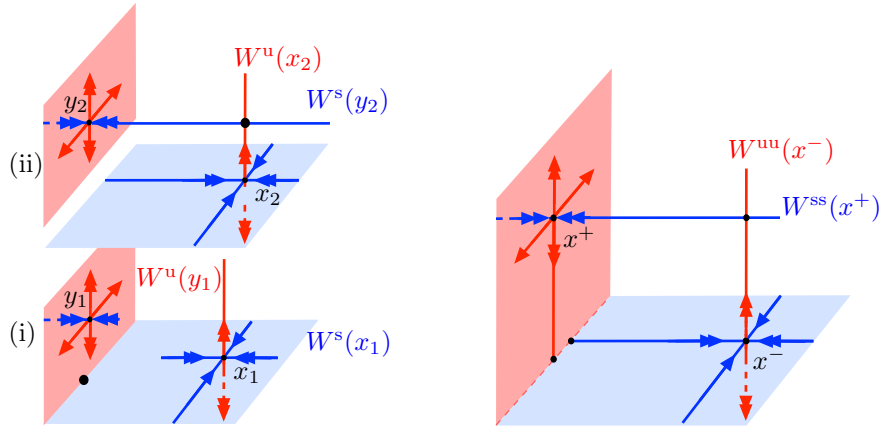

Note that in the above definition \emph{a priori} no transverse intersection is required. In the equidimensional case, when both intersections are transverse then the measures are \emph{homoclinically related}

In what follows, we will focus on heterodimensional cycles, only. Let us assume that in Definition \ref{defcycmeas}, the $\s$-index of $\mu$ is bigger than the one of $\nu$. Note that in this case in the intersection (i) there is an ``excess of dimension'' in the sense that $\dim W^\s(x_1)+\dim W^\u(y_1)>\dim M$, while in the intersection (ii) there is a ``deficiency'' in the sense that $\dim W^\u(x_2)+\dim W^\s(y_2)<\dim M$. 

In what is below, we consider a set $\Lambda$ with a partially hyperbolic splitting $E^\ss\oplus E^\c\oplus E^\uu$ with three non-trivial directions, where $E^\ss$ is uniformly contracting, $E^\uu$ is uniformly expanding, and $\dim E^\c=1$. More precisely, following \cite{Gou:07}, there is a metric (equivalent to the original metric on $M$) which induces a norm on the tangent bundle such that for every $x\in \Lambda$,
\begin{equation}\label{defdomspl}
	\lVert Df|_{E^\ss(x)}\rVert 
	< m(Df|_{E^\c(x)}) = \lVert Df|_{E^\c(x)}\rVert 
	< m(Df|_{E^\uu(x)}) \le \lVert Df|_{E^\uu(x)}\rVert,
\end{equation}
where $m(A)=\lVert A^{-1}\rVert^{-1}$ denotes the  minimum norm (compare \eqref{conorm}; the equality is a consequence of the fact that $E^\c$ is one-dimensional). Note that in this case, for every hyperbolic ergodic measure $\mu$, the Oseledets subbundles satisfy either $E^\s_\mu=E^\ss\oplus E^\c$ and $E^\u_\mu=E^\uu$ or $E^\s_\mu=E^\ss$ and $E^\u_\mu=E^\c\oplus E^\uu$ over the support of $\mu$.
Moreover, the Lyapunov exponent of $\mu$ corresponding to $E^\c$ is well defined and denoted by $\chi^\c (\mu)$. 

\begin{remark}\label{r.strongmanifolds}
Under the above conditions, for every point $x\in \Lambda$, there are defined the \emph{strong stable} and \emph{strong unstable manifolds}, denoted by $W^\ss(x)$ and $W^\uu(x)$, that are tangent to $E^\ss(x)$ and $E^\uu(x)$. These manifolds depend continuously on $x$.
\end{remark}

Given two hyperbolic measures $\mu^\pm$ related by a heterodimensional cycle, we aim to study how the segment $[\mu^-,\mu^+]$ is approximated by ergodic measures. Without further assumptions, in general, nothing can be said (in Section \ref{appnonrich} we provide some paradigmatic examples illustrating different approximation scenarios). It turns out that the strong un-/stable manifolds play an important role. Indeed, when they are ``involved in a cycle'', much more can be said about $[\mu^-,\mu^+]$. This motivates the following definition.

\begin{definition}[Rich heterodimensional cycle]\label{defrichcycmeas}
Given measures $\mu^\pm \in  \cM_{\rm hyp}(f)$ related by a heterodimensional cycle, we say that this  cycle is \emph{rich} if:
\begin{enumerate}
\item $\mu^\pm$ both are uniformly hyperbolic and satisfy $\chi^\c(\mu^-)<0<\chi^\c(\mu^+)$, 
\item there exists a partially hyperbolic set $\Lambda$ with three nontrivial directions and a one-dimensional central bundle containing the supports of $\mu^-$ and $\mu^+$,
\item for every $x\in \supp(\mu^+)$ and every $y\in\supp(\mu^-)$ it holds
\[
	\big(W^\uu(x) \pitchfork W^\s(y)\big)\cap \Lambda \ne \emptyset
	\quad \mbox{and}\quad
	\big(W^\u(x) \pitchfork W^\ss(y)\big) \cap \Lambda\ne \emptyset,
\]
\item there is a $\mu^-$-generic point $x^-\in \Lambda$ and a $\mu^+$-generic point $x^+\in \Lambda$ such that 
\[
 	W^\s(x^+)\cap W^\u(x^-)\cap \Lambda \ne \emptyset.
\]
\end{enumerate}
\end{definition}

Note that, under the conditions in Definition \ref{defrichcycmeas}, it holds $W^\ss(x)=W^\s(x)$ for every $x\in \supp(\mu^+)$ and $W^\uu(y)=W^\u(y)$ for every $y\in\supp(\mu^-)$.

In Sections  \ref{sec:SRB} and \ref{s.robustblenders}, we provide two classes of maps that present rich heterodimensional cycles between measures. Both classes are, in an appropriate sense, robust. Indeed, in these examples, the intersections in the cycle are ``abundant''. 

\begin{remark}[``Sparse'' versus rich heterodimensional cycles]
Comparing Definition \ref{defrichcycmeas} with Definition \ref{defcycmeas}, we point out the following differences. For that assume again that the $\s$-index of $\mu$ is bigger than the one of $\nu$ (compare also Figure \ref{figcyclesbis}): In the former, all intersections occur inside a partially hyperbolic set. Moreover, intersection (i) is replaced by the intersections in (3); note that in (3) there is no ``excess of dimension''. Also, we require this intersection for every pair of points (not only for some pair of generic points). Finally, intersection (ii) corresponds to item (4). 

The hypothesis (2) of uniform hyperbolicity of the support of the measures seems restrictive. We hope to weaken this requirement. 

Compare Figure \ref{figcycles} (right figure), in which the Dirac measures $\delta_P$ and $\delta_Q$ are related by a rich heterodimensional cycle.
\end{remark}

Theorem \ref{mmt.cycle} follows from the following result. 

\begin{theorem}\label{t.cycle} 
Let $f\in \Diff (M)$. Consider measures $\mu^\pm\in \cM_{\rm hyp}(f)$ related by a rich heterodimensional cycle. Then for every $\nu \in [\mu^-, \mu^+]$ there is a sequence of hyperbolic periodic measures $(\mu_{\nu,n})_n$ converging to $\nu$ in the weak$\ast$ topology, as $n\to +\infty$. Moreover, the periodic measures can be chosen to be supported on periodic orbits in an arbitrarily small neighborhood of $\La$.
\end{theorem}

\section{Accumulation of measure segments in rich cycles}\label{sec3}

In this section, we prove Theorem \ref{t.cycle}.
We start by recalling Gan's shadowing lemma, see Section \ref{secGan}. After some preliminaries in Section \ref{secprelim}, estimates of Birkhoff averages in Section \ref{secBirk}, and construction of quasi-hyperbolic pseudo-orbits in Section \ref{secquasihyp}, finally the proof of Theorem \ref{t.cycle} is completed in Section \ref{secfimprova}.

\subsection{Gan's shadowing lemma}\label{secGan}

It is well-known that pseudo-orbit nearby hyperbolic sets can be shadowed by true orbits provided that the ``jumps'' of the pseudo-orbits are sufficiently small (see, \cite[Chapter 18]{KatHas:95}). Here we use a generalization due to Gan \cite{Gan:02} assuming some weak form of hyperbolicity.  
It allows to shadow pseudo-orbits obtained by concatenation of \emph{quasi-hyperbolic segments}. A quasi-hyperbolic segment is an orbit segment endowed with a (locally) invariant (locally) dominated splitting $E\oplus  F$ so that the end time of this segment is a Pliss-hyperbolic time for $E$ and the initial time is a Pliss-hyperbolic time for $f^{-1}$ on $F$ (see conditions (1) and (2) in Definition \ref{defGan}).
To state this lemma, we need some preliminary notation and definitions.

Given a linear map $A\colon U\to V$ between two Euclidean spaces, we denote by $\|A\|$ the norm of $A$
by $m(A)$ the minimum norm of $A$, that is,
\begin{equation}\label{conorm}
	m(A) \eqdef \inf\{\|A(v)\| \colon v \in U , \|v\| = 1\}.
\end{equation}
In what is below we consider $f \in \Diff(M )$. Given a point $x\in M$ and $n\geq 0$ we denote 
\[
	\llbracket x, n\rrbracket
	\eqdef\{x, f(x), f^2(x), \ldots , f^n(x)\}
\]	 
the \emph{orbit segment} of length $n+1$.

\begin{definition}[Quasi-hyperbolic orbit segment]\label{defGan}
Given $\lambda>0$, an orbit segment $\llbracket x,n\rrbracket$ of $f$
is \emph{$\lambda$-quasi-hyperbolic} with respect to a splitting $T_x M = E_x \oplus F_x$ if the following holds:
let $E_j = Df^j(E_x)$ and $F_j = Df^j(F_x)$, 
\begin{itemize}[leftmargin=0.8cm ]
\item[(1)] 
$\displaystyle
	\prod_{j=0}^{k-1} \| Df|_{E_j} \| 
		\leq e^{-k\lambda},\mbox{ for every } k = 1, 2, \dots, n$,	
\item[(2)] 
$\displaystyle
	\prod_{j=k}^{n-1} m\left(Df|_{F_j} \right) 
		\geq e^{(n-k)\lambda},\mbox{ for every\,}   k = 0, 1, \dots , n-1$, and	
\item[(3)] 
$\displaystyle
	\frac{\|Df|_{E_j}\|}{m(Df|_{F_j} )}
		\leq e^{-2\lambda}, \mbox{ for every } k = 0, 1, \ldots , n-1.
$		
\end{itemize}
\end{definition}
 
\begin{definition}[Pseudo-orbit]
Given $\lambda\in(0, 1)$ and $\delta>0$, a bi-infinite sequence of orbit segments $\big\{\llbracket x_i, n_i\rrbracket \big\}^\infty_{i=-\infty} $ is a \emph{$\lambda$-quasi-hyperbolic $\delta$-pseudo-orbit} with respect to the sequence of splittings $T_{x_i} M = E_{x_i} \oplus F_{x_i}$, $i\in\bZ$, if the following holds for every $i\in \bZ$: 
\begin{itemize}[leftmargin=0.8cm ]
\item $\llbracket x_i, n_i\rrbracket$ is $\lambda$-quasi-hyperbolic with respect to 
$T_{x_i}M =E_{x_i} \oplus F_{x_i}$ and
\item $d(f^{n_i}(x_i) , x_{i+1} )\le \delta$. 
\end{itemize}
The sequence is \emph{periodic} if there is $m\in\bN$ such that $x_{i+m} = x_i$ and $n_{i+m} = n_i$ for every $i$. The minimal number $m$ with the previous property is the \emph{period} of the sequence.
\end{definition}

\begin{definition}[Shadowing]
Given $\varepsilon>0$, a point $x$ \emph{$\varepsilon$-shadows} a sequence of orbit segments $\big\{\llbracket x_i, n_i\rrbracket \big\}^\infty_{i=-\infty}$ if
\[
	d(f^j(x), f^{j-N_i}(x_i) )\leq \varepsilon,
	\quad\text{ for every }i\in\bZ\text{ and for every } j=N_i,\ldots,N_{i+1}-1,
\] 
where $N_i$ is defined as
\begin{equation}
	N_i
	\eqdef \begin{cases}
             0&\text{if }i=0,\\
             n_0+n_1+\dots+n_{i-1}&\text{if }n>0,\\
             n_i+n_{i+1}+\dots+n_{-1}&\text{if }n<0.
            \end{cases}
\end{equation}
\end{definition}

\begin{theorem}[Gan's shadowing lemma, {\cite{Gan:02}}]\label{theoGan}
Let $f \in \Diff(M )$, $\La\subset M$ a closed $f$-invariant set, and  $T_\La M = E \oplus F$ a continuous $Df$-invariant splitting over $\La$.  

For every $\lambda>0$ there exist $L>0$ and $\delta_0 > 0$ such that for every $\delta \in (0, \delta_0 ]$ and every $\lambda$-quasi-hyperbolic $\delta$-pseudo-orbit $\{\llbracket x_i , n_i \rrbracket\}_{i=-\infty}^\infty$ with respect to the sequence of splittings $E_{x_i} \oplus F_{x_i}$, where $x_i\in\Lambda$ for every $i$, there exists a point $x$ which $L\delta$-shadows $\{\llbracket x_i , n_i\rrbracket \}^\infty_{i=-\infty}$. 

Moreover, if this $\delta$-pseudo-orbit is periodic of period $m$, then the point $x$ can be chosen to be periodic with period $m$.
\end{theorem}

The above theorem is even easier to apply when the splitting is truly dominated, rather than ``locally dominated''. Furthermore, it becomes highly flexible when one of the bundles, $E$ or $F$, is uniformly hyperbolic, as will be in our applications. Indeed, one only needs a single hyperbolic time to obtain a hyperbolic segment (as discussed in Section \ref{secfimprova}, where we consider the dominated splitting $E\oplus F$ into bundles $E=E^\ss\oplus E^\c$ and $F=E^\uu$).

\subsection{Preliminaries towards the proof of Theorem \ref{t.cycle}}\label{secprelim}

In the remainder of this section, let $\mu^\pm$ be two ergodic measures of $f$, $\Lambda\subset M$ be a partially hyperbolic set of $f$ as in Definition \ref{defrichcycmeas}, and $x^\pm\in\Lambda$ be corresponding generic points as in the hypotheses of Theorem \ref{t.cycle}. We first need a quantitative version of the intersections in that definition.  We start by establishing some notation.

\begin{remark}[Large enough local invariant manifolds]\label{r.localR} 
In what is below, we need to consider invariant manifolds which are ``large enough''.  Let us make this precise. For that assume that $\Upsilon$ is a hyperbolic basic set with a splitting of the form $E^\ss\oplus E^\c\oplus E^\uu$, where $E^\s=E^\ss\oplus E^\c$. First recall that, for $r>0$ sufficiently small, the local invariant manifolds
\[
	W^\dagger_r (x)
	\eqdef \{y\colon d(f^\ell(y),f^\ell(x))<r\text{ for all }\ell\ge0\}\cap W^\dagger(x),
\]
$\dagger=\s, \ss$, are well-defined and depend continuously on the points in $\Upsilon$. The sets
\[
	W^\dagger_{r,\ell}(x)
	\eqdef f^{-\ell} \big(W^\dagger_r (f^\ell(x))\big),\quad
	\dagger=\s,\ss,
\]
are nested and contained in $W^\dagger (x)$, contain $W^\dagger_r (x)$, and ``spread'' along the whole set $W^\dagger (x)$ as $\ell$ increases. Analogously, we define 
\[
	W^\dagger_{r,\ell} (y)
	\eqdef f^{\ell} \big(W^\dagger_r (f^{-\ell}(x))\big), \quad \dagger=\uu,
\]	
Note that in this case $W_{r,\ell}^\uu(y)$ is a local unstable manifold.

If $\Upsilon$ is a basic set with a splitting  $E^\ss\oplus E^\c\oplus E^\uu$ such that $E^\u=E^\c\oplus E^\uu$ we define $W_{r,\ell}^\dagger(x)$, $x\in\Upsilon$, for $\dagger\in\{\u,\uu,\ss\}$, analogously.
\end{remark}

Let $O\supset \Lambda$ be some small neighborhood of $\Lambda$ and let $\Lambda'$ be the maximal invariant set contained in $\overline O$. If $O$ was chosen small enough, then the partially hyperbolic splitting $E^\ss\oplus E^\c\oplus E^\uu=T_\Lambda M$ extends to some partially hyperbolic splitting defined on $\Lambda'$. By hypothesis (3) in Definition \ref{defrichcycmeas}, for every $x\in\supp(\mu^+)$ and $y\in\supp(\mu^-)$
\[
	\big(W^\uu(x) \pitchfork W^\s(y)\big)\cap \Lambda \ne \emptyset
	\quad \mbox{and}\quad
	\big(W^\u(x) \pitchfork W^\ss(y)\big) \cap \Lambda\ne \emptyset.
\]
Hence, by Remark \ref{r.localR}, there are $r>0$ and $\ell_0\in\bN$ so that for every $\ell\ge\ell_0$ for every $x\in\supp(\mu^+)$ and $y\in\supp(\mu^-)$
\[
	\big(W^\uu_{r,\ell}(x) \pitchfork W^\s_{r,\ell}(y)\big)\cap \Lambda \ne \emptyset
	\quad \mbox{and}\quad
	\big(W^\u_{r,\ell}(x) \pitchfork W^\ss_{r,\ell}(y)\big) \cap \Lambda\ne \emptyset.
\]
By the continuous variation of the invariant manifolds in these intersections, for every $x'$ near $x$ and $y'$ near $y$, it holds 
\[
	\big(W^\uu_{r,\ell}(x') \pitchfork W^\s_{r,\ell}(y')\big)\cap O \ne \emptyset
	\quad \mbox{and}\quad
	\big(W^\u_{r,\ell}(x') \pitchfork W^\ss_{r,\ell}(y')\big) \cap O\ne \emptyset.
\]
By a compacity argument, for $L\in\bN$ sufficiently large, for \emph{every} $x'\in\supp(\mu^+)$ and \emph{every} $y'\in\supp(\mu^-)$, it holds
\begin{equation}\label{choiceR}
	\big(W^\uu_{r,L}(x') \pitchfork W^\s_{r,L}(y')\big)\cap \Lambda' \ne \emptyset
	\quad \mbox{and}\quad
	\big(W^\u_{r,L}(x') \pitchfork W^\ss_{r,L}(y')\big)\cap \Lambda' \ne \emptyset.
\end{equation}
Possibly after increasing $L$, we can assume that condition (4) in Definition \ref{defrichcycmeas} holds in the following stronger form: 
\[
 	W^\s_{r,L}(x^+)\cap W^\u_{r,L}(x^-)\cap \Lambda' \ne \emptyset
\]
and therefore, we can consider a point 
\begin{equation}\label{defchoicep}
	p\in W^\s_{r,L}(x^+)\cap W^\u_{r,L}(x^-)\cap \Lambda'.
\end{equation}	 
For notational simplicity, we continue to write $\Lambda=\Lambda'$. This concludes the first preliminary step. 

We now proceed to a safety construction. Let $O'\supset \overline O\supset \Lambda$ be some small neighborhood of $\Lambda$ and let $\Lambda''$ be the maximal invariant set contained in $\overline{O'}$. If $O'$ was chosen small enough, then the partially hyperbolic splitting $E^\ss\oplus E^\c\oplus E^\uu=T_\Lambda M$ extends to some partially hyperbolic splitting defined on $\Lambda''$. 

\subsection{Birkhoff averages on shadowing orbits}\label{secBirk}

We start with some auxiliary results.  

\begin{lemma}\label{lem:1}
	There exists $N_0\in\bN$ such that for every $N^+\ge N_0$ and $y\in \supp(\mu^-)$ there is $q\in W^\uu_{r,L}(x^-)$ such that
\[
	f^{N^+}(q)\in W^\uu_{r,L}(f^{N^+}(p))\cap W^\s_{r,L}(y)
	\,\,\text{ and }\,\,
	\{q,\ldots,f^{N^+}(q)\}\subset O,
\]	
where $p$ is as in \eqref{defchoicep}.
Moreover, $q$ can be chosen in the partially hyperbolic set $\Lambda''$.
\end{lemma}

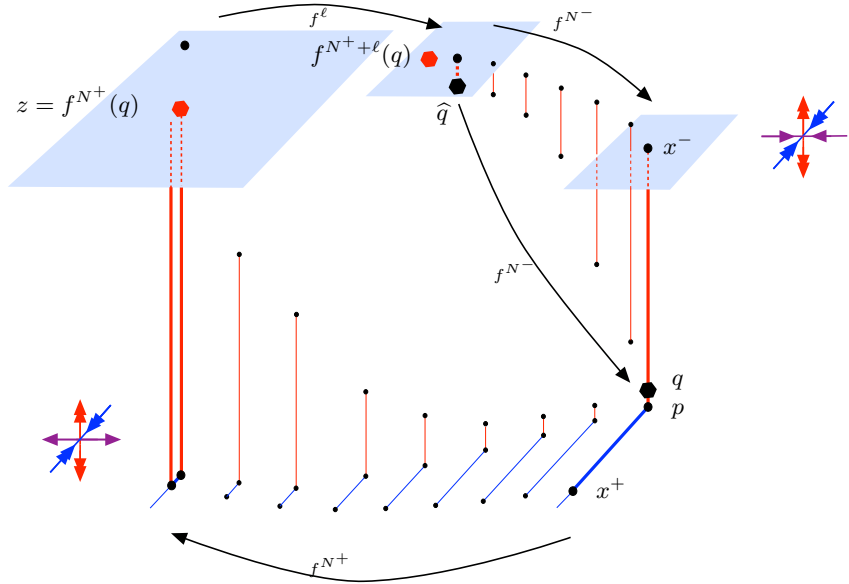
\begin{figure}[h] 
 \begin{overpic}[scale=.6]{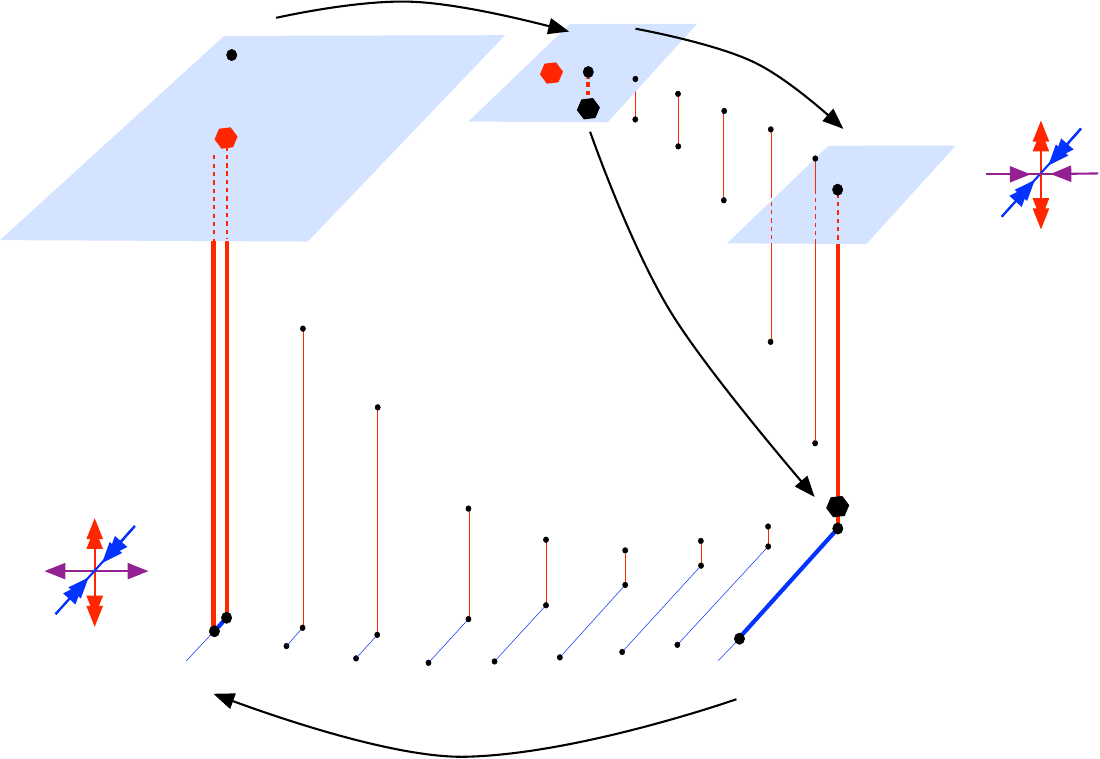}
	\put(35,66.5){\tiny{ $f^{\ell}$}}
	\put(64,66){\tiny{ $f^{N^-}$}}
	\put(57,36){\tiny{ $f^{N^-}$}}
	\put(35,1){\tiny{ $f^{N^+}$}}
	\put(77,51){\small{ $x^-$}}
	\put(36,62){\small{$f^{N^++\ell}(q)$}}
	\put(1,56){\small{$z=f^{N^+}(q)$}}
	\put(79,23.5){\small{$q$}}
	\put(79,20){\small{$p$}}
	\put(70,10){\small{$x^+$}}
	\put(51,55){\small{$\widehat q$}}
\end{overpic}
\caption{Closing a heterodimensional cycle}
\end{figure}

\begin{proof}
Note that $p\in W^\u_{r,L}(x^-)=W^\uu_{r,L}(x^-)$ and
\begin{equation}\label{eqconv}
	d(f^{n}(p),f^{n}(x^+))\to0\text{ as }n\to\infty.
\end{equation}	 
By \eqref{choiceR} applied to the points $f^n(x^+)$ and any $y\in \supp(\mu^-)$ we get that
\[
	\big(W^\uu_{r,L}(f^n(x^+)) \pitchfork W^\s_{r,L}(y)\big)\cap \Lambda' \ne \emptyset.
\]
By \eqref{eqconv}, $W^\uu_{r,L}(f^n(p))$  approaches $W^\uu_{r,L}(f^n(x^+))$ as $n\to+\infty$. 
This implies that there is $N_0$ such that for every $N^+ \ge N_0$ it holds 
\[
	\big(W^\uu_{r,L}(f^{N^+}(p)) \pitchfork W^\s_{r,L}(y)\big)\cap O
	\neq \emptyset, \quad \mbox{for every $y\in \supp(\mu^-)$}.
\]
Given $y\in \supp(\mu^-)$, consider any point 
\[
	z\in  \big(W^\uu_{r,L}(f^{N^+}(p)) \pitchfork  W^\s_{r,L}(y)\big)\cap O,
\]	 
and let $q= q(N^+)= f^{-N^+}(z)\in W^\uu_{r,L}(p)\cap W^\uu_{r,L} (x^-)$. As in this construction, $q$ can be chosen arbitrarily close to $p$ (by increasing $N^+$), it follows that $q$ can be chosen such that the orbit segment $\{q,\ldots,f^{N^+}(q)\}$ is entirely contained in $O$. This proves the first assertion. 

Moreover, $x^\pm$ and $p$ belong to $\Lambda$. This orbit segment is close to a segment of the orbit of $p$ (in $\Lambda$, by \eqref{defchoicep}) and a segment of the orbit of a point in $W^\uu_{r,L}(f^{N^+}(p))\cap W^\s_{r,L}(f^{-N^--\ell}(x^-))$ (in $\Lambda$, by \eqref{choiceR}). This implies $q\in \Lambda''$ and proves the lemma. 
\end{proof}

\begin{lemma}\label{lemdeltaclose}
For every $\delta>0$ there exists $\ell_0=\ell_0(\delta)\in\bN$ such that for every $\ell\ge\ell_0$ there is $N_0(\ell)$ such that for every $N^\pm\ge N_0$ there is a point $q=q(\ell,N^-,N^+)\in W^\uu_{r,L}(x^-)$ which satisfies
\[
	d\big(f^{N^++\ell}(q),f^{-N^-}(q)\big)<\delta.
\]	
\end{lemma}

\begin{proof}
Given $\delta>0$, there is $\ell_0(\delta)\in\bN$ such that for every $\ell\ge\ell_0$ and $z\in \supp(\mu^-)$ we have $\diam(f^\ell(W^\s_{r,L}(z)))<\delta/2$. Fix now $\ell\ge\ell_0(\delta)$ and let $N_0=N_0(\ell)$ as provided by Lemma \ref{lem:1}. Given any $N^\pm\ge N_0$, apply this lemma to the points $y=f^{-N^--\ell}(x^-)$ and $p$ to get a point $q=q(\ell,N^-,N^+)\in W^\uu_{r,L}(x^-)$ satisfying
\[
		f^{N^+}(q)\in W^\uu_{r,L}(f^{N^+}(p))\cap W^\s_{r,L}(f^{-N^--\ell}(x^-)).
\]		 
By construction, it follows $f^{N^++\ell}(q) \in f^\ell(W^\s_{r,L}(f^{-N^--\ell}(x^-)))$. By the first argument in this proof, we get $d(f^{N^++\ell}(q),f^{-N^-}(x^-))<\delta/2$.

Similarly, by construction, $q\in W^\uu_{r,L} (x^-)$ and hence $d(f^{-n}(q),f^{-n}(x^-))\to 0$ as $n\to+\infty$. Hence this distance is smaller that $\delta/2$ for every $N^-$ big enough. The assertion follows now from the triangle inequality.
\end{proof}

Recall \eqref{defBn} defining the Birkhoff averages of Dirac measures. We state the following straightforward facts for further reference.

\begin{lemma}\label{l.weakclose}
	For every $\varepsilon>0$ there is $\tau=\tau(\varepsilon)>0$ such that for every $x,y\in M$ and $n\in\bN$ satisfying $d(f^k(x),f^k(y))<\tau$ for all $k=0,\ldots,n-1$ it holds 
\[
	D(\fB_n(x),\fB_n(y))<\varepsilon.
\]	
\end{lemma}

\begin{lemma}\label{l.birkhoff}
For every $\varepsilon>0$ and $\ell\in\bN$ there is $N_1 =N_1(\varepsilon, \ell) \in\bN$ such that for every $N^\pm\ge N_1$ the point $q=q(\ell,N^-,N^+)$ provided by Lemma \ref{lemdeltaclose} satisfies
\[
	D\big(\fB_{N^-}(f^{-N^-}(q)),\mu^-\big)<\varepsilon \qquad \mbox{and} \qquad
	D\big(\fB_{N^++\ell}(q),\mu^+\big)<\varepsilon.
\]
\end{lemma}

\begin{proof}
The genericity of $x^-$ with respect to $\mu^-$ and the fact that $q\in W^\uu_{r,L} (x^-)$ implies the convergence $\fB_{-n}(q)=\fB_n(f^{-n}(q))\to \mu^-$ as $n\to\infty$. This implies the first inequality.

To prove the second inequality, note that the genericity of $x^+$ for the measure $\mu^+$ and the fact that $p \in W^\ss_{r,L}(x^+)$ implies the convergence  $\fB_{N^+}(p)\to \mu^+$ as $N^+\to \infty$.  Recall that, by construction of $q$, we have $f^{N^+}(q)\in W^\uu_{r,L}(f^{N^+}(p))$. The uniform expansion of the strong unstable bundle implies that most of the iterates of the segment of orbits $\{ p, f(p), \dots, f^{N^+}(p)\}$ and $\{ q, f(q), \dots, f^{N^+}(q)\}$ are very close. Therefore, the probability measures $\fB_{N^+}(p)$ and $\fB_{N^+}(q)$ are also very close. Hence, as $\ell$ is fixed, taking $N^+$ sufficiently large, $\fB_{N^++\ell}(p)$ and $\fB_{N^++\ell}(q)$ are also close.
\end{proof}

For simplicity, write $\chi^\dagger\eqdef\chi^\c(\mu^\dagger)$, $\dagger\in\{-,+\}$. Given a number $s\in(0,1)$, let
\begin{equation}\label{defchis}
	\chi(s)
	\eqdef s\chi^-+(1-s)\chi^+.
\end{equation}

\begin{lemma}\label{l.convexsum}
For every $s\in(0,1)$ such that $\chi(s)<0$, $\varepsilon>0$, and $\ell\in\bN$ there is $N_2=N_2(s,\varepsilon, \ell)$ such that for every  $N^\pm\ge N_2$ satisfying
\begin{equation}\label{e.relations}
		\left|\frac{N^-}{N^-+N^++\ell}-s\right|<\frac\varepsilon2
\end{equation}
the point $q=q(\ell,N^-,N^+)$ provided by Lemma \ref{lemdeltaclose} satisfies
\[
	D \big(\fB_{N^- +N^+ + \ell}( f^{-N^-} (q)), s\mu^- +(1-s) \mu^+\big) < \varepsilon.
\]
\end{lemma}

\begin{proof}
First note that \eqref{e.relations} implies
\begin{equation}\label{e.relationsbis}
		\left|\frac{N^++\ell}{N^-+N^++\ell}-(1-s)\right|<\frac\varepsilon2.
\end{equation}
We claim that, with the choice in \eqref{e.relations}, Lemma~\ref{l.birkhoff} implies the assertion for $N^\pm$ large enough. To see why this is so write
\[\begin{split}
	\fB_{N^- +N^+ + \ell}( f^{-N^-} (q)) 
	&= \frac{N^- \fB_{N^-}( f^{-N^-} (q)) }{N^- +N^+ + \ell}+
		\frac{(N^+ + \ell) \fB_{N^+ + \ell}(q)}{N^- +N^+ + \ell}\\
	\text{\tiny{(by \eqref{e.relations} and \eqref{e.relationsbis})}}\quad
	&= (s+\frac\varepsilon2) \fB_{N^-}( f^{-N^-} (q))  + (1-s+\frac\varepsilon2) \fB_{N^+ + \ell}(q).
\end{split}\]
This together with Lemma~\ref{l.birkhoff} implies the assertion.
\end{proof}

\subsection{Construction of a quasi-hyperbolic pseudo-orbit}\label{secquasihyp}

The goal of this section is to prove Proposition \ref{pl.quasihyp} stated below. Without loss of generality, we can assume that \eqref{defdomspl} holds for all $x\in O'$, that is,
\begin{equation}\label{defdomsplbis}
	\lVert Df|_{E^\ss(x)}\rVert 
	< m(Df|_{E^\c(x)}) = \lVert Df|_{E^\c(x)}\rVert 
	< m(Df|_{E^\uu(x)}) \le \lVert Df|_{E^\uu(x)}\rVert.
\end{equation}
This implies
\begin{equation}\label{domspl3}
	\lambda_0\eqdef
	\max_{x\in O'}\log\frac{\lVert Df|_{E^\c(x)}\rVert }{m(Df|_{E^\uu(x)})}
	<0.
\end{equation}
As the splitting is partially hyperbolic and hence $x\mapsto \lVert Df|_{E^\ss(x)}\rVert$ and $x\mapsto m(Df|_{E^\uu(x)})$ are continuous functions, we can also assume that there is $\lambda^-<0<\lambda^+$ such that for every $x\in O'$, we have
\begin{equation}\label{domsplphy1}
	\lVert Df|_{E^\ss(x)}\rVert < e^{\lambda^-} < 1 < e^{\lambda^+} < m( Df|_{E^\uu(x)}).
\end{equation}
Finally, as $\supp(\mu^-)$ and $\supp(\mu^+)$ are hyperbolic sets with splittings $(E^\ss\oplus E^\c)\oplus E^\uu$ and $E^\ss\oplus (E^\c\oplus E^\uu)$, respectively, we can assume that for every $y\in\supp(\mu^-)$ and $z\in\supp(\mu^+)$,
\begin{equation}\label{domsplphy2}
	\lVert Df|_{E^\c(y)}\rVert < e^{\lambda^-} < 1 <	e^{\lambda^+} < \lVert Df|_{E^\c(z)}\rVert.
\end{equation}

Let us consider the potential 
\[	
	\psi\colon M\to\bR, \quad\psi(x)\eqdef \lVert Df|_{E^\c(x)}\rVert.
\]	 
Note again that domination implies that $\psi$ is continuous. Let $p$ be provided by \eqref{defchoicep}. Observe that, recalling that $\chi^\pm=\chi^\c(\mu^\pm)$, this implies
\begin{equation}\label{expsp}
	\lim_{n\to\infty}\frac1n\log \prod_{i=0}^{n-1}\psi(f^{-n}(p))
	= \chi^- 
	\,\text{ and }\,
	\lim_{n\to\infty}\frac1n\log \prod_{i=0}^{n-1}\psi(f^{n}(p))
	= \chi^+ .
\end{equation}

 The following is the main result in this section. Recall $\chi(s)$ defined in \eqref{defchis}.

\begin{proposition}\label{pl.quasihyp}
	Consider $s\in(0,1)$ such that $\chi(s)<0$. Then there is $\lambda>0$ such that for every $\varepsilon>0$ and $\ell\in\bN$ there is $N_3 =N_3(s,\varepsilon, \ell) \in\bN$ such that for every $N^\pm\ge N_3$ satisfying
\[
		\left|\frac{N^-}{N^-+N^++\ell}-s\right|<\frac\varepsilon2
\]
 the point $q=q(\ell,N^-,N^+)$ provided by Lemma \ref{lemdeltaclose} has the following property. The orbit segment $\llbracket f^{-N^-}(q),N^-+N^++\ell\rrbracket$ is $\lambda$-quasi hyperbolic with respect to the splitting $E\oplus F=(E^\ss\oplus E^\c)\oplus E^\uu$.
\end{proposition}

\begin{proof}
We first fix some auxiliary quantifiers. Without loss of generality, we can assume that the numbers $\lambda^\pm$ satisfying \eqref{domsplphy1} and \eqref{domsplphy2} also satisfy
\begin{equation}\label{fixexpon}
	\max\{\lvert\lambda^-\rvert,\lambda^+\}
	< \lvert\chi(s)\rvert .
\end{equation}
Let
\begin{equation}\label{choicelambda}
	\lambda
	\eqdef\frac12\min\{|\lambda_0|,|\lambda^-|,\lambda^+\}
	> 0
	\quad\text{ and }\quad
	\tau\eqdef\max\{|\chi^-|,\chi^+\}.
\end{equation}
Let $\varepsilon>0$ be sufficiently small such that 
\begin{equation}\label{eqeps}
	0
	< \lambda
	< \min\{|\lambda^-|,\lambda^+\} -3\varepsilon-\varepsilon\tau.	
\end{equation}
As $\psi$ is continuous, there exists $\delta>0$ such that for all $x\in \Lambda'$, $d(y,x)<\delta$ implies
\begin{equation}\label{eqdelta}
	\max\Big\{\frac{\psi(x)}{\psi(y)},\frac{\psi(y)}{\psi(x)}\Big\}
	<e^\varepsilon.
\end{equation}
As the splitting $E^\ss\oplus E^\c\oplus E^\uu$ is partially hyperbolic, there exists $m_0\in\bN$ such that for every $x\in \Lambda'$, $y\in W^\ss_{r,L}(x)$, and $z\in W^\uu_{r,L}(x)$, for every $n\ge m_0$, it holds
\begin{equation}\label{eqveryclose}
	d(f^n(z),f^n(x))\le\frac\delta2
	\quad\text{ and }\quad
	d(f^{-n}(y),f^{-n}(x))\le\frac\delta2.
\end{equation}
Fix some constant
\begin{equation}\label{constC}
	C
	\ge \max_{x\in \Lambda'}\psi(x).
\end{equation}
Note that \eqref{expsp} implies that, possibly after increasing $C$, for all $n\in\bN$
\begin{equation}\label{expop}
	\prod_{i=0}^{n-1}\psi(f^{-i}(p))
	\le Ce^{n(\chi^-+\varepsilon)}
	\,\text{ and }\,
	\prod_{i=0}^{n-1}\psi(f^i(p))
	\le Ce^{n(\chi^++\varepsilon)}.
\end{equation}
Finally, choose $n_0\in\bN$ such that
\begin{equation}\label{choiceC}
	C^{2(m_0+\ell)}
	\le e^{n\varepsilon}
	\quad\text{ for all $n\ge n_0$}.
\end{equation}
Assume that
\begin{equation}\label{fixNpm}
	N^\pm
	\ge n_0+m_0
\end{equation}
and that 
\begin{equation}\label{choiceNpm}
	\left|\frac{N^-}{N^-+N^++\ell}-s\right|<\frac\varepsilon2.
\end{equation}
This finishes our preliminary choice of constants and observations.

Notice that \eqref{eqveryclose} together with $p\in W^\uu_{r,L}(x^-)$ implies that 
\[
	d(f^{-N^-+n}(p),f^{-N^-+n}(x^-))
	\le \frac\delta2
		\quad\text{ for every }n=1,\ldots,N^--m_0.
\]
Analogously, $f^{N^+}(q)\in W^\uu_{r,L}(f^{N^+}(p))$ implies that 
\begin{equation}\label{pqclose}
	d(f^n(q),f^n(p))
	\le \frac\delta2
		\quad\text{ for every }n=-N^-,\ldots,0,\ldots,N^+-m_0.
\end{equation}

Let $\widehat q\eqdef f^{-N^-}(q)$. Consider the splitting $E\oplus F=(E^\ss\oplus E^\c)\oplus E^\uu$. We divide the proof of the proposition into two lemmas.

\begin{lemma}\label{lemclaim1}
	 Property (1) in Definition \ref{defGan} holds for the segment $\llbracket \widehat q,N^-+N^++\ell\rrbracket$. 	
\end{lemma}

\begin{proof} 
By \eqref{defdomsplbis}, it is enough to check property (1) for the bundle $E^\c$. For every $n=0,\ldots,N^--m_0-1$, we get 
\begin{eqnarray}
	\psi(f^n(\widehat q\,))=\psi(f^{n-N^-}( q))
	&=& \frac{\psi(f^{n-N^-}( q))}{\psi(f^{n-N^-}(p))} \frac{\psi(f^{n-N^-}(p))}{\psi(f^{n-N^-}(x^-))}
		\psi(f^{n-N^-}(x^-))\notag\\
	\text{\tiny{(by \eqref{eqdelta})}}\quad	
	&\le& e^{2\varepsilon}\psi(f^{n-N^-}(x^-))	\notag\\
	\text{\tiny{(by \eqref{domsplphy2})}}\quad	
	&\le& e^{2\varepsilon}e^{\lambda^-}
	= e^{\lambda^-+2\varepsilon}.
	\label{somenu1}
\end{eqnarray}
Note also that, by \eqref{constC}, for every $n=N^--m_0,\ldots,N^-$, we have
\begin{equation}\label{somenu2}
	\psi(f^n(\widehat q\,))
	\le C.
\end{equation}
Let us now check Property (1) in the following cases according to the value of $k$.
\smallskip

\noindent\textbf{Case $k\in\{1,\ldots,N^-\}$.}
By \eqref{choiceC} and the fact that \eqref{fixNpm} gives $N^--m_0\ge n_0$, we get  
\begin{equation}\label{Gan1}\begin{split}
	\prod_{j=0}^{k-1}\lVert Df|_{E^\c(f^j(\widehat q\,))}\rVert
	&= \prod_{j=0}^{k-1}\psi(f^j(\widehat q\,))\\
	\text{\tiny{(by \eqref{somenu1})}}\quad
	&\le e^{k(\lambda^-+2\varepsilon)}\\
	\text{\tiny{(by \eqref{eqeps})}}\quad
	&< e^{-k\lambda}
\end{split}\end{equation}
for every $k=1,\ldots,N^-$.
\smallskip

\noindent\textbf{Case $k\in\{N^-+1,\ldots,N^-+N^+-m_0\}$.}
For every $n=-N^-,\ldots,N^+-m_0$, \eqref{pqclose} and \eqref{eqdelta} imply that 
\begin{equation}\label{formulaa}
	\psi(f^n(q))
	= \frac{\psi(f^n(q))}{\psi(f^n(p))}\cdot\psi(f^n(p))
	\le e^\varepsilon \cdot\psi(f^n(p)).
\end{equation}
This, implies that for all $k=N^-+1,\ldots,N^-+N^+-m_0$
\[\begin{split}
	 \prod_{j=0}^{k-1}\psi(f^j(\widehat q\,))
	&=\prod_{j=0}^{N^--1}\psi(f^j(\widehat q\,))\cdot \prod_{j=N^-}^{k-1}\psi(f^j(\widehat q\,)) \\
	\text{\tiny{(by the definition of $\widehat q$)}}\quad
	&=\prod_{j=N^-}^{1}\psi(f^{-j}(q))\cdot \prod_{j=0}^{k-N^--1}\psi(f^j(q)) \\ 
	\text{\tiny{(by \eqref{formulaa})}}\quad
	&=e^{N^-\varepsilon}\prod_{j=1}^{N^-}\psi(f^{-j}( p))\cdot 
		e^{(k-N^-)\varepsilon}\prod_{j=0}^{k-1-N^-}\psi(f^j( p)) \\
	\text{\tiny{(by \eqref{expop})}}\quad
	&\le e^{N^-\varepsilon} Ce^{N^-(\chi^-+\varepsilon)}\cdot 
		e^{(k-N^-)\varepsilon}Ce^{(k-N^-)(\chi^++\varepsilon)}\\
	\text{\tiny{(by \eqref{choiceC} together with $N^-\ge n_0$)}}\quad
	&\le e^{N^-(\chi^-+3\varepsilon)}\cdot e^{(k-N^-)(\chi^++3\varepsilon)}.
\end{split}\]
Let us write $k-N^-=h$. Note that  $h<N^++\ell$,
\[
	\chi^-<0,\quad
	\frac{N^-}{N^-+h}>\frac{N^-}{N^-+N^++\ell}, 
	\quad\text{ and }\quad
	\chi^+>0.
\]	 This implies 
\[\begin{split}
	\frac{N^-\chi^-+h\chi^+}{N^-+h}
	&= \frac{N^-}{N^-+h}\chi^-+\frac{h}{N^-+h}\chi^+
	\le \frac{N^-}{N^-+N^++\ell}\chi^-+\frac{N^++\ell}{N^-+N^++\ell}\chi^+\\
	\text{\tiny{(by \eqref{choiceNpm})}}\quad
	&\le (s-\varepsilon)\chi^-+(1-s+\varepsilon)\chi^+\\
	\text{\tiny{(by \eqref{choicelambda})}}\quad
	&\le \chi+\varepsilon\max\{|\chi^-|,\chi^+\}
	= \chi+\varepsilon\tau.
\end{split}\]
Hence, 
\begin{equation}\label{theusa}
	\frac{N^-\chi^-+h\chi^+}{N^-+h}
	\le \chi+\varepsilon\tau.
\end{equation}

Continuing with the above estimate, we get
\begin{equation}\label{Gan2}\begin{split}
	\prod_{j=0}^{k-1}\lVert Df|_{E^\c(f^j(\widehat q\,))}\rVert
	= \prod_{j=0}^{k-1}\psi(f^j(\widehat q\,))
	&\le e^{N^-(\chi^-+3\varepsilon)}\cdot e^{(k-N^-)(\chi^++3\varepsilon)}\\
	\text{\tiny{(using that $k-N^-=h$ and \eqref{theusa})}}\quad
	&= e^{N^-\chi^-+h\chi^+}e^{k3\varepsilon}
	\le e^{(N^-+h)(\chi+\varepsilon\tau)}e^{k3\varepsilon}\\
	&= e^{k(\chi +3\varepsilon+\varepsilon\tau)}\\
	\text{\tiny{(by \eqref{fixexpon} and \eqref{eqeps})}}\quad
	&< e^{-k\lambda}
\end{split}\end{equation}
for all $k=N^-+1,\ldots,N^-+N^+-m_0$.
\smallskip

\noindent\textbf{Case $k\in\{N^-+N^+-m_0+1,\ldots,N^-+N^+-m_0+\ell\}$.}
We check that \eqref{Gan2} implies
\begin{equation}\label{Gan3}\begin{split}
	\prod_{j=0}^{k-1}\lVert Df|_{E^\c(f^j(\widehat q\,))}\rVert
	&\le e^{k(\chi +3\varepsilon+\varepsilon\tau)}C^{2(m_0+\ell)}\\
	\text{\tiny{(by \eqref{choiceC} and \eqref{fixNpm})}}\quad
	&\le e^{k(\chi+3\varepsilon+\varepsilon\tau)}\\
	\text{\tiny{(by \eqref{eqeps})}}\quad
	&< e^{-k\lambda}
\end{split}\end{equation}
for all $k=N^-+N^+-m_0+1,\ldots,N^-+N^++\ell$. This proves the lemma.
\end{proof}

\begin{lemma}\label{lemclaim2}
	Properties (2) and (3) in Definition \ref{defGan} hold for the orbit segment $\llbracket \widehat q,N^-+N^++\ell\rrbracket$. 	
\end{lemma}

\begin{proof}
	Assuming that $N^\pm\ge N_0(\ell)$, Lemma \ref{lem:1} asserts that this orbit segment is entirely contained in $O$. Hence, the property (2) follows from \eqref{domsplphy1}.  Property (3) is an immediate consequence of \eqref{domspl3} and \eqref{choicelambda}.
\end{proof}	

Lemmas \ref{lemclaim1} and \ref{lemclaim2} together prove the assertion of the proposition.
\end{proof}

\subsection{End of the proof of Theorem \ref{t.cycle}}\label{secfimprova}

It is enough to study the following two cases: $\chi(s)<0$ and $\chi(s)>0$. We only study the former one. The latter case is an immediate consequence considering $f^{-1}$ instead of $f$ and the splitting $E'\oplus F'$ given by $E'\eqdef E^\ss$ and $F'\eqdef E^\c\oplus E^\uu$ instead of the splitting $E\oplus F$ given by $E\eqdef E^\ss\oplus E^\c$ and $F\eqdef  E^\uu$.

Let $s\in(0,1)$ such that $\chi(s)<0$. Fix $\varepsilon>0$. 
Let $\lambda>0$ be as in Proposition \ref{pl.quasihyp}.
Let $L>0$ and $\delta_0>0$ be as provided by Theorem \ref{theoGan} applied to this value $\lambda$, the set $\Lambda'$, and the splitting $E\oplus F$ given by $E\eqdef E^\ss\oplus E^\c$ and $F\eqdef  E^\uu$. Let $\tau=\tau(\varepsilon)>0$ be as provided by Lemma \ref{l.weakclose}.

Fix now $\delta\in(0,\delta_0]$ sufficiently small such that $L\delta<\tau$. 
Let $\ell_0=\ell_0(\delta)\in\bN$ as in Lemma \ref{lemdeltaclose}. 
Fix $\ell\ge \ell_0$. 
Let $N_0=N_0(\ell),N_1=N_1(\varepsilon,\ell),N_2=N_2(s,\varepsilon,\ell)$, and $N_3=N_3(s,\varepsilon,\ell)\in\bN$ as provided by Lemmas \ref{lemdeltaclose}, \ref{l.birkhoff}, and  \ref{l.convexsum}, and by Proposition \ref{pl.quasihyp}, respectively.
Fix numbers $N^\pm\ge \max\{N_0,N_1,N_2,N_3\}$.  
Consider the point $q=q(\ell,N^-,N^+)$ as in Lemma \ref{lemdeltaclose}. By this lemma, the segment $\llbracket f^{-N^-}(q),N^-+N^++\ell\rrbracket$ is a $\delta$-pseudo-orbit. By Proposition \ref{pl.quasihyp}, this segment is also $\lambda$-quasi hyperbolic. Let $n=N^-+N^++\ell$ and $\widehat q=f^{-N^-}(q)$. Hence, by Theorem \ref{theoGan}, the orbit segment $\llbracket \widehat q,n\rrbracket$  is $L\delta$-shadowed by some periodic orbit segment $\llbracket x,n\rrbracket$ (which is of period $n$). By our choice of $\delta$, both orbit segments $\llbracket x,n\rrbracket$ and $\llbracket \widehat q,n\rrbracket$ are $\tau$-close. Hence, by Lemma \ref{l.weakclose}, both probability measures $\fB_n(x)$ and $\fB_n(\widehat q\,)$ are $\varepsilon$-close. By Lemma \ref{l.convexsum}, the measure $\fB_n(\widehat q\,)$ is $\varepsilon$-close to the measure $s\mu^-+(1-s)\mu^+$. Hence, the measure $\mu_n(x)\eqdef\fB_n(x)$ is $2\varepsilon$-close to $s\mu^-+(1-s)\mu^+$. As $\varepsilon$ is arbitrary, this proves the theorem.
\qed

\section{Robust  and rich cycles of between saddle-SRB measures}\label{sec:SRB}

In this section, we introduce a setting where ergodic measures involved in a heterodimensional cycle have natural continuations that also form a cycle. This establishes robust heterodimensional cycles between measures. Key concepts we explore are \emph{saddle-attractor/repeller basic sets} and the class of \emph{saddle-SRB measures} supported on them. These sets and measures have well-defined continuations.  The detailed definitions and properties are postponed to Section \ref{secsaddleSRB}.

The following is the main result of this section that implies Theorem \ref{mmt.SRBrobust}. Note that the $C^2$ differentiability hypothesis is essential to deal with saddle-SRB measures.

\begin{theorem}[Robust cycles of measures]\label{t.SRBrobust} 
Every three-dimensional compact manifold supports a $C^2$ diffeomorphism $f$ having a saddle-attractor basic set $\La^-$ and saddle-repeller basic set $\La^+$ whose associated saddle-SRB measures $\mu^-=\mu_{\La^-}$ and $\mu^+=\mu_{\La^+}$ satisfy:
\begin{itemize}[leftmargin=0.4cm ]
\item $\mu^-$ and $\mu^+$ have a rich heterodimensional cycle,
\item for every $C^2$ diffeomorphism $g$ sufficiently $C^2$-close to $f$, the saddle-SRB measures $\mu^-_g=\mu_{\La^-_g}$ and $\mu^+_g=\mu_{\La^+_g}$ (continuation of  $\mu^-$ and $\mu^+$, respectively)  are related by a rich heterodimensional cycle.
\end{itemize}
\end{theorem}

This section is organized as follows. Section \ref{secsaddleSRB} introduces \emph{saddle-attractors/repellers basic sets} and \emph{saddle-SRB measures}. In Section \ref{secCycles}, we establish a robustness property, stated in Theorem~\ref{tp.SRBrobust}, and derive Theorem~\ref{t.SRBrobust} from it. In Section~\ref{ss.Plykinsaddle}, we construct cycles between saddle-SRB measures. In Section~\ref{ss.proofofpSRBrobust}, we prove that these cycles can be inserted into a partially hyperbolic set, completing the proof of Theorem~\ref{tp.SRBrobust}.

\subsection{Saddle-SRB measures}\label{secsaddleSRB}

In this section, we define saddle-SRB measures and state some auxiliary properties of them.

\begin{definition}[Saddle-attractor/repeller]\label{d.saddle-attractor}
Let $f$ be a $C^1$-diffeomorphism on a three-manifold and $\La$ be a basic set of $f$. We say that $\La$ is a \emph{saddle-attractor basic set} if 
\begin{itemize}[leftmargin=0.4cm ]
\item there is a partially hyperbolic splitting $T_\La M=E^\ss\oplus E^\cu\oplus E^\uu$ in three one-dimensional bundles, where $E^\ss$ is the stable bundle and $E^\cu\oplus E^\uu$ is the unstable bundle of $\Lambda$, respectively,
\item there is a surface with boundary $S$ that is $f$-strictly invariant (that is, $f(S)$ is contained in the interior of $S$), transverse to $E^\uu$, and $f$-normally hyperbolic such that it is expanding in the ``normal'' direction,
\item $\La\subset S$ and $\La$ is a hyperbolic attractor for the restriction $f|_S$. 
\end{itemize}

We say that $\La$ is a \emph{saddle-repeller basic set} if it is a saddle-attractor basic set for $f^{-1}$. 
\end{definition}

Note that if $\Lambda$ is a saddle-attractor then its unstable direction is two-dimensional and the stable one is one-dimensional, analogously for saddle-repellers. 

When the basic set $\Lambda$ in Definition \ref{d.saddle-attractor} is a Plykin attractor/repeller relative to $f|_S$ (see Section~\ref{sss.Plya} for a definition), then we call the resulting saddle-attractor/repeller a \emph{Plykin saddle-attractor/repeller}. This sort of set will play an important role in Section \ref{secCycles}. Indeed, the basic sets in Theorem \ref{tp.SRBrobust} are of Plykin type. 

\begin{remark}[Notation and terminology]
We will consider a setting where there is a saddle-attractor basic set $\Lambda^+$ and a saddle-repeller basic set  $\Lambda^-$ such that the union $\Lambda^+\cup\Lambda^-$ has a partially hyperbolic splitting $E^\ss\oplus E^\c\oplus E^\uu$. By Remark~\ref{r.strongmanifolds}, for every $x\in\Lambda^\pm$, there are defined its strong stable manifold  $W^\ss(x)$ and its strong unstable manifold $W^\uu(x)$. These manifolds are one-dimensional. For $x\in\Lambda^+$, $W^\ss(x)=W^\s(x)$ and $W^\uu(x) \subsetneq W^\u(x)$. Moreover, $W^\s(x)$ is one-dimensional and $W^\u(x)$ is two-dimensional. Analogously, for $y\in\Lambda^-$, $W^\ss(y)\subsetneq W^\s(y)$ and $W^\uu(y) =W^\u(y)$; moreover $W^\u(y)$ is one-dimensional and $W^\s(y)$ is two-dimensional. When the stable and unstable manifolds are two-dimensional, in some cases, we call them \emph{weak stable} and \emph{weak unstable} to emphasize the difference between the strong ones. 
\end{remark}

Recall the notation in Remark \ref{r.localR}.

\begin{remark}[Strong foliations]\label{r.stablefoliation}
Given a saddle-attractor basic set $\La^+$, the $f$-invariant surface $S$ in Definition \ref{d.saddle-attractor} can be written as
\[
	S= W^\s_{r,L}(\La^+)
	=\bigcup_{x\in\La^+} W^\ss_{r,L}(x).
\]	
We consider the strong stable foliation $\cW^\ss_{\La^+}\eqdef\{W^\ss_{r,L}(x)\}_{x\in\La^+}$ of $W^\s_{r,L}(\Lambda^+)$.%
\footnote{In fact, more accurately, this would be called the ``stable foliation''. However, as we consider a partially hyperbolic context with three directions and these leaves are tangent to the strong stable bundle, we prefer this name. We also omit in this notation the dependence on $r$ and $L$.}

Analogously, for the local unstable manifold of a saddle-repeller basic set $\La^-$, we consider the unstable foliation $\cW^\uu_{\La^-}$ of the $f$-invariant surface $W^\u_{r,L}(\La^-)$.
\end{remark}

Recall that a saddle-attractor basic set $\La$ is \emph{two-dominated} if 
\[
	\lVert Df|_{E^\cu}\rVert^2< m( Df|_{E^\uu}).
\]
Analogously for a saddle-repeller basic set.

\begin{remark}[SRB measures and foliations]\label{remSRB}
Let $\La^+$ be a two-dominated saddle-attractor basic set of a $C^2$-diffeomorphism $f$. By \cite{HiPuSh}, the two-domination property implies that the surface $S= W^\s_{r,L}(\La^+)$ is of class $C^2$. Moreover, by \cite{BowRue:75}, $\La^+$ carries a unique ergodic probability measure $\mu^+=\mu_{\Lambda^+}$ which is absolutely continuous in the $E^\cu$-direction. In our normally hyperbolic setting, in analogy to the case of attractors, we call the measure $\mu^+$ \emph{saddle Sinai-Ruelle-Bowen} (or, shortly, \emph{saddle-SRB}) \emph{measure}. Note that for the existence of such a measure, the $C^2$ regularity is required. 

Moreover, for every curve  $\sigma\colon [0,1]\to \Si$ which is transverse to the stable foliation $\cW^\ss_{\La^+}$, one has
\begin{equation}\label{eqSRB}
	\Leb\big\{t\in[0,1]\colon  
		\sigma(t)\in W^\ss_{r,L}(x)\text{ for some }\mu^+\mbox{-generic point }x\big\}	
	=1.
\end{equation}

In the case of a saddle-repeller $\La^-$, analogously, there is a probability measure $\mu^-$ that is a saddle-SRB measure for $f^{-1}$. Note that the measures $\mu^\pm$ are uniformly hyperbolic and $\chi^\c (\mu^-)<0 < \chi^\c(\mu^+)$. Thus, these measures have different $\s$-index.
\end{remark}

\begin{remark}[Continuations of saddle-SRB measures]\label{r.contsrb}
Every basic set of a diffeomorphism has a continuation for $C^1$-nearby diffeomorphisms. Moreover, the properties of being normally hyperbolic, of having a partially hyperbolic splitting, and of this splitting being two-dominated are $C^1$-robust. Hence, given a diffeomorphism $f$ with a two-dominated saddle-attractor/repeller basic set $\La^\pm$, any $C^1$ perturbation $g$ of $f$ gives rise to a $C^1$ diffeomorphism with a two-dominated saddle-attractor/repeller basic set $\La_g^\pm$. 

Therefore, if $f$ is $C^2$ and has a saddle-attractor $\La^+$ and a saddle-repeller $\La^-$ as above,
for every $g$ close to $f$,  the continuations $\Lambda^\pm_g$ of $\Lambda^\pm$ are well defined. Moreover, if $g$ is $C^2$, then by Remark \ref{remSRB}, there are saddle-SRB measures $\mu^\pm_g=\mu_{\Lambda^\pm_g}$. These measures depend continuously on $g$ and we can refer to $\mu^\pm_g$ as the continuations of $\mu^\pm$ (see, for example, the arguments in \cite[5.4 Proposition]{BowRue:75} stated in the case of flows).
\end{remark}

\subsection{Proof of Theorem~\ref{t.SRBrobust}: Robust cycles between saddle-SRB measures}\label{secCycles}

The following result is the main step towards the proof of Theorem \ref{t.SRBrobust}. 

\begin{theorem}\label{tp.SRBrobust}
Every compact three-dimensional manifold supports a $C^2$ diffeomorphism $f$ having a Plykin saddle-attractor $\La^+$ and a Plykin saddle-repeller $\La^-$, which are both two-dominated, and a partially hyperbolic set $\Lambda\supset \La^+ \cup \La^-$ such that for $r>0$ small and $L\in\bN$ sufficiently large
\begin{itemize}[leftmargin=0.8cm ]
\item[(a)]  $\Lambda\cap \left(W^\s_{r,L}(\La^+)\pitchfork W^\u_{r,L}(\La^-)\right)$ contains a curve 
$\sigma$ that is transverse in $W^\s_{r,L}(\La^+)$ to the foliation $\cW^\ss_{\La^+}$ and transverse in $W^\u_{r,L}(\La^-)$ to the foliation $\cW^\uu_{\La^-}$, respectively.
\item[(b)] For  every pair of points $x\in \Lambda^+$ and  $y\in \Lambda^-$ both transverse intersections
\[
W^\uu_{r,L}(x) \pitchfork W^\s_{r,L}(y) \quad \mbox{and} \quad
W^\u_{r,L}(x) \pitchfork W^\ss_{r,L}(y)
\]
contain points in the set $\Lambda$.
\end{itemize}
\end{theorem}

Let us quickly comment on this theorem and our constructions. Theorem~\ref{tp.SRBrobust} exclusively deals with Plykin saddle-attractors and -repellers, not with the saddle-SRB measures themselves (as they are uniquely determined by the dynamics). The key fact is that the topological intersection properties in item (b) 
are translated into intersections of invariant sets of the generic points of the measures.

The proof of Theorem~\ref{tp.SRBrobust} has two steps. First, we construct the Plykin saddle-sets and arrange them in a cyclic configuration (see Section~\ref{ss.Plykinsaddle}). The second, more intricate part, is to ensure that we can select intersection points within a partially hyperbolic set (i.e., the construction of the set $\Lambda$), detailed in Section~\ref{ss.proofofpSRBrobust}. The intricate and folded structure of the Plykin saddle-sets requires careful selection of points in $\Lambda$ to achieve partial hyperbolicity.

We now give the proof Theorem~\ref{t.SRBrobust} using Theorem~\ref{tp.SRBrobust}.

\begin{figure}[h] 
 \begin{overpic}[scale=.6]{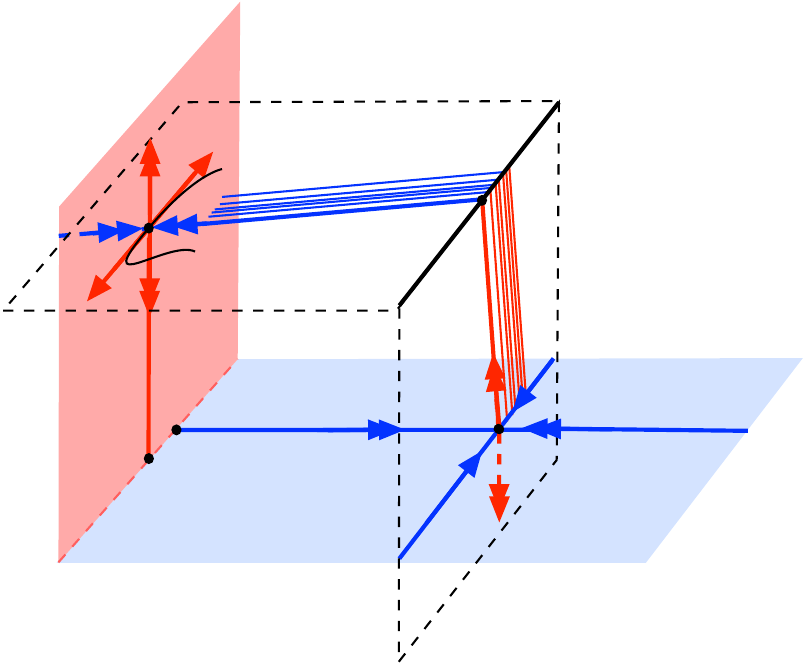}
	\put(53,53){$\sigma$}
\end{overpic}
\caption{Heterodimensional cycle between saddle-SRB measures}
\end{figure}

\begin{proof}[Proof of Theorem~\ref{t.SRBrobust}] 
Let us first see that the measures $\mu^+$ and $\mu^-$ verify conditions (1)--(4) in Definition~\ref{defrichcycmeas} of a rich cycle. Since we are dealing with saddle-attractor and saddle-repeller basic sets $\La^\pm$, we have that $\supp (\mu^\pm)$ are uniformly hyperbolic sets. Moreover, $\chi^\c (\mu^-) < 0< \chi^\c (\mu^+)$ holds true. Hence, condition (1) follows. The partially hyperbolic set $\Lambda$ satisfies conditions (2). Hypothesis (b) is just condition (3). By Remark \ref{remSRB} together with hypothesis (a), the curve $\sigma$ contains a full Lebesgue measure set of points which are in the strong stable manifold of some $\mu^+$-generic point and simultaneously a full Lebesgue measure set of points which are in the strong unstable manifold of some $\mu^-$-generic point. This implies condition (4). 

We now state the robustness of the cycle. By Remark~\ref{r.contsrb}, every $C^1$ diffeomorphism $g$ sufficiently close to $f$ gives rise to a $C^1$ diffeomorphism with a saddle-attractor basic set that is two-dominated. Therefore, for every $g$ sufficiently $C^2$-close to $f$,  the continuations $\Lambda^\pm_g$ of $\Lambda^\pm$ are two-dominated saddle-attractors/repellers and hence the saddle-SRB measures $\mu^+_g$ are well defined.

We consider a small neighborhood of $\Lambda$ and consider, for $g$ nearby $f$, the maximal invariant set $\Lambda'_g$ of $g$ in that neighborhood. Note that it is also partially hyperbolic. By transversality and continuous dependence of the invariant manifolds, there is a continuation $\sigma_g$ of $\sigma$ satisfying items (a) and (b) in Theorem~\ref{tp.SRBrobust} relative to $g$ and $\Lambda'_g$. If follows from \eqref{eqSRB} that Lebesgue almost every point in $\sigma_g$ belongs simultaneously to the stable manifold of a $\mu^+_g$-generic point $x^+\in\La^+_g$ and to the unstable manifold of a $\mu^-_g$-generic point  $x^-\in\La^-_g$. 
This ends the proof of the theorem.
\end{proof}

\subsection{Cyclically related Plykin saddle-attractors and -repellers}\label{ss.Plykinsaddle}

In Section~\ref{sss.Plya} we recall the construction of Plykin attractors of surface diffeomorphisms. In Section~\ref{secPlyb} we introduce a ``skeleton map'' that will be modified by introducing two plugs, one with a Plykin saddle-attractor and another one with a Plykin saddle-repeller, see Section~\ref{secPlyc}. These plugs are inserted in Section~\ref{sss.Plyd}. 

\subsubsection{Plykin attractors}\label{sss.Plya}

Recall that a \emph{Plykin attractor} $\Theta$ is a hyperbolic attractor contained in a two-dimensional disk such that the unstable manifold of any point in $\Theta$ is nontrivial and contained in $\Theta$. Moreover, $\Theta$ attracts every point in the disk, except three points. See \cite{Rob:99} for details. Note that this attractor is a one-di\-mensional lamination (formed by unstable manifolds) whose leaves are tangent to the unstable bundle. Moreover, each leaf of is dense in $\Theta$. Let us now provide some further properties of this attractor. 

In what follows, denote by $\Delta_\delta$ a two-dimensional disk of radius $\delta$. Fix a local $C^2$ diffeomorphism $\phi_0$ defined on $\Delta_{2\delta}$ having a Plykin attractor $\Theta\subset \Delta_{\delta}$ with hyperbolic splitting $F^\s\oplus F^\u$. The map $\phi_0$ has three repelling fixed points $r_0,s_0,r_2\in \Delta_\delta$ such that $W^\s(\Theta)\supset \Delta_{2\delta}\setminus\{r_0,s_0,r_2\}$. Without loss of generality, we can assume that $r_0=(0,0)\in\bR^2$. Let 
\begin{equation}\label{eqdeflambda}
	\lambda
	\eqdef 2 \max_{x\in \Theta}\,\lVert D\phi_0|_{F^\u(x)}\rVert >1.
\end{equation}
 Also assume that $\phi_0(\Delta_{2\delta})\subset\Delta_\delta$. 

As said already, $\Theta$ is a one-di\-mensional lamination whose leaves are each dense in $\Theta$ and tangent to $F^\u$. By Remark~\ref{r.stablefoliation}, there is defined the stable foliation. It will give rise to the strong stable foliation $\cW^\ss$ in Section \ref{secsaddleSRB} in the ``product model'' we will introduce below (see Section \ref{secPlyc}).

We extend this local diffeomorphism $\phi_0$ to the whole $\bR^2$. We continue to denote this extension by $\phi_0$. To anticipate the next steps, we also assume that $\phi_0$ coincides with the homothety of ratio $1/2$ outside $\De_{2\delta}$.

\subsubsection{The ``skeleton map''}\label{secPlyb}

We consider a diffeomorphism $f_0$ defined on $\bR^3$ that satisfies the conditions below (compare Figure \ref{figbefore}). Subsequently, surgeries will be performed on this map. The conditions are as follows:
\begin{itemize}[leftmargin=0.8cm ]
\item[(i)] $f_0$ has hyperbolic saddle fixed points $p_0,p_1, q_0,q_1$  so that 
\[
	\dim W^\u(p_0)=1=\dim W^\u(p_1), \quad 
	\dim W^\s(q_0)=1=\dim W^\s(q_1),
\]	
\item[(ii)] $W^\u(p_0)\pitchfork W^\s(p_1)\neq\emptyset$ and $W^\s(q_0)\pitchfork W^\u(q_1)\neq\emptyset$,
\item[(iii)] $W^\u(p_0)$ intersects $W^\s(q_1)$ quasi-transversely
 at a point $r_0$ and $W^\s(q_0)$ intersects $W^\u(p_1)$ quasi-transversely at a point $s_0$,%
\footnote{That is, $T_{r_0}W^\u(p_0)\oplus T_{r_0}W^\s(q_1)$ and $T_{s_0}W^\u(p_1)\oplus T_{s_0}W^\s(q_0)$ both are two-dimensional.}
\item[(iv)] the unstable eigenvalue $\lambda^\u_{p_0}$ of $Df_0(p_0)$ is larger than $\lambda^2$, the stable eigenvalue of $Df_0(q_0)$ is $\lambda^\s_{q_0}= (\lambda^\u_{p_0})^{-1}$, where $\lambda$ is as in \eqref{eqdeflambda},
\item[(v)] there is a neighborhood of the saddle $p_0$ where $f_0$ is smoothly conjugate to a linear saddle map having a representation by a diagonal matrix with two eigenvalues equal to $\frac{1}{2}$. Analogously, for $q_0$ with a diagonal matrix with two eigenvalues equal to $2$,
 \item[(vi)] The stable (resp. unstable)  eigenvalues of $p_1$ (resp. $q_1$)  are non-real numbers.
\item[(vii)] $W^\s(p_0)\pitchfork W^\u(q_0)$ contains a curve $\sigma$.
\end{itemize}

\begin{figure}[h] 
 \begin{overpic}[scale=.4]{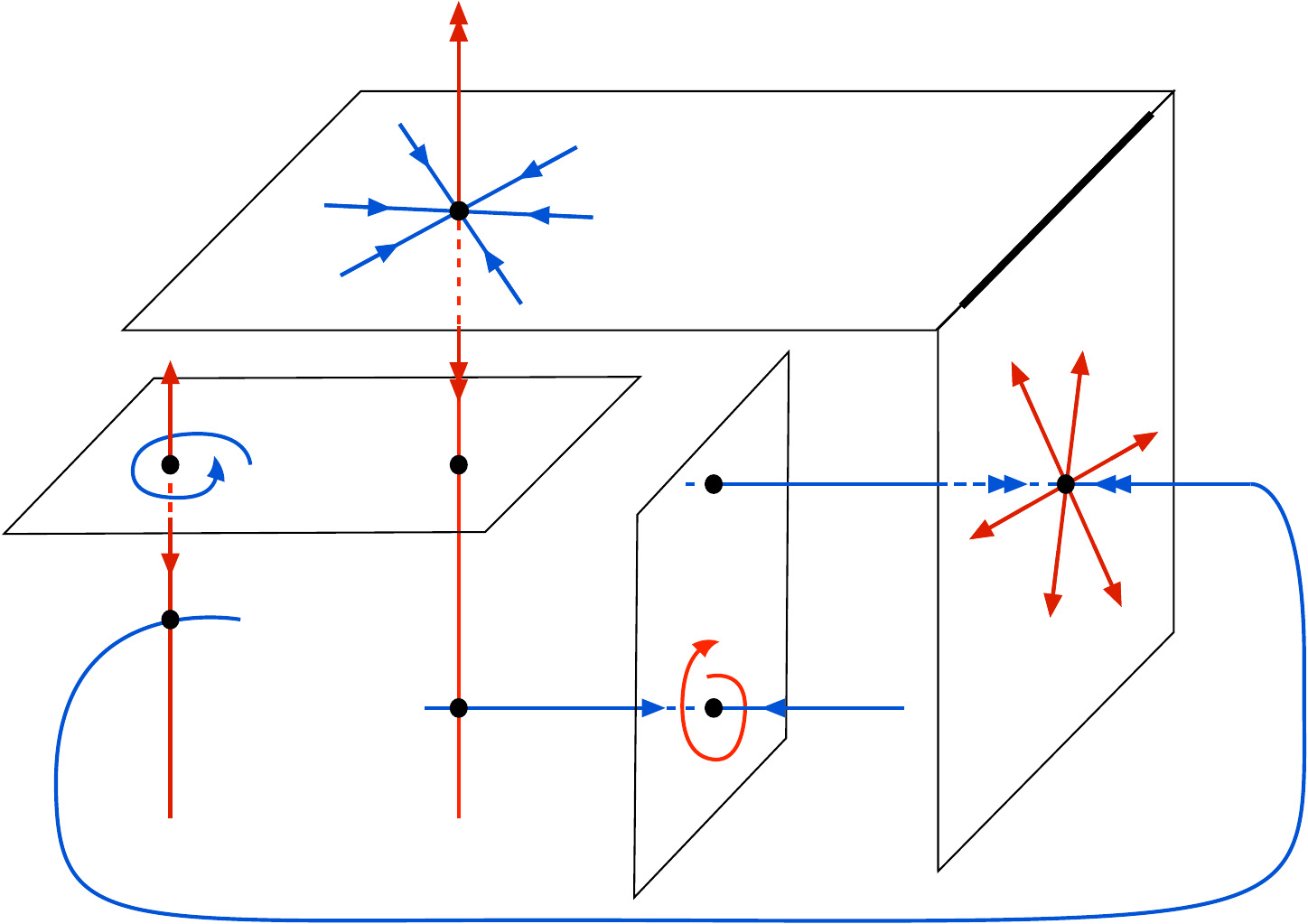}
 	\put(9,20){\small{${s_0}$}}
 	\put(31,14){\small{${r_0}$}}
 	\put(32,52){\small{$p_0$}}
 	\put(10,33){\small{$p_1$}}
 	\put(54,14){\small{$q_1$}}
 	\put(82,31.5){\small{$q_0$}}
 	\put(81,52){\small{$\sigma$}}
	\put(40,50.5){\small{$\Delta$}}
 \end{overpic}
  \caption{Conditions (i)--(vii)}
  \label{figbefore}
\end{figure}

\subsubsection{The plugs}\label{secPlyc}

We now construct two local plugs, one around the one-dimensional local unstable manifold of the point $p_0$. The construction of the one around the one-dimensional local stable manifold of $q_0$ is analogous (see Remark \ref{remotherplug}).

We first introduce some notation. Consider the linearization of $f_0$ restricted to $W^\s(p_0)$,
\[
	D^\s f_0(p_0)\eqdef Df_0(p_0)|_{E^\s_{p_0}}, \quad
	D^\s f_0(p_0)\colon E^\s_{p_0}\simeq \RR^2\to  E^\s_{p_0}\simeq \RR^2.
\]	 
By assumption (v), $D^\s f_0(p_0)$ is a homothety of ratio $1/2$. We consider a smooth conjugacy of $f_0|_{W^\s(p_0)}$ with $D^\s f_0(p_0)$. Let $\sigma$ be the curve in assumption (vii) and denote by $\sigma_{p_0}$ its image under this conjugacy. Assume $\delta>0$ is sufficiently small such that
\[
	\sigma_{p_0}\cap \De_{\delta} =\emptyset.
\]

Recall the choice of $\phi_0$ in Section \ref{sss.Plya}. Consider a smooth isotopy 
$\{\phi_t\colon\bR^2\to\bR^2\colon t\in[0,1]\}$ from $\phi_0$ to the linear map $\phi_1=D^\s f_0(p_0)$. 
We assume that $\phi_t=D^\s f_0(p_0)$ outside $\Delta_{2\delta}$, that $\phi_t=\phi_0$ for every
$t>0$ small, and that $\phi_t =D^\s f(p_0)$ for every $t$ close to $1$.

Recall $\lambda^\u_{p_0}$ in assumption (iv). For $T>0$, consider the diffeomorphisms 
\begin{equation}\label{PhiT}
	\Phi_T\colon \RR^3\to \RR^3,\quad
	\Phi_T(x,y,z)
	\eqdef \begin{cases}
		(\phi_{z/T}(x,y), \lambda^\u_{p_0} z)&\text{ if }|z|\leq T,\\
		(\frac12 (x,y), \lambda^\u_{p_0} z)&\text{ if }|z|\geq T,
	\end{cases}
\end{equation}
compare Figure \ref{Fig.localplug}.

Given $\varepsilon >0$ small, consider the cone field $\cC^\uu_\varepsilon \eqdef \{C^\uu_\varepsilon(p)\}_{p\in\bR^3}$ defined on $\bR^3$ by
\[
	C^{\uu}_\varepsilon(p)
	\eqdef \big\{(u,v)\in T_p \RR^3=\RR^2\times \RR\colon \|u\|\leq \varepsilon |v| \big\}.
\]
We say that the map $\Phi_T$ \emph{preserves the cone field $\cC^\uu_\varepsilon$} if 
$$
D\Phi_T(p)(C^{\uu}_\varepsilon(p))\subset C^{\uu}_\varepsilon(\Phi_T(p))
\quad \mbox{for every $p\in\bR^3$.}
$$

The following lemma is one important step in the proof of Theorem \ref{tp.SRBrobust}. It guarantees that the surgery described in Section \ref{sss.Plyd} does not destroy the transverse normally hyperbolic structure at the saddle fixed points $p_0$ and $q_0$ that are being replaced by the Plykin saddle-attractors and -repellers, respectively.

\begin{lemma}\label{lemconefield}
	For every $\varepsilon>0$ sufficiently small, there exists $T_0\in\bN$ such that for every $T\ge T_0$ the map $\Phi_T$ preserves the cone field $\cC^\uu_\varepsilon$. 
\end{lemma}

\begin{proof}
Let
\begin{equation}\label{defC}
	C
	\eqdef\max_{t\in[0,1]} \Big\lVert\frac{\partial}{\partial t} \phi_{t}\Big\rVert.
\end{equation}
Let $p=(x,y,z)\in\bR^3$ and $(u,v)\in T_p\bR^3=\bR^2\times\bR$. Consider first the case $|z|\le T$. Check that
\[
	D_p\Phi_T(u,v)
	= \Big(D\phi_{z/T}(u)+\frac{1}{T}\frac{\partial}{\partial t} \phi_{t}|_{t=z/T}(x,y,z)\cdot u, 
		\lambda^\u_{p_0} v\Big)
	\eqdef (u',v').
\]
Using \eqref{eqdeflambda} and $\lVert u\rVert\le\varepsilon|v|$, we have that
\[
	\lVert u'\rVert
	= \big\lVert D\phi_{z/T}(u)+\frac{1}{T}\frac{\partial}{\partial t} \phi_{t}|_{t=z/T}(x,y,z)u\big\rVert
	\le \lVert \lambda u \rVert+ \frac 1T C\lVert u\rVert
	\leq \lambda\varepsilon |v| +\frac 1T C\varepsilon |v|.
\]
From assumption (iv), $\lambda^\u_{p_0}>\lambda^2$, whenever $T$ was large enough, one gets 
\[
	\lVert u'\rVert
	\le \left(\lambda+\frac1TC\right)\varepsilon|v|
	= \left(\lambda+\frac1TC\right)\varepsilon\frac{1}{\lambda^\u_{p_0}}|v'|
	< \left(\frac1\lambda+\frac{1}{\lambda^2T}C\right)\varepsilon|v'|
	< \varepsilon|v'|,
\]
that is, $(u',v')\in C^{\uu}_{\varepsilon}(\Phi_T(p))$. 

In the case $|z|\ge T$, we have $(u',v')=(\frac12u,\lambda^\u_{p_0}v)$ and hence 
\[
	\lVert u'\rVert
	= \frac12\lVert u\rVert
	\le \frac12\varepsilon|v|
	= \frac12(\lambda^\u_{p_0})^{-1}\varepsilon|v'|
	< \varepsilon|v'|,
\]	
getting $(u',v')\in C^{\uu}_{\varepsilon}(\Phi_T(p))$. This proves the lemma.
\end{proof}

In what follows, we fix large $T>0$ sufficiently large and consider the solid cylinder 
\[
	\Gamma_{p_0}\eqdef\De_{2\delta} \times [-\lambda^\u_{p_0} T, \lambda^\u_{p_0} T].
\]	
The plug around the local unstable manifold of the point $p_0$ that we consider below is the restriction of $\Phi_T$ to $\Gamma_{p_0}$ and denoted by $(\Gamma_{p_0}, \Phi_T)$ (see Figure \ref{Fig.localplug}).  Note that, in a neighborhood of the boundary of $\Gamma_{p_0}$, the map $\Phi_T$ coincides with the linear map 
$$
	Df(p_0)= \left(\begin{array}{ccc}
                  \frac12&0&0\\
                  0&\frac12&0\\
                  0&0&\lambda^\u_{p_0}
                 \end{array}\right)= 
                 \left(\begin{array}{cc}
                   D^\s f(p_0)&0\\
                0&\lambda^\u_{p_0}
                 \end{array}\right).
                 $$

\begin{figure}[h] 
 \begin{overpic}[scale=.4]{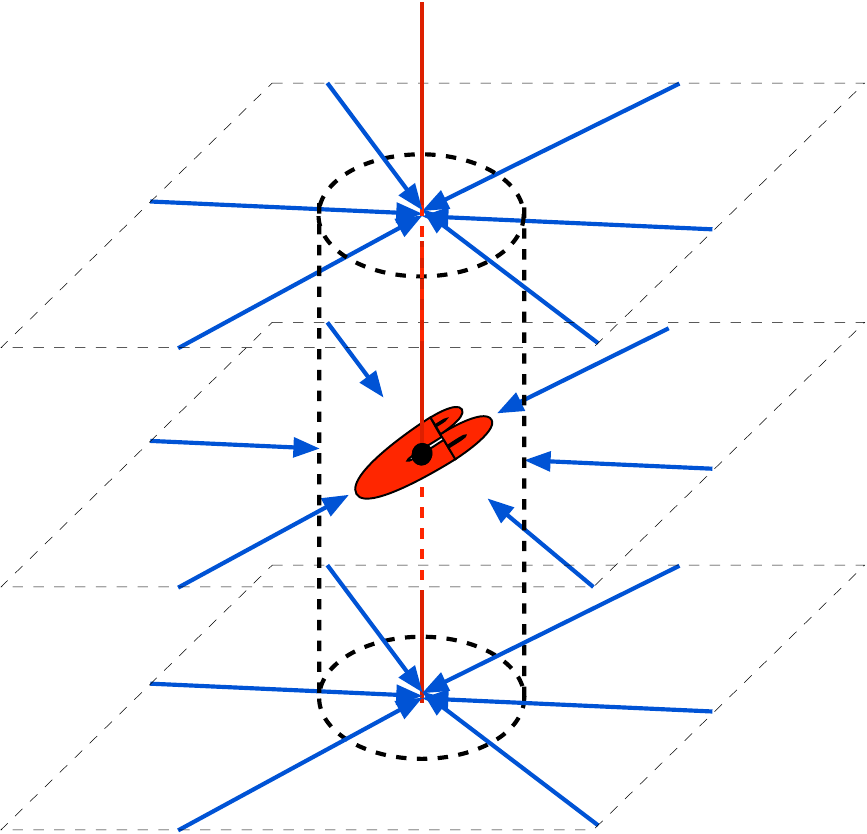}
 	\put(97,23){\small{$\phi_1(x,y)$ at $z=T$}}
 	\put(97,53){\small{$\phi_0(x,y)$ at $z=0$}}
 	\put(97,81){\small{$\phi_1(x,y)$ at $z=-T$}}
 \end{overpic}
  \caption{The local plug $(\Gamma_{p_0},\Phi_T)$}
  \label{Fig.localplug}
\end{figure}

\begin{remark}[The  plug $(\Gamma_{q_0}, \Phi_{T})$]\label{remotherplug}
To define the second plug (around the one-dimensional local stable manifold of $q_0$), consider a smooth conjugacy of the restriction of $f_0$ to $W^\u(q_0)$ with its linear part $Df_0(q_0)$ (the latter is the homothety of ratio $2$). We denote by $\sigma_{q_0}$ the image of the segment $\sigma$ by this conjugacy and assume that $\delta>0$ is sufficiently small such that $\sigma_{q_0}\cap \De_{\delta} =\emptyset$. The  plug $(\Gamma_{q_0}, \Phi_{T})$ is obtained by considering the linearization of $f_0^{-1}$ restricted to $W^\u(q_0)$. For the final step, we concatenate a switch of coordinates $(y,z,x)\mapsto(x,y,z)$ with the plug $(\Phi_T(\Gamma_{p_0}), \Phi_T^{-1})$. 
\end{remark}

\subsubsection{The surgery:  inserting the plugs}\label{sss.Plyd}

We now define the map $f$. It is obtained from $f_0$, performing two surgeries to insert the local plugs  in Section \ref{secPlyc}. This construction is local and done using the linearizing coordinates of $f_0$ around $p_0$ and $q_0$, respectively.

We denote by $H_\varepsilon$ the homotethy of ratio $\varepsilon$ and consider the solid cylinders 
\[
	\Ga_{\varepsilon,p_0}\eqdef H_\varepsilon(\Gamma_{p_0})
	\quad\text{ and }\quad
	\Ga_{\varepsilon,q_0}\eqdef H_\varepsilon(\Gamma_{q_0}).
\]	 
Note that, for $\varepsilon $ small enough, the sets $f_0^{-1} (\Ga_{\varepsilon,p_0}) \cup \Ga_{\varepsilon,p_0} \cup  f_0(\Ga_{\varepsilon,p_0})$ and 
$f_0^{-1} (\Ga_{\varepsilon,q_0}) \cup \Ga_{\varepsilon,q_0} \cup  f_0(\Ga_{\varepsilon,q_0})$ are disjoint. We now perform the following surgery: 
\begin{itemize}[leftmargin=0.4cm ]
\item replace $f_0$ in $\Ga_{\varepsilon,p_0}$  by the conjugate of $\Phi_T$  by $H_\varepsilon$,
\item replace $f_0$ in $\Ga_{\varepsilon,q_0}$ by the conjugate of $\Phi_T^{-1}$ by $H_\varepsilon$,
\item outside these sets the map $f_0$ is not modified. 
\end{itemize}
This surgery provides a $C^2$ diffeomorphism $f$. To see why this is so, note that the conjugate of $\Phi_T$ coincides with $f_0$ in the neighborhood of the boundary of $\Ga_{\varepsilon,p_0}$ and the conjugate of $\Phi_T^{-1}$ coincides with $f_0$ in the neighborhood of the boundary of $\Ga_{\varepsilon,q_0}$.

\begin{remark}[The map $f$ has saddle-SRB measures]\label{r.valetudo}
The above construction implies that $f$ has a Plykin saddle-attractor $\La^+$ that is two-dominated and whose local stable manifold coincides with the local stable manifold of $p_0$ for $f_0$ (except at three points in $\Ga_{\varepsilon,p_0}$). Analogously, $f$ also has a Plykin saddle-repeller $\La^-$ that is two-dominated and whose local unstable manifold coincides with the local unstable manifold of $q_0$ for $f_0$ (except at three points in $\Ga_{\varepsilon,q_0}$). 

The set $\La^+$ has a partially hyperbolic splitting $T_{\La^+} M=E^\ss\oplus E^\cu\oplus E^\uu$,  where $E^\ss$ is its stable bundle and $E^\cu\oplus E^\uu$ is its unstable bundle. Similarly, 
the set $\La^-$ has a partially hyperbolic splitting $T_{\La^-} M=E^\ss\oplus E^\cs\oplus E^\uu$,  where $E^\ss \oplus E^\cs$ is its stable bundle and $E^\uu$ is its unstable one.
We consider the splitting
\begin{equation}\label{e.joinsplittion}
	T_{\La^+ \cup \La^-} M
	\eqdef E^\ss\oplus E^\c \oplus E^\uu, \quad \text{ where }\quad 
	E^\c(x)\eqdef  
	\begin{cases}
		E^\cs(x)& \mbox{if $x\in\La^-$},\\
 		E^\cu(x)& \mbox{if $x\in\La^+$}. 
	\end{cases}
\end{equation}
We will refer to $E^\c$ as the \emph{central bundle}.

As a consequence of the properties of the Plykin attractor $\Theta$ and normal hyperbolicity, the
set $\La^+$ (respectively, $\La^-$) is a one-di\-mensional lamination whose leaves are each dense in $\La^+$
(resp. $\La^-$) and tangent to the central bundle $E^\c$. Because of the two-domination, the sets $\Lambda^\pm$ support saddle-SRB measures $\mu^\pm$. Compare Remark \ref{remSRB}. 
\end{remark}

\begin{remark}[Continuations of the saddle-SRB measures]
The sets $\Lambda^\pm$ and the measures $\mu^\pm$ have well-defined continuations (to Plykin saddle-sets and saddle SRB measures) when considering diffeomorphisms which are nearby $f$, recall Remark~\ref{r.contsrb}.
\end{remark}

\begin{remark}[Invariant manifolds of the Plykin saddle-sets $\La^\pm$]\label{interesectionPlya}
Recalling the definitions and the choices of the curves $\sigma_{p_0}$ and $\sigma_{q_0}$, the two-dimensional $C^2$ manifolds  $W^\s_{r,L} (\La^+)$ and $W^\u_{r,L} (\La^-)$ (both relative to $f$) intersect transversely along the curve $\sigma$. Moreover, this curve is differentiable. After an arbitrarily small perturbation and by  shrinking $\sigma$, if necessary, one may assume that $\sigma$  is transverse in $W^\s_{r,L}(\La^+)$ to the strong stable foliation $\cW^\ss_{\Lambda^+}$ and transverse in $W^\u_{r,L}(\La^-)$ to the strong unstable foliation $\cW^\uu_{\Lambda^-}$, respectively. Recall Remark \ref{r.stablefoliation}.
\end{remark}

\begin{remark}[Transverse intersections]\label{interesectionPlyb}
The points $p_1$ and $q_1$ continue to be fixed points of $f$ with the same stable and unstable directions as for $f_0$. The construction implies that $W^\s(\Lambda^-)$ intersects $W^\uu(p_1)$. Possibly after a perturbation, we can guarantee that this intersection contains transverse points. As a consequence of the $\lambda$-lemma (also called inclination lemma), we have  $W^\s(y)\pitchfork W^\uu(p_1)\ne\emptyset$  for every $y\in \La^-$. Similarly, $W^\u(x)\pitchfork W^\ss(q_1)\ne\emptyset$ for every $x\in \La^+$. 
\end{remark}

\subsection{End of the proof: construction of the partially hyperbolic set $\Lambda$}\label{ss.proofofpSRBrobust}

Once a pair of cyclically related Plykin saddle-sets is constructed, the major difficulty is to ensure that the corresponding intersections occur within a partially hyperbolic set satisfying conditions (a) and (b) in Theorem \ref{tp.SRBrobust}. Moreover, this partially hyperbolic splitting needs to extend the one in \eqref{e.joinsplittion}. The main result of this section is Proposition~\ref{pl.fsatisfies}, it answers positively this question and therefore implies Theorem \ref{tp.SRBrobust}. 

\begin{proposition}\label{pl.fsatisfies}
Consider the diffeomorphism $f$ constructed above. Then there is a partially hyperbolic set $\Lambda$ containing  $\Lambda^\pm$ satisfying conditions (a) and (b) of  Theorem~\ref{tp.SRBrobust}.
\end{proposition}

Note first that the cycle of the sets $\La^\pm$ involves two types of intersections: 
\begin{itemize}[leftmargin=0.8cm ]
\item[(qt)] the quasi-transverse intersection between the one-dimensional (strong) stable manifolds of points in $\Lambda^+$ and the one-dimensional (strong) unstable manifold of points in $\Lambda^-$,
\item[(t)] the transverse intersection between the two-dimensional (weak) unstable manifold of points in $\Lambda^+$ and the two-dimensional (weak) stable manifold of points in $\Lambda^-$. 
\end{itemize}
We need to ``put" these intersection points within a partially hyperbolic set $\Lambda$. Dealing with the quasi-transverse intersection points is simpler. However, due to the ``folding" nature of the two-dimensional (weak) unstable and (weak) stable manifolds tangencies may occur, making the handling of the latter type of points more intricate.

\begin{proof}[Proof of Proposition \ref{pl.fsatisfies}]
Consider the splitting $T_{\La^-\cup\La^+} M=E^\ss\oplus E^\c\oplus E^\uu$ in \eqref{e.joinsplittion}. We need to extend this splitting to a greater set $\Lambda$ to be defined below. Besides $\La^-\cup\La^+$, this set will also contain quasi-transverse and transverse intersections of the invariant sets of $\Lambda^\pm$. More precisely, 
\[
	\Lambda
	\eqdef \Lambda_{\rm qt}\cup\Lambda_{\rm t},
\]
where $\Lambda_{\rm qt}$ and $\Lambda_{\rm t}$ are defined in \eqref{eqdefLambdaqt} and \eqref{e.Kt}, respectively. Both sets contain $\Lambda^-\cup\Lambda^+$. Moreover, $\Lambda_{\rm qt}$ contains the intersections in item (a) in Theorem~\ref{tp.SRBrobust} and $\Lambda_{\rm t}$ contains the intersections in item (b) in Theorem~\ref{tp.SRBrobust}. We will see that the  splitting  in \eqref{e.joinsplittion} extends coherently to $\Lambda$. 

We begin by defining $\Lambda_{\mathrm qt}$,
\begin{equation}\label{eqdefLambdaqt}
	\Lambda_{\rm qt} \eqdef \Lambda^+ \cup \Lambda^- \cup \bigcup_{i\in \bZ} f^{i} (\sigma) .
\end{equation}
This set is invariant, by construction. Moreover, as $f^i(\sigma)$ accumulates on $\Lambda^+$ and $\Lambda^-$ as $i\to-\infty$ and $i\to\infty$, respectively, it follows that $\Lambda_{\mathrm qt}$ is compact.  

\begin{lemma}\label{lcl.K0}
The set $\Lambda_{\mathrm qt}$ is partially hyperbolic with a splitting which extends the one of $\La^+ \cup \La^-$.
\end{lemma}

\begin{proof}
By Remark~\ref{interesectionPlya}, the segment $\sigma$ is a differentiable curve that is transverse to the stable (respectively, unstable) foliation defined on $W^\s (\La^+)$ (respectively $W^\u (\La^-)$). Given $x\in \sigma$, we let  $E^\c(x)= T_x \sigma$, 
$E^\ss(x)= T_x W^\ss (x)$, and $E^\uu(x)= T_x W^\uu (x)$. By iterating this splitting, we get an invariant
 splitting defined on the closure of the orbit of $\sigma$. By transversality, this splitting ``glues" with the ones of $\La^\pm$ and therefore extends continuously to the splitting defined on $\La^-$ and $\La^+$. 
 
Finally note that the splittings on $\Lambda^+$ and $\Lambda^-$ extend to some neighborhoods $U^+$ and $U^-$, respectively. Then observe that $\Lambda_{\mathrm qt}$ contains besides the points in $\Lambda^+\cup\Lambda^-$ only ones that ``transit in finite time'' from $U^+$  to $U^-$ and remain there forever. This straightforwardly implies partial hyperbolicity of the splitting on $\Lambda_{\rm qt}$.
\end{proof}	

In what follows, we define $\Lambda_{\mathrm{t}}$. For that, we need the preliminary results in Lemmas~\ref{lcl.cicloPl}, \ref{lcl.ph}, and \ref{lcl.PH}. 
Recall again our notation in Remark \ref{r.localR}. 

\begin{lemma}\label{lcl.cicloPl}  
For $r>0$ small and $i\in\bN$ sufficiently large, for every $x\in \Lambda^+$ and  $y\in \Lambda^-$, the intersections
\[
	W^\uu_{r,i}(x)\cap W^\s_{r,i}(y)
	\quad \text{ and }\quad
	W^\u_{r,i}(x)\cap W^{\ss}_{r,i}(y)
\]	
contain some transverse points.
\end{lemma}

\begin{proof}
As the proof involves invariant manifolds relative to both the initial map $f_0$ and the map $f$ resulting from the surgery, we add this dependence in the notation for clarity. To accompany the proof, refer to Figure~\ref{figbefore} and recall conditions (i)--(vii) in Section~\ref{secPlyb}. 

As $W^\uu (p_1,f_0)= W^\u (p_1,f_0)$, condition (iii) implies that $W^\uu(p_1, f_0)$ intersects quasi-transversely $W^\ss(q_0, f_0)=W^\s(q_0, f_0)$. This implies that for every $y\in \Lambda^-$ the set $W^\uu(p_1,f)$ and the (weak) stable set $W^\s (y, f)$ have some transverse intersection. See also  Remark~\ref{interesectionPlyb}. Therefore, for every $i\in\bN$ large enough, the manifolds $W^\uu_{r,i} (p_1,f)$ and $W^\s_{r,i}(y,f)$ have some transverse intersection, for every $y\in\Lambda^-$.
 
On the other hand, recalling condition (ii) and the fact that $W^\u (p_0,f_0)=W^\uu (p_0,f_0)$, $W^\uu (p_0,f_0)$ intersects transversely  $W^\s(p_1,f_0)$. This implies that there is $i\in\bN$ large enough such that the sets $W^\uu_{r,i}(x,f)$ and $W^\s_{r,i}(p_1,f)$ have some transverse intersection, for every $x\in \Lambda^+$. By the $\lambda$-lemma, for every $x\in\Lambda^+$ the set $W^\uu(x,f)$ accumulates on $W^\uu(p_1,f)=W^\u (p_1,f)$ in the $C^1$-topology. Therefore $W^\uu(x,f)$ intersects transversely the (weak) stable manifold $W^\s_{r,i}(y,f)$, for every $y\in\La^-$. Thus, after increasing $i$ if necessary, one gets that for every $x\in \Lambda^+$ and $y\in \Lambda^-$ the manifolds $W^\uu_{r,i}(x,f)$ and $W^\s_{r,i}(y,f)$ have some transverse intersection.
  
The second item follows similarly, by interchanging the roles of $p_0,p_1$ and $q_0,q_1$. This completes the proof of the lemma.
\end{proof}

For the next lemma, recall that $W^\u(x)$ is two-dimensional for every $x\in \La^+$ and that $W^\s(y)$ is two-dimensional for every $y\in \La^-$. 

\begin{lemma}\label{lcl.ph}
For every $x\in\La^+$ there are points $y\in \La^-$ and  $z\in W^\uu(x)\pitchfork W^\s(y)$  such that the bundles 
\[
	E_z\eqdef T_z W^\ss(y), \quad  F_z\eqdef
	T_z( W^\u(x)\pitchfork W^\s(y)),  	
	\quad 
	G_z\eqdef T_z W^\uu(x)
\]	
satisfy $T_zM= E_z \oplus F_z\oplus G_z$.
\end{lemma}	

\begin{proof} 
	Fix $x\in\Lambda^+$. By construction, $W^\u(x)\pitchfork W^\s(p_1)$ contains a curve $\alpha$. Note that $\alpha$ can be chosen to intersect $W^\uu(x)$ in its interior at some point $\widehat z$. 

Given $y_0\in \La^-$, consider the strong stable lamination $W^\ss(\cdot)$ of $W^\s(y_0)$, and the family  of laminations obtained considering backward iterations
$$
\cW^\ss_i (y_0)\eqdef\{f^{-i}(W^\ss(\cdot))\}
$$ 
defined on $W^\s(f^{-i}(y_0))$.

\begin{claim}\label{cf.limitl}
Let $y_0\in \La^-$.  Then the family  of laminations $\cW^\ss_i(y_0)$ has at least two different limit laminations on $W^\s(p_1)$.
\end{claim}

This claim implies that, after shrinking the curve $\alpha$ if necessary, at least one of these limit laminations
is transverse to $\alpha$.

\begin{proof}[Proof of Claim~\ref{cf.limitl}]
Recall again that $W^\s(y_0)\pitchfork W^\uu(p_1)$ (Remark \ref{interesectionPlyb}).
By the $\lambda$-lemma, $W^\s(f^{-i}(y_0))$ accumulates at $W^\s(p_1)$. Each $W^\s(f^{-i}(y_0))$ is foliated by strong stable leaves, $\cW^\ss_i\eqdef\{f^{-i}(W^\ss(\cdot))\}$. Taking subsequences, we obtain limit laminations in $W^\s(p_1)$. As, by condition (vi) in Section \ref{secPlyb}, stable eigenvalues of $Df_0(p_1)=Df(p_1)$ are non-real, and there are at least two limit laminations. This proves the claim.
\end{proof}

We now select the points $y$ and $z$. Consider a subsequence $(i_k)_k$ such that the sequence of laminations $\cW^\ss_{i_k}$  converges to a lamination $\cW^\ss_\infty=\{W^\ss_\infty(\cdot)\}$ transverse to $\alpha$. For $k$ large enough consider a sequence of curves $(\alpha_k)_k$, 
\[
	\alpha_k\subset W^\s(f^{-i_k}(y_0))\pitchfork W^\u(x),
\]
such that $\alpha_k$ transversally intersects $W^\uu(x)$ at some point $z_k$ and that $\alpha_k\to\alpha$ as $k\to\infty$. Note that $W^\s(f^{-i_k}(y_0))$ is foliated by the strong stable manifolds of points in $\Lambda^-$, indeed by the leaves of the foliation $\cW^\ss_i$. Hence, there is $y_k\in W^\s(f^{-i_k}(y_0))\cap\Lambda^-$ such that there is a point $z_k\in W^\ss(y_k)\cap W^\uu(x)$. Up to passing to a subsequence, we can assume that $z_k\to\widehat z$. 

As $\alpha$ is transverse to $\cW^\ss_\infty$, we have that $T_{\widehat z}\alpha\oplus T_{\widehat z}W^\ss_\infty(\widehat z)=T_{\widehat z}W^\s(p_1)$. Hence, by continuity, for $k$ sufficiently large, we have that $T_{z_k}\alpha_k\oplus T_{z_k}W^\ss(y_k)$. Now fix some $k$ sufficiently large and let $z\eqdef z_k$ and $y\eqdef y_k$. Then the bundles $E_z$ and $F_z$ as in the assertion of the lemma satisfy $E_z\oplus F_z$. The lemma follows as the bundle $G_z$ as defined in the lemma is transverse to $E_z\oplus F_z$.
\end{proof}

\begin{lemma}\label{lcl.PH} 
For every pair of points $x\in\La^+$ and $y\in \La^-$, there are points 
\[
	z=z(x,y)\in W^{\uu}(x)\pitchfork W^\s(y) \quad \mbox{and} \quad 
	w=w(x,y)\in W^{\u}(x)\pitchfork W^\ss(y)
\]
such that  
\begin{itemize}
\item
$E_z=T_zW^\ss(y')$, for some $y'\in W^\s(y)$,  $F_z= T_z(W^{\u}(x)\pitchfork W^\s(y))$,
and $G_z= T_z W^{\uu}(x)$ satisfy $T_zM= E_z \oplus F_z\oplus G_z$,
\item
$E'_w =T_w W^\ss (x)$, $F'_w= T_w (W^{\u}(x)\pitchfork W^\s(y))$, and $G'_w=T_wW^\uu(x')$, for some $x'\in W^\u(x)$, satisfy $T_wM= E'_w \oplus F'_w\oplus G'_w$.
 \end{itemize}
\end{lemma} 

\begin{proof}
We only prove the lemma for the points $z$, the proof for $w$ is analogous considering $f^{-1}$ instead of $f$.

By Remark \ref{r.valetudo}, the set $\La^-$ is a one-di\-mensional lamination whose leaves are each dense in $\La^-$ and tangent to the central bundle $E^\c$. Thus, the local stable manifold of $\La^-$, which is the union of the local strong stable manifold $W^{\ss}_{r,L}(y)$, $y\in\La^-$, is a two-di\-mensional lamination whose leaves are dense in $\La^-$, tangent to a continuous plane field, and subfoliated by the strong stable leaves.
As a consequence, for any $x\in\La^+$ such that $W^{\uu}(x)\pitchfork W^\s(y_0)\ne \emptyset$ for some  $y_0\in \La^-$, taking $z\in  W^{\uu}(x)\pitchfork W^\s(y_0)$ and defining $E_z,F_z,G_z$ as in the assertion of the lemma it follows $T_zM= E_z\oplus F_z\oplus G_z$. 
\end{proof}

We are now ready to define the set $\Lambda_{\mathrm{t}}$. First observe that for every $(x,y)\in\Lambda^+\times\Lambda^-$, there is a neighborhood $\Gamma_{(x,y)}$ such that the functions $z(\cdot)$ and $w(\cdot)$ in Lemma \ref{lcl.PH} can be chosen to vary continuously on it. By compactness of $\Lambda^+\times\Lambda^-$, there are finitely many points $(x_i,y_i)$ and associated neighborhoods $\Gamma_i=\Gamma_{(x_i,y_i)}$ with $\Lambda^+\times\Lambda^-\subset\bigcup_i\Gamma_i$ such that the associated functions $z_i$ and $w_i$ vary continuously on $\Gamma_i$. We let
\begin{equation}\label{e.Kt}
	\Lambda_{\rm t} 
	\eqdef   \bigcup_i \Lambda_{\rm t,i} \cup \La^+ \cup \La^-,
	\,\,\text{ where }\,\,
	\Lambda_{\rm t,i} \eqdef
  	\bigcup_{(x,y)\in \Gamma_i}
 \bigcup_{k\in \mathbb{Z}} f^k (\{w_i(x,y)), z_i(x,y)\}).
\end{equation}
It follows from the construction that this set is invariant and compact.  

\begin{lemma}\label{lcl.K0bis}
The set $\Lambda_{\mathrm t}$ is partially hyperbolic with a splitting which extends the one of $\La^+ \cup \La^-$.
\end{lemma}

\begin{proof}
For every $i$ and $(x,y)\in\Gamma_i$, we consider the bundles in Lemma \ref{lcl.PH} at the points $w_i(x,y)$ and $z_i(x,y)$. We iterate these bundles by $Df^k$ and extend to the whole $\Lambda_{\rm t,i}$.  This provides an invariant splitting. Note that (as in the case of $\Lambda_{\mathrm{qt}}$) these bundles align with the ones in $\Lambda^+$ and $\Lambda^-$. 

Arguing as in the proof of Lemma \ref{lcl.K0}, the splittings on $\Lambda^+$ and $\Lambda^-$ extend to neighborhoods $U^+$ and $U^-$, respectively. The set $\Lambda_{\mathrm t}$ contains besides $\Lambda^+\cup\Lambda^-$ only points that ``transit in finite time'' from $U^+$  to $U^-$ and remain there and transit from $U^-$  to $U^+$ and remain there. This implies partial hyperbolicity of the splitting on $\Lambda_{\rm t}$.
\end{proof}

Note that the partially hyperbolic splittings of $\Lambda_{\rm t}$ and $\Lambda_{\rm qt}$ coincide in their intersection $\La^+\cup \La^-$. This provides a partially hyperbolic splitting in $\Lambda_{\rm qt}\cup\Lambda_{\rm t}$, concluding the proof of Proposition~\ref{pl.fsatisfies}.
\end{proof}

The proof of Theorem~\ref{tp.SRBrobust} is now also complete.
\qed

\section{Rich cycles of measures in step skew products with blenders}\label{s.robustblenders}

In this section, we consider skew products having a $\cs$-blender and a $\cu$-blender related by a robust heterodimensional cycle. We see that this dynamical configuration leads to the occurrence and abundance of rich heterodimensional cycles of measures. Here, robustness is understood within the class of step skew products. Note that the term ``rich cycle'' is a slight abuse of notation since this terminology was introduced above in the differentiable setting. However, the considered skew products admit smooth realizations (see Section \ref{secsmooth}) that will allow to use the tools developed in previous sections.

Let $\Sigma_4 \eqdef \{1,2,3,4\}^\ZZ$ and endow this space with the standard metric. Its elements are written in the form $\Sigma_4\ni\omega \eqdef (\omega_i)_{i\in \bZ}$, $\omega_i \in \{1,2,3,4\}$. We denote by $\sigma \colon \Sigma_4 \to \Sigma_4$ the standard shift map. Given diffeomorphisms $h_i\colon \bR \to \bR$, $i=1,2,3,4$, we consider their associated one-step skew product
\begin{equation}\label{sspF}
	F\colon\Sigma_4\times \RR\to\Sigma_4\times \RR, \quad
	F(\omega,t)\eqdef(\sigma(\omega),h_{\omega_0}(t)).
\end{equation}
Here, we assume that $h_1$ and $h_2$ are uniform dilations by a factor greater than $\lambda>1$ with fixed points $p^+<0$ and $q^+>1$, respectively. We also assume that $h_3$ and $h_4$ are uniform contractions by a factor smaller than $\lambda^{-1}<1$ with fixed points $p^-<0$ and $q^->1$, respectively. Moreover, we consider the following two \emph{superposition hypotheses} (compare Figure~\ref{fig.blendersuperposition}): 
\begin{eqnarray}
	&&\mbox{for every $t\in [0,1]$ it holds either }h_1(t)\in (0,1) \text{ or }h_2(t)\in (0,1),	\label{e.superpositioncu}\\
	&&\mbox{for every $t\in [0,1]$ it holds either }h_3^{-1}(t)\in (0,1) \text{ or }h_4^{-1}(t)\in (0,1).
	\label{e.superpositioncs}
\end{eqnarray}
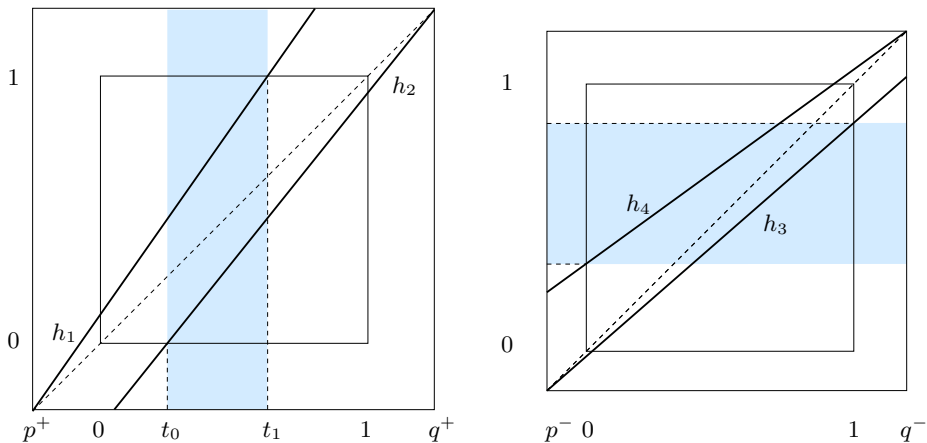
\begin{figure}[h] 
  \begin{overpic}[scale=.35]{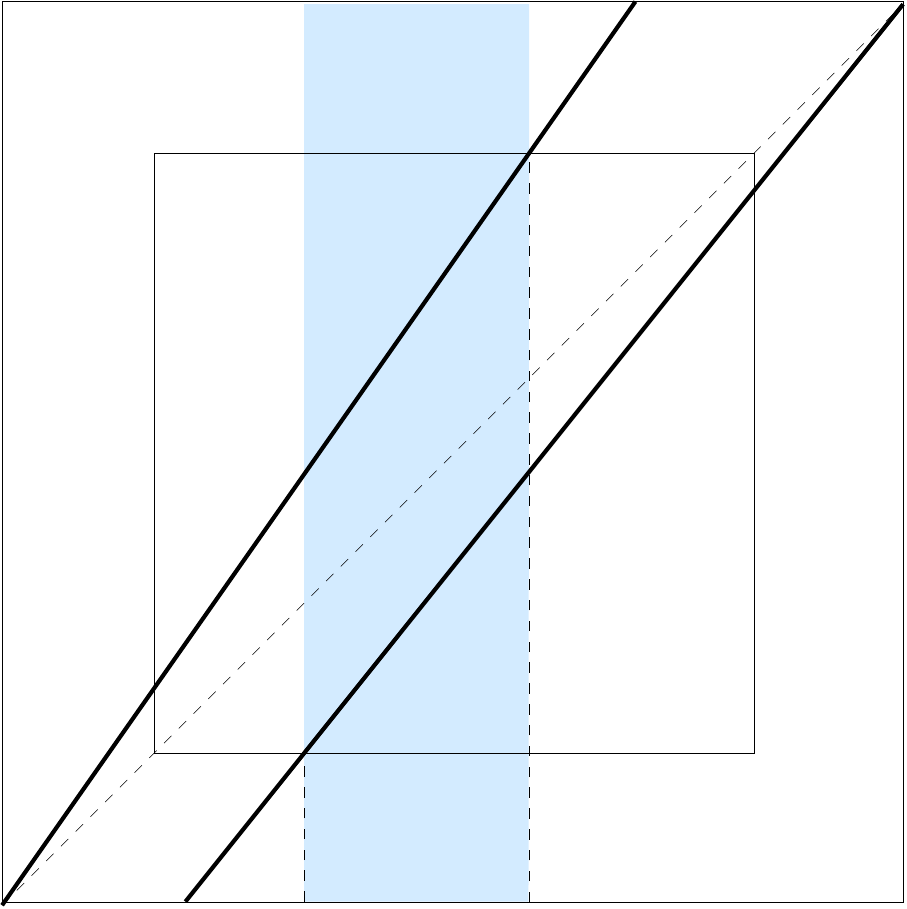}
 	\put(-2,-6){\small{$p^+$}}
 	\put(98,-6){\small{$q^+$}}
	\put(5,18){\small$h_1$}
	\put(89,79){\small$h_2$}
	\put(32,-6){\small$t_0$}
	\put(57,-6){\small$t_1$}
	\put(81,-6){\small$1$}
	\put(15,-6){\small$0$}
	\put(-6,81){\small$1$}
	\put(-6,16){\small$0$}
 \end{overpic}  
 \hspace{1cm}
 \begin{overpic}[scale=.35]{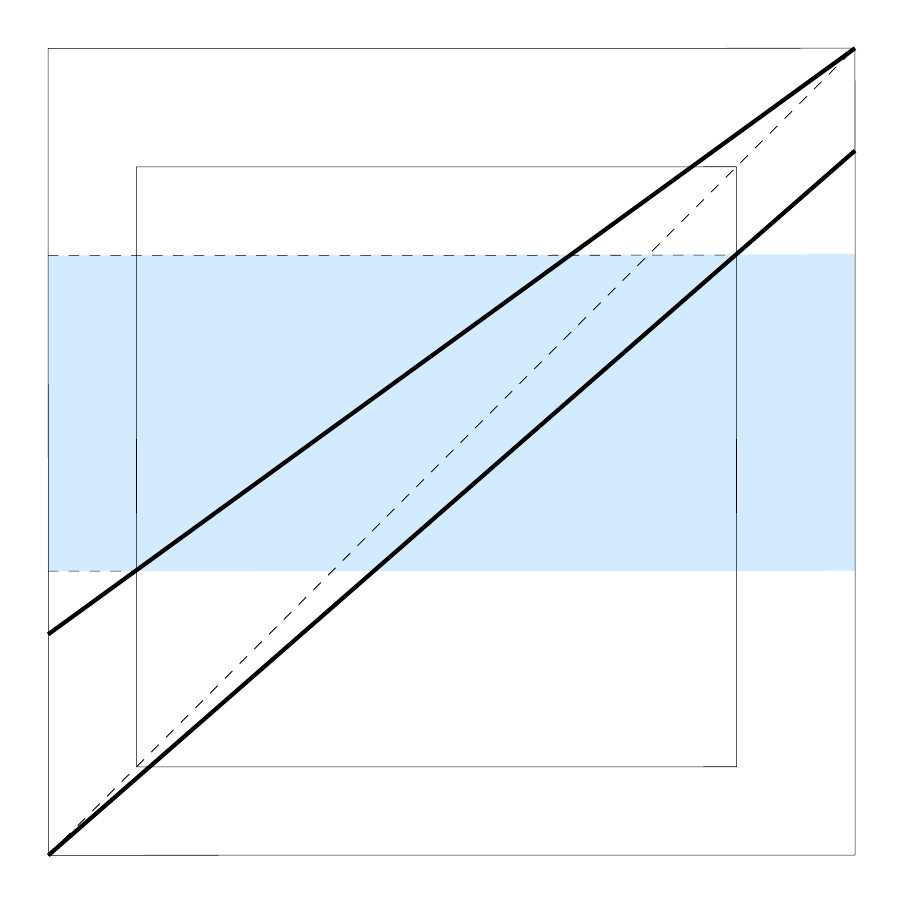}
 	\put(5,-6){\small{$p^-$}}
 	\put(93,-6){\small{$q^-$}}
	\put(25,50){\small$h_4$}
	\put(59,45){\small$h_3$}
	\put(80,-6){\small$1$}
	\put(14,-6){\small$0$}
	\put(-6,80){\small$1$}
	\put(-6,15){\small$0$}
 \end{overpic}
  \caption{Superposition regions for the maps $h_1,h_2$ and the maps $h_3,h_4$}
  \label{fig.blendersuperposition}
\end{figure}
We say that a step skew product $F$ as in \eqref{sspF} satisfying \eqref{e.superpositioncu} (respectively \eqref{e.superpositioncs}) has a \emph{$\cu$-blender} (respectively, a \emph{$\cs$-blender}).  For a justification of this terminology see Remark \ref{bigremark} item (5). When both conditions hold, we say that $F$ is a \emph{skew product with overlapping blenders} (where the ``overlap'' refers to the interval $(0,1)$). The overlapping condition will guarantee ``dynamical interactions'' between the \emph{$\cu$-blender} and the \emph{$\cs$-blender}.  

Note that hypothesis \eqref{e.superpositioncu} implies that
\begin{equation}\label{eqqconsequence}
p^+<h_2^{-1}(p^+)<h_1^{-1}(q^+)<q^+.
\end{equation}

\begin{remark}[Superposition region]\label{remsupersuper}
Hypothesis \eqref{e.superpositioncu} implies that there are $0<t_0<t_1<1$ such that
\[
	0<t_0=h_2^{-1}(0)< h_1^{-1}(1)=t_1<1.
\]
The interval $(t_0,t_1)$ is called  \emph{superposition region} and is characterized by the fact that  if $h_1(t),h_2 (t)\in(0,1)$ then $t\in (t_0,t_1)$ (compare Figure~\ref{fig.blendersuperposition}).

The analogous statement holds for the maps $h_3^{-1}$ and $h_4^{-1}$.
\end{remark}

\begin{remark}[Alternative superposition region(s)]
Instead of the interval $[t_0,t_1]=[h_2^{-1}(0),h_1^{-1}(1)]$, we could also work for instance with $[t_0',t_1']=[h_2^{-1}(p^+),h_1^{-1}(q^+)]$ as superposition region. We prefer, however, hypotheses \eqref{e.superpositioncu} and \eqref{e.superpositioncs} for notational simplicity in order to have the \emph{common} reference interval $[0,1]$ for both sets of maps $h_1,h_2$ as well as $h_3^{-1},h_4^{-1}$.\end{remark}

\begin{remark}[Robustness of blenders]\label{remRobustskew}
	Consider a skew product with overlapping blenders $F$ as above. Let $J\subset\bR$ be some compact interval containing $[p^+,q^+]\cup[p^-,q^-]$ in its interior. Then every skew product $G$ with fiber maps $g_1,\ldots,g_4\colon\bR\to\bR$ which are sufficiently \emph{$C^2$-close} to $h_1,\ldots,h_4$ in $J$, respectively, also is a skew product with overlapping blenders. In this case we simply say that $G$ is sufficiently $C^2$-close to $F$.
\end{remark}

We are interested in two particular subsets of $F$-invariant measures. To define them, consider the canonical projection 
\begin{equation}\label{canproj}
	\Pi\colon\Sigma_4\times \RR\to \RR,\quad
	\Pi(\omega,x)\eqdef x.
\end{equation} 
Recall \eqref{defBn} and the definition of limit measures $\fB_{\pm\infty}(\cdot)$ thereafter. Given an $F$-ergodic measure $\mu$, we denote by $W^\s(\mu)$ (respectively, $W^\u(\mu)$) the set of \emph{forward} (respectively, \emph{backward}) \emph{generic points} of $\mu$, that is, 
\begin{equation}\label{defgenericF}\begin{split}
	W^\s(\mu)
	&\eqdef \big\{(\omega,x)\in\Sigma_4\times\bR\colon\fB_\infty(\omega,x)=\mu\big\},\\
	W^\u(\mu)
	&\eqdef \big\{(\omega,x)\in\Sigma_4\times\bR\colon\fB_{-\infty}(\omega,x)=\mu\big\}.
\end{split}\end{equation}
We say that an $F$-invariant probability measure $\mu$ is \emph{equivalent in projection to the Lebesgue measure} if $\Pi_\ast(\mu)$ is equivalent to the Lebesgue measure in a restriction to some nonempty open subset of $\RR$. Below we will consider $F$-ergodic measures $\mu^\pm$ such that the intersection
\begin{equation}\label{large}
	\Pi(W^\s(\mu^+))\cap\Pi(W^\u(\mu^-))\cap[0,1]
\end{equation}
is ``large'' (Lebesgue-almost every point in $[0,1]$ belongs to this set). To guarantee that ergodic measures satisfying \eqref{large} exist, in what is below we require that the maps $h_1,\ldots,h_4$ are $C^2$ to be able to invoke Lasota-Yorke arguments (see Theorem \ref{t.L-Y}). 

In what follows, we let $\Sigma_2^+\eqdef \{1,2\}^\ZZ$ and    $\Sigma_2^-\eqdef \{3,4\}^\ZZ$ and denote by $\sigma^\pm$ the restriction of $\sigma$ to $\Sigma_2^\pm$. We consider the restrictions $F^\pm \colon \Sigma_2^\pm\times \bR \to \Sigma_2^\pm\times \bR$ of $F$ to the corresponding subspaces.

\begin{definition}[The sets of measures $\Leb(F^\pm)$]\label{defLebFpm}
We define $\Leb^+(F)$ as the set of $F^+$-invariant ergodic measures $\mu$ which are equivalent in projection to the Lebesgue measure and have the property that $\Pi(W^\s(\mu))$ contains a set of full measure inside some dense open subset of $[0,1]$.
 We analogously define the set of $F^-$-invariant ergodic measures equivalent in projection to Lebesgue and considering $\Pi(W^\u(\mu))$ instead, and denote it by $\Leb^-(F)$.
\end{definition}

\begin{remark}\label{remRRskew}
	For every pair of measures $\mu^-\in\Leb^-(F)$ and $\mu^+\in\Leb^+(F)$, we have
\[
	\Pi(W^\s(\mu^+))\cap\Pi(W^\u(\mu^-))\cap[0,1]
	\ne\emptyset.
\]	
\end{remark}

Our main result is the following that implies Theorem \ref{mmt.skew}. 

\begin{theorem}\label{t.skew} 
	Assume that $F$ is a step skew product with overlapping blenders and $C^2$ fiber maps. Then the convex hull of $\Leb^-(F)\cup\Leb^+(F)$ is contained in the closure of ${\cM_{\rm per}(F)}$. Moreover, this property holds for every skew product map sufficiently $C^2$-close to $F$.
\end{theorem}

Because of Remark \ref{remRobustskew}, the second assertion in Theorem \ref{t.skew} is immediate.

Given any pair of measures $\mu^-\in\Leb^-(F)$ and $\mu^+\in\Leb^+(F)$, the first step in the proof of  Theorem \ref{t.skew} is to show that both measures are ``related by a rich heterodimensional cycle'' and that hence, by Theorem \ref{t.cycle}, the segment of measures $[\mu^-,\mu^+]$ is accumulated by periodic measures. This would be an immediate consequence if we could deal with the skew product $F$ as if it would be a diffeomorphism restricted to some convenient set. This is why in Section \ref{secsmooth} we introduce smooth realizations $\fF$ of $F$. As by Remark \ref{bigremark} below any measure in $\Leb(F^\pm)$ can be viewed as an $\fF$-invariant one, we refrain from introducing any new notation for $\fF$-invariant measures. 

The first step to establish the cycle property is the following. 

\begin{proposition}\label{p.skew}
Assume that $F$ is a step skew product $C^2$ fiber maps and a $\cu$-blender (a $\cs$-blender). Then the set of measures $\Leb^+(F)$ (the set of measures $\Leb^-(F)$) is uncountable. Moreover, this property holds for every step skew product sufficiently $C^2$-close to $F$.
\end{proposition}

\begin{remark}[Rich heterodimensional cycles for skew products]\label{remRRbis}
	The concept of a rich heterodimensional cycle between measures can be formulated in the skew product setting. In this case, hyperbolicity of an ergodic measure must be restated in terms of fiber-hyperbolicity. Stable and unstable bundles correspond to the shift base-dynamics and hence ``partial hyperbolicity'' is automatic. The central Lyapunov exponent corresponds to the fiber-Lyapunov exponent. Stable and unstable invariant manifolds naturally translate to stable and unstable sets, similarly for the strong manifolds. Note that in this case transversality cannot be (and is not) required. For a formulation of heterodimensional cycles of sets, see \cite{DiaEstRoc:16}. We refrain from giving details. 

The property in Remark \ref{remRRskew} is analogous to property (4) in Definition \ref{defrichcycmeas}. On the other hand, the superposition hypotheses \eqref{e.superpositioncu} and \eqref{e.superpositioncs} correspond to (3) in Definition \ref{defrichcycmeas}. Indeed stronger intersection properties hold  (see Proposition \ref{p.blender} in the differentiable context) that allow to verify the occurrence of a rich heterodimensional cycle-property. 
\end{remark}

Remark \ref{p.skew} is the main step towards the following result proved in Section \ref{newsecProofLY}.

\begin{theorem}\label{tpropt.skew}
	Let $\fF$ be a smooth realization of a step skew product with $C^2$ fiber maps and overlapping blenders. Then every pair of measures $\mu^-\in\Leb^-(F)$ and $\mu^+\in\Leb^+(F)$ has a rich heterodimensional cycle.
\end{theorem}

Theorem \ref{t.skew} will be a consequence of Theorem \ref{tpropt.skew}, Theorem \ref{t.cycle}, and some additional convexity arguments (to deal with convex combinations of nonergodic measures). 

The remainder of this section is organized as follows. Smooth realizations of skew products are introduced in Section \ref{secsmooth}. In Section \ref{secacipLY}, we study special properties of the fiber maps of the skew product $F$ that give rise to measures that are absolutely continuous with respect to the Lebesgue measure and provide measures in $\Leb(F^\pm)$. In Sections \ref{secproofLYPro}, \ref{newsecProofLY}, and  \ref{secproofThm3} we prove Proposition \ref{p.skew}, and Theorems \ref{tpropt.skew} and \ref{t.skew}, respectively.

\subsection{Realization of the skew product by diffeomorphisms}\label{secsmooth}

We now explain how $F$ can be realized by a ``smooth analogue'' $\fF$ and provide a dictionary between them. We consider a  smooth realization of $F$ in a set of the form $\Sigma_4\times J$, where $J\subset \bR$ is an interval containing $[p^+,q^+]\cup [p^-,q^-]$ in its interior: there are a diffeomorphism $\fF\colon \bR^3\to \bR^3$ and a cube $[0,1]^2\times J$ so that the restriction of $\fF$ to this cube is of the form
$$
\fF(x,y,t)= \big(\ff(x,y), h_{x,y}(t)\big).
$$
Here, $\ff\colon \RR^2\to\RR^2$ is a diffeomorphism having a linear horseshoe on $R\eqdef[0,1]^2$  
conjugate to $\Sigma_4$ and there is small $\varepsilon>0$ such that
$$
	R\cap \ff^{-1}(R)
	=\bigcup_{i=1}^4 H_i, 
	\quad  \text{ where }\quad
	H_i \eqdef [0,1]\times [\frac i5-\varepsilon,\frac i5+\varepsilon].
$$
The explicit expression of $\ff$ is
$$
	\ff(x,y)
	= \Big( 2\varepsilon x +\frac i5-\varepsilon, \frac {y-\frac i5}{2\varepsilon}+ \frac12\Big),
	\quad \text{ if }  (x,y) \in H_i.
$$
The maps $h_{x,y}$ are defined by
$$
	h_{x,y}(t)
	=h_i(t) \quad \text{ for }y\in [\frac i5-\varepsilon,\frac i5+\varepsilon], \quad i=1,2,3,4,
$$
where $h_i$ are the maps in \eqref{sspF}. Compare Figure \ref{fig.blender}.

\begin{figure}[h] 
 \begin{overpic}[scale=.5]{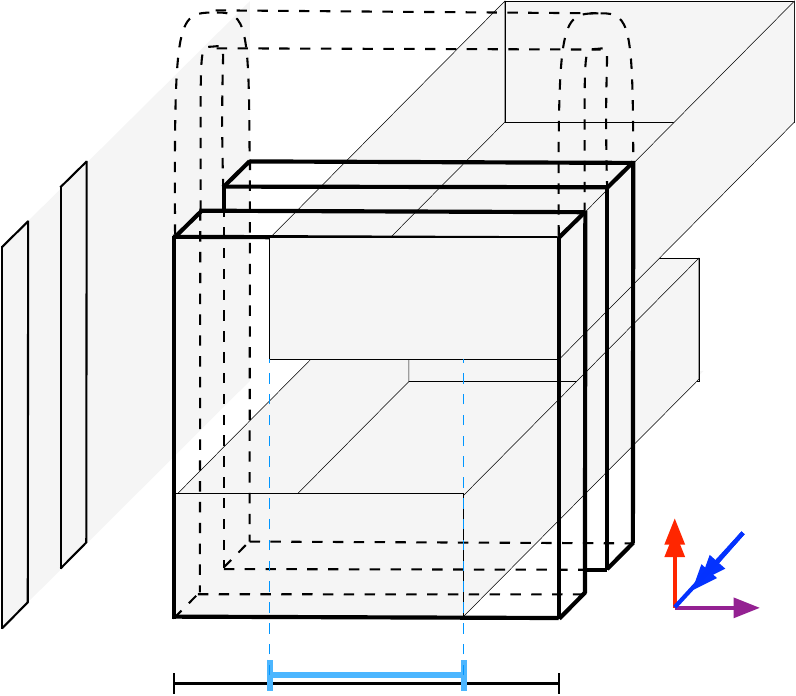}
	\put(2,64){\rotatebox{42}{\small{$R=[0,1]^2$}}}
  	\put(96,10){\small$t$}
  	\put(84,24){\small$y$}
  	\put(94,21){\small$x$}
 \end{overpic}
 \hspace{0.5cm}
  \begin{overpic}[scale=.35]{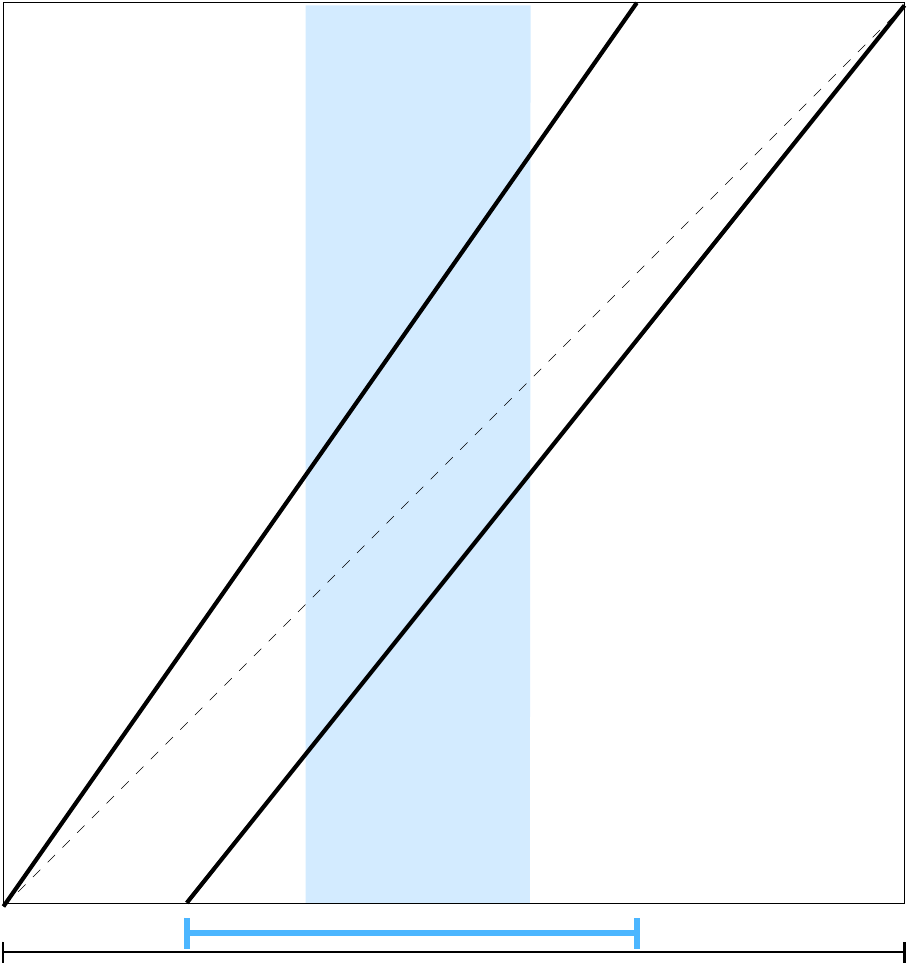}
 	\put(5,23){\small$h_1$}
	\put(84,79){\small$h_2$}
 	\put(0,-5){\small$p^+$}
  	\put(90,-5){\small$q^+$}
  	\put(15,-5){\small$h_2^{-1}(p^+)$}
  	\put(54,-5){\small$h_1^{-1}(q^+)$}
 \end{overpic}
  \caption{The smooth realization $\fF$}
  \label{fig.blender}
\end{figure}

We call $\fF$ a \emph{smooth skew product with overlapping blenders}.

In the remainder of this section, we can and will use either the skew product $F$ on $\Sigma_4\times J$ or the diffeomorphism $\fF$ on $\fR\times J$, depending on which point of view is more convenient. 

Consider the maximal invariant set for $\fF$,
\begin{equation}\label{e.thesetL}
	\Lambda\eqdef \bigcap_{n\in\ZZ}\fF^n\big([0,1]^2\times J\big),
\end{equation}
and also its subsets 
\begin{equation}\label{e.thesetsLpm}
	\Lambda^-\eqdef \bigcap_{n\in\ZZ}\fF^n\Big([0,1]\times [\frac 12,1]\times J\Big),\qquad
	\Lambda^+\eqdef \bigcap_{n\in\ZZ}\fF^n\Big([0,1]\times [0,\frac 12]\times J\Big).
\end{equation}
The restriction of  $\fF$ to $\Lambda^-$ is conjugate to $F^-$ on its maximal invariant set in $\Sigma_2^-\times J$, analogously for the restriction of  $\fF$ to $\Lambda^+$ and $F^+$ on $\Sigma_2^+\times J$.

We collect some useful properties of $\fF$ to get rich heterodimensional cycles between measures in $\Leb^-(F)$ and $\Leb^+(F)$. 

\begin{remark}[Summary of properties of the dynamics of $\fF$ in $\Lambda$]\label{bigremark}
$\,$

\begin{itemize}[leftmargin=0.6cm ]
\item[(1)] \textbf{Localization of measures.} 
Given any $t\not\in [p^+,q^+]$ then, for every $\omega\in\Si_2^+$, the second coordinate of the sequence $(F^+)^n(\omega,t)$ tends to $\pm\infty$ as $n\to +\infty$. Hence, the support of any $F^+$-invariant measure is contained in $\Si_2^+\times [p^+,q^+]$. The analogous statement holds for $F^-$ and the interval $[p^-,q^-]$. This immediately translates to the localization of the measures of $\fF$.

\item[(2)] \textbf{Uniformly hyperbolic sets.} We denote by $\fR$ the maximal invariant set of $\ff$ in $R$. Thus, $\fR$ is a hyperbolic basic set such that $\ff|_\fR$ is conjugate to $\sigma\colon \Si_4\to \Si_4$ and has a contraction rate $2\varepsilon$ in the first coordinate and an expansion rate $(2\varepsilon)^{-1}$ in the second one.  
Furthermore, the restriction of $\fF$ to $\fR\times J$ is conjugate to the restriction of $F$ to $\Si_4\times J$ (a set that contains the support of every $F$-invariant measures by the previous item).  This conjugacy projects to the identity map on $J$. In particular, it preserves the Lebesgue measure in the center direction.  Moreover, the canonical projection \eqref{canproj} corresponds to the projection $\fR\times J \to J$.

\item[(3)] \textbf{Partially hyperbolic splitting.}
We chose $\varepsilon>0$ such that the expansion/contraction rates dominate the derivative of the maps $h_i$ on $J$. Then there is a partially hyperbolic splitting with one-dimensional center $E^\ss\oplus E^\c\oplus E^\uu$ defined on $\fR\times J$, where $E^\ss\oplus E^\uu$ is the hyperbolic splitting for $\ff|_\fR$. 

\item[(4)] \textbf{Local strong invariant manifolds.} 
	For every $p=(x,y,t) \in \fR\times J$, it holds $[0,1]\times \{(y,t)\}\eqdef W^{\ss}_{r,L}(p)\subset W^\ss(p)$ and $\{x\}\times [0,1]\times\{t\}\eqdef W^\uu_{r,L}(p)\subset W^\uu(q)$.

\item[(5)] \textbf{Occurrence of blenders} 
The sets $\Lambda^+$ and $\Lambda^-$ are uniformly hyperbolic basic sets, with respect to $\fF$, where the center direction is uniformly expanding for $\Lambda^+$ and uniformly contracting for $\Lambda^-$. With the terminology introduced in \cite{BonDia:12}, the sets $\Lambda^+$ and $\Lambda^-$ are referred to as $\cu$-blender and $\cs$-blender horseshoes, respectively.

\item[(6)] \textbf{$\s$-index of $F$-invariant measures.} 
Every $F^\pm$-invariant measure can be viewed as an $\fF$-invariant measure. By a slight abuse of notation, we use the same symbol for both of them. Consider the correspondingly defined \emph{Lyapunov exponent} (relative to the bundle $E^\c$). For later purposes, we define it for any \emph{$\fF$-invariant} (not necessarily ergodic) measure. Let
\begin{equation}\label{defLyapexp}
	\chi^\c(\mu)
	\eqdef \int \log\,\lVert D\fF|_{E^\c}\rVert\,d\mu.
\end{equation}
Any ergodic measure supported on $\Lambda^-$ has negative $\chi^\c$-exponent and, in particular,  is uniform hyperbolic with $\s$-index two. Analogously, $\Lambda^+$ supports ergodic measures with positive $\chi^\c$-exponent and which are uniformly hyperbolic with $\s$-index one.

Every measure in $\mu^-\in\Leb^-(F)$ corresponds to one supported on $\Lambda^-$. Similarly, every measure in $\mu^+\in\Leb^+(F)$ corresponds to one supported on $\Lambda^+$. Therefore, the measures $\mu^\pm$ are uniformly hyperbolic and have different $\s$-indices. We refrain from introducing the notation $\Leb(\fF^\pm)$ and, by a slight abuse of notation, use $\Leb(F^\pm)$ instead.
\end{itemize}
\end{remark}

For the sake of completeness, we establish an intersection property that is in the core of the definition of a blender (similar results can be found in \cite{BonDiaVia:95}). 

\begin{proposition}[Intersection property of blenders]\label{p.blender} 
	For every $t\in [p^+,q^+]$  and $x\in [0,1]$ the segment $\{x\}\times [0,1]\times \{t\}$ contains a point $z\in W^\ss(a)$ for some $a\in\Lambda^+$.

Similarly, for every $t\in [p^-,q^-]$  and every $y\in [0,1]$ the segment $[0,1]\times \{(y,t)\}$ contains a point in $W^\uu(b)$ for some $b\in \Lambda^-$.
\end{proposition}

\begin{proof}
Observe first that given $p=(x,y,t) \in \fR\times J$ and writing $(x', y', z')=\fF(p)$ it follows
\begin{equation}\label{usethis}\begin{split}
	\fF\big([0,1]\times \{(y,t)\}\big) &\subset [0,1]\times \{( y', t')\},\\
	\fF^{-1}\big(\{ x'\}\times [0,1]\times \{ t'\}\big)&\subset \{x\}\times [0,1] \times\{ t\}.
\end{split}\end{equation}

We only prove the first part of the assertion. Fix $(x,t)$ as in the statement. By the superposition assumption \eqref{e.superpositioncu}, the point  $t\in [p^+,q^+]$ has an image either by $h_1$ or by  $h_2$ in $[p^+,q^+]$. This means that there is $\omega\in\Si_2^+$ such that $t_n=\Pi((F^+)^n(\omega,t) )\in[p^+,q^+]$. We let $t_0=t$. One deduces that there is a sequence $x_n$ in $[0,1]$, such that $x_0=x$ and that
\[
 	\fF(\{x_n\}\times [0,1]\times \{t_n\}) \cap (H_1\cup H_2)\times J )
	\supset\{x_{n+1}\}\times [0,1]\times \{t_{n+1}\}.
\]
Consider 
\[
 V_n \eqdef \fF^{-n}(\{x_n\}\times [0,1]\times \{t_n\})
\]
and observe that this is nested sequence of segments contained  $\{x\}\times [0,1]\times \{t\}$ whose length decreases exponentially. Thus 
$$
\bigcap_n V_n = \{z\}, \quad  z=(x,y,t).
$$  
By construction,
 $\fF^n(z)\in (H_1\cup H_2)\times J $ 
for every $n\geq 0$. This implies that the $\omega$-limit set of $z$ is contained in the maximal invariant set in $(H_1\cup H_2)\times J $,  which is $\Lambda^+$. In other words, $z\in W^\s(\Lambda^+)$.
As $\Lambda^+$ is a hyperbolic basic set which is expanding in the center direction, this implies that there is $a\in \Lambda^+$ with $z\in W^\s(a)=W^\ss(a)$.
\end{proof}

In the next two results, given $p\in \Lambda$ we use the notation 
\begin{equation}\label{notation}
	p=(x_p,y_p,t_p).
\end{equation}

\begin{corollary} \label{c.homorelperpoints} 
Let $a=(x_a,y_a,t_a)\in\Lambda$ be a hyperbolic periodic point of $\s$-index one with $t_a\in [p^+,q^+]$.  
Then there is $x\in (0,1)$ such that $\{x\}\times [0,1]\times J$ is contained in the (two-dimensional)  
unstable manifold $W^\u(a)$. 

Similarly, for every hyperbolic periodic point $b=(x_b,y_b,t_b)\in\Lambda$ of $\s$-index two with $t_b\in[p^-,q^-]$, there is $y\in (0,1)$ such that $[0,1]\times \{y\}\times  J$ is contained in the
(two-dimensional)  stable manifold $W^\s(b)$.
\end{corollary}

\begin{proof} 
We only prove the first part. By  Proposition~\ref{p.blender}, $W^\uu_{\rm loc}(a)$ intersects the strong stable manifold of some point of $\Lambda^+$. Note that the two-dimensional unstable manifold $W^\u(a)$ contains a strip $S=\{x_a\}\times [0,1]\times [t_a-\delta, t_a+\delta]$, for some small $\delta>0$.  As periodic points are dense in the hyperbolic basic set $\Lambda^+$, there is a periodic point $p=(x_p,y_p,t_p)\in\Lambda^+$ whose local stable manifold $W^\s_{\rm loc}(p)=W^\ss_{\rm loc}(p)$ intersects transversely $S$.  Now the $\lambda$-lemma implies that the positive iterates of $W^\u(a)$ accumulate on every compact part of $W^\u_{\rm loc}(p)$. Noting now that $\{x_p\}\times [0,1]\times J\subset W^\u(p)$, we deduce that sufficiently large iterates of the strip $S$ (corresponding to multiple of the period of $a$) contains a rectangle of the form $\{x\}\times [0,1]\times J$. By construction, this rectangle is contained in $W^\u(a)$.
\end{proof}

Given a periodic point $a$ of period $m$, denote 
\[
	\mu_a
	\eqdef \fB_m(a,F)
	= \frac1m\sum_{k=0}^{m-1}\delta_{F^k(a)}.
\]
Recall that two hyperbolic period points in $\Lambda$ of the same $\s$-index are \emph{homoclinically related} (relative to $\Lambda$) if the stable and the unstable manifolds of their orbits intersect cyclically and transversally in points in $\Lambda$. As a straightforward consequence of Corollary~\ref{c.homorelperpoints}  one gets the following. Indeed, the last assertion follows from result in \cite{Sig:74}.

\begin{corollary}\label{c.convexblender} 
$\,$
\begin{enumerate}[ leftmargin=0.7cm ]
\item Every pair of hyperbolic periodic points $b_1,b_2\in\Lambda$ of $\s$-index two and with $t_{b_1},t_{b_2}\in[p^-,q^-]$ are homoclinically related (relative to $\Lambda$).
\item Every pair of hyperbolic periodic points $a_1,a_2\in\Lambda$ of $\s$-index one and with $t_{a_1},t_{a_2}\in[p^+,q^+]$ are homoclinically related (relative to $\Lambda$).
\end{enumerate}
In particular, the segments of measures $[\mu_{a_1},\mu_{a_2}]$ and $[\mu_{b_1},\mu_{b_2}]$ are contained in the closure of $\cM_{\rm per}(\fF|_\Lambda)$.
\end{corollary}

\subsection{Associated Lasota-Yorke-type core dynamics}\label{secacipLY}

The results in this section are the key to prove that $\Leb(F^\pm)$ are uncountable sets (Proposition \ref{p.skew}). 

Recall the superposition region $(t_0,t_1)$ from Remark \ref{remsupersuper}. In this section, we study the dynamics (and associated invariant measures) of the following family of interval maps. Given  $T\in (t_0,t_1)$, let
\begin{equation}\label{defvarphiT}
	\varphi_T \colon [p^+,q^+]\to [p^+,q^+],\quad
	\varphi_T(t)\eqdef 
	\begin{cases}
		h_1(t)&\text{ if }t\in[p^+,T],\\
		h_2(t)&\text{ if }t\in(T,q^+]
	\end{cases}
\end{equation}
(compare Figure \ref{fig.acipA}). It follows from our hypotheses that $\varphi_T([p^+,q^+])=[p^+,q^+]$. Moreover, $\varphi_T$ is a uniformly expanding map with a discontinuity at $T$. 

\begin{figure}[h] 
 \begin{overpic}[scale=.4]{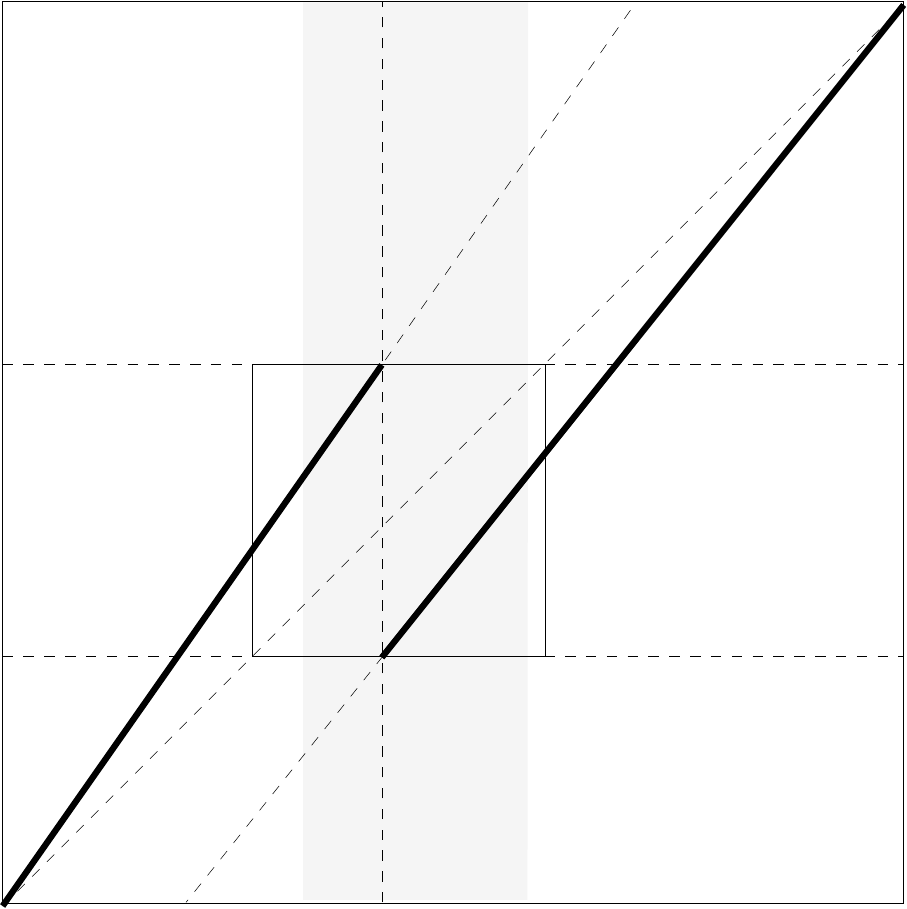}
	\put(5,18){\small$h_1$}
	\put(89,79){\small$h_2$}
 	\put(-2,-6){\small{$p^+$}}
 	\put(98,-6){\small{$q^+$}}
	\put(40,-6){\small$T$}
	\put(32,-6){\small$t_0$}
	\put(57,-6){\small$t_1$}
	\put(-40,59){\small$h_1(T)=\varphi_T(T)_-$}
	\put(-40,27){\small$h_2(T)=\varphi_T(T)_+$}
 \end{overpic}
  \caption{The map $\varphi_T$}
  \label{fig.acipA}
\end{figure}
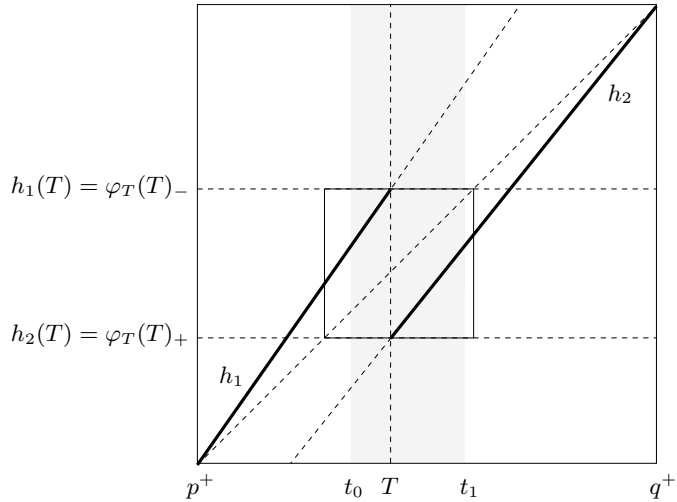

Note that similar type of associated maps, though with a slightly different focus, were introduced in \cite{NakNakToyYan:23} and named \emph{core dynamics}.

Let us start with some preliminary considerations. We first recall a classical key result that goes back to \cite{LasYor:73} (see also the more comprehensive discussion in \cite[Chapters 5--6]{BoyGor:97}). 

\begin{theorem}[{\cite{BoyGor:97}}]\label{t.L-Y} 
Let $\psi\colon [a,b]\to [a,b]$ be a piecewise $C^2$ interval map such that 
\[
	\min\{\psi'_+(t), \psi'_-(t)\}>\lambda>1
	\quad\text{ for every $t\in [a,b]$},
\]	 
where $\psi'_+(t)$ and $\psi'_-(t)$ are the right and left derivatives of $\psi$ at $t$, respectively. Then there is a $\psi$-invariant ergodic probability  $\mu$ such that $\supp(\mu)$ is an interval and that $\mu$ is equivalent to the Lebesgue measure in $\supp(\mu)$.
\end{theorem}

Note that any map $\varphi_T$ as defined in \eqref{defvarphiT} satisfies the hypotheses of Theorem \ref{t.L-Y}.

In what follows,  given a piecewise expanding map $\psi$ as in Theorem~\ref{t.L-Y}, we denote by $\cM_{\mathrm{eqLeb}}(\psi)$ the set of $\psi$-ergodic invariant probability measures which are equivalent to the Lebesgue measure in some nonempty open interval. 

\begin{proposition}\label{p.theseteqLeb}
$\,$
\begin{enumerate}[ leftmargin=0.7cm ]
\item For every $T\in (t_0,t_1)$ the set $\cM_{\mathrm{eqLeb}}(\varphi_T)$ is a singleton $\{\mu_T\}$.
\item For every  $T_1, T_2 \in (t_0,t_1)$, $T_1\ne T_2$, it holds $\mu_{T_1} \ne \mu_{T_2}$.
\end{enumerate}
\end{proposition}

\begin{proof}
We start with some preliminary results. Fix some $T\in (t_0,t_1)$.

\begin{claim}\label{cl.forward}
Let $I\subset [p^+, q^+]$ be an interval with nonempty interior. Then there are a subinterval $H\subset I$ and $m=m(T,H)\geq 0$ such that $T\in \interior(\varphi_T^m(H))$ and $\varphi_T^m$ is a diffeomorphism on $\interior(H)$.
\end{claim}

\begin{proof} 
Just note that the length of the interval $\interior(\varphi_T^n(I))$ grows exponentially until $T\in \interior(\varphi_T^n(I))$. Now it is enough to take some appropriate subinterval $H$ of $I$.
\end{proof}

Let  $\varphi_T(T)_+  \eqdef \lim_{t\to T} h_2(t)$ and $\varphi_T(T)_-\eqdef\lim_{t\to T} h_1(t)$ (compare Figure \ref{fig.acipA}).

\begin{claim} \label{cl.invariant}
 The interval $[h_2(T), h_1(T)]= [\varphi_T(T)_+, \varphi_T(T)_-]$ is $\varphi_T$-invariant.
\end{claim}

\begin{proof}
Note that in the interval $[h_2(T), T]$ one has  $\varphi_T= h_1>\id$, thus $h_1([h_2(T), T]) \subset   [h_2(t), h_1(T)]$. Similarly, in the interval $(T, h_1(T)]$  one has $\varphi_T= h_2<\id$, therefore $h_2((T, h_1(T)) \subset   (h_2(T), h_1(T)]$. This implies the assertion.
 \end{proof}

\begin{claim}
There is an open and dense subset of points $t\in[p^+,q^+]$ such that there is $n(t)>0$ so that $(\varphi_T)^{n(t)}\in [h_2(T), h_1(T)]$ for every $n\geq n(t)$.
\end{claim}

\begin{proof}
By Claim~\ref{cl.forward}, any open interval $I\subset [p^+,q^+]$ contains an interval $J$ such
that $\varphi^{m}_T(J)=[h_2(T), h_1(T)]$ for some $m\in\bN$. Now it suffices to apply Claim~\ref{cl.invariant}.
\end{proof}

As the map $\varphi_T$ satisfies the hypotheses of Theorem \ref{t.L-Y}, $\cM_{\mathrm{eqLeb}}(\varphi_T)$ is nonempty. Pick some measure $\mu_T\in\cM_{\rm eqLeb}(\varphi_T)$. Note that $\mu_T$ is $\varphi_T$-ergodic and equivalent to the Lebesgue measure in some nonempty open interval $I_T\subset[0,1]$.

\begin{lemma}\label{l.mainstep}
It holds $\supp(\mu_T)=[h_2(T),h_1(T)]$.
\end{lemma}

\begin{proof} 
Let $H$ and $m$ be provided by Claim \ref{cl.forward} applied to the interval $I_T$. This implies that there are $\mu_T$-generic points in the $\varphi_T$-invariant interval $[h_2(T), h_1(T)]$ (recall Claim~\ref{cl.invariant}). This implies that
\[
 	\supp(\mu_T)\subset [h_2(T), h_1(T)].
\]
The next iterate  $\varphi_T^{m+1}(H)$ is the union of two intervals, one starting at $h_2(T)$ and the other one at $h_1(T)$. This proves that $\inf\supp(\mu_T)= h_2(T)$ and  $\sup\supp(\mu_T)= h_1(T) $. The assertion now follows from Theorem \ref{t.L-Y}.
 \end{proof}

To prove the proposition, suppose that $\cM_{\mathrm{eqLeb}}(\varphi_T)$ contains two measures $\mu_T$ and $\nu_T$. Then these measures must be equivalent in some interval of the form $(h_2(T),h_2(T)+\varepsilon)$, for some small $\varepsilon>0$. From $\varphi_T$-ergodicity, it follows that $\nu_T=\mu_T$.
 
Finally, if $T_1, T_2 \in (t_0,t_1)$, $T_1\ne T_2$, then the measures have different support and therefore are different. This completes the proof of the proposition. 
\end{proof}

\begin{remark}\label{r.skewproductF-}
We can obtain similar results considering the fiber maps $h_3^{-1}$ and $h_4^{-1}$ of the skew product $F^-$.
\end{remark}

\begin{remark}[Further types of measures in $\Leb(F^\pm)$]
In this section, the goal was to show that the set of measures $\Leb(F^\pm)$ is large. For that, we focused on special measures arising from Lasota-Yorke results. Probably, there are further classes of measures with this property. Let us briefly present another class of measures whose nature is different as they arise from stationary measures of an iterated function system. We will present the class of so-called \emph{Bernoulli convolutions}. Although this setting is very specific as it requires \emph{affine} fiber maps, it illustrates the complexity of the class of measures $\Leb(F^\pm)$. The discussion of possible further examples studying non-affine generalizations is beyond the scope of this paper. In \cite{Sol:23}, there is described the state of the art of the so-called transversality method which was developed to prove absolute continuity. We also point out the results in \cite{Bie:23,BarSimSolSpi:} which fits well into the spirit of this paper.  

As in the setting of stationary measure, one typically focuses on contractions, let us consider $\Leb^-(F)$. Consider the skew product \eqref{sspF} in a particular cases where $h_3(x)=\lambda x$ and $h_4(x)=\lambda x+\lambda-1$, where $\lambda\in(1/2,1)$. Applying each map with probability $1/2$, there is a unique associated stationary measure $\mu_\lambda$, that is, the measure satisfies $\mu_\lambda =2^{-1}(h_3)_\ast\mu_\lambda+2^{-1}(h_4)_\ast\mu_\lambda$. Moreover, its support equals the interval $[0,1]$. For certain parameters $\lambda$ this measure is absolutely continuous. Which parameters remains a challenging problem. Note that $\mu_\lambda$ is just the projection of the measure of maximal entropy for $F^-$ by $\Pi_\ast$ that, in the case of absolute continuity, is in $\Leb^-(F)$.  
\end{remark}

\subsection{Proof of Proposition~\ref{p.skew}}\label{secproofLYPro}

We prove the proposition for $F^+$. The proof for $F^-$ is analogous, replacing $h_1,h_2$ by $h_3^{-1}, h_4^{-1}$ and using Remark~\ref{r.skewproductF-}. Fix some $T\in (t_0,t_1)$.

\begin{lemma} \label{l.projectionofmeasures}
	There is an $F^+$-ergodic measure $\nu_T$ such that $\Pi_\ast(\nu_T)=\mu_T$ and hence is equivalent in projection to the Lebesgue measure.
\end{lemma}

\begin{proof}
Take a generic point $t$ for the measure preserving system $([p^+,q^+],\varphi_T,\mu_T)$ and consider a point $\omega=(\omega_i)_i \in \Sigma_4$ satisfying the following for every $i\in \bN$,
\[
	\omega_i
	=\begin{cases}
		1&\text{ if }(\varphi_T)^i(t)\in [h_2(T),T],\\
		2&\text{ if }(\varphi_T)^i(t)\in (T, h_1(T)].
	\end{cases}
\]
Now, consider any accumulation point $\widetilde\nu_T$ of the sequence 
\[
	\nu_{n,T}
	=\fB_n\big((\omega,t),F\big)
	=\frac 1n\sum_{k=0}^{n-1} \delta_{F^k(\omega,t)}.
\]
Note that $\widetilde\nu_T$ is $F$-invariant and that $\Pi_\ast(\widetilde\nu_T)=\mu_T$. Furthermore, since $\mu_T$ is $\varphi_T$-ergodic, almost every measure in the ergodic decomposition of $\widetilde\nu_T$ (relative to $F$) also projects to $\mu_T$. This implies the assertion.
\end{proof}

Let now $\nu_T$ as provided by Lemma \ref{l.projectionofmeasures} and consider
\[
	\Pi (W^\s(\nu_T))=
	\{t\in \RR\colon
	 \mbox{there is }\omega\in \Sigma_2^+\text{ such that }(\omega,t)\in W^\s(\nu_T)\}.
\]

\begin{lemma} \label{lem222}
	The set $[p^+,q^+]\setminus \Pi (W^\s(\nu_T))$ has zero Lebesgue measure.
\end{lemma}

\begin{proof} 
Let us first show an auxiliary result.

\begin{claim}\label{cl.denseinI}
For every open interval $I=(a,b)\subset [p^+,q^+]$ there are $\omega\in \Si_2^+$ and $n_0\in\bN$ such that
\[ 
	[p^+,q^+] \subset \Pi \big((F^+)^n (\{\omega\}\times I)\big) \quad 
	\mbox{for every $n\ge n_0$.}
\]
\end{claim}

\begin{proof} 
The consequence \eqref{eqqconsequence} implies that for every $t\in (p^+,q^+)$ we have either $h_1(t)\in (p^+,q^+)$ or $h_2(t)\in(p^+,q^+)$. Hence, there is $\omega=\omega(t) \in\Si_2^+$ such that
\[
\Pi(F^n (\omega,t))\in (p^+,q^+)
\quad \text{ for every }n \ge 0.
\]

Consider now $t\in (a,b)$, its associated $\omega$, and the sets $(F^+)^n(\{\omega\}\times I)$. Note that $\Pi ((F^+)^n(\{\omega\}\times (a,t)))$ and $\Pi((F^+)^n(\{\omega\}\times (t,b)))$ are two intervals whose length increase exponentially and that $\Pi((F^+)^n (\omega,t))\in (p^+,q^+)$ for every $n$. Thus, we have 
\[
 	[p^+,q^+] \subset \Pi ((F^+)^n \{\omega\}\times I) \quad 
	\mbox{for every $n$ large enough},
\]
proving the claim.
\end{proof}

As a consequence  of Claim~\ref{cl.denseinI}, given any interval  $I$  there are $n$, $\omega$, and an open interval $J_I\subset I$ such that $\Pi_\ast(\nu_T)$ is equivalent to Lebesgue in  $\Pi \big( (F^+)^n (\{\omega\}\times J_I) \big)$. In particular,
$$
J_I \setminus \Pi (W^\s(\nu_T))
$$
has zero Lebesgue measure. Let $U$ now be the union of all intervals $J_I$ for a countable family of intervals $I$ forming a base of the topology of $(p^+, q^+)$. By construction, $U$ is open and dense in $[p^+,q^+]$ and $U\setminus W^\s(\nu_T)$ has zero Lebesgue measure. As $[p^+,q^+]\setminus U$ is countable and hence has zero Lebesgue measure, the assertion of the lemma follows.
\end{proof}

\begin{corollary}
	$\nu_T\in\Leb^+(F)$.
\end{corollary}

\begin{proof}
Recall $[p^+,q^+]\supset[0,1]$. Hence, Lemmas \ref{l.projectionofmeasures} and \ref{lem222}  imply the assertion. 
\end{proof}

The assertion that $\Leb^+(F)$ is uncountable is a consequence of the fact that, by Proposition \ref{p.theseteqLeb}, the measures $\mu_T$, $T\in (t_0,t_1)$, are pairwise distinct. This proves the proposition.
\qed

\subsection{Proof of Theorem \ref{tpropt.skew}}\label{newsecProofLY}

We need to check conditions (1)-(4) in Definition \ref{defrichcycmeas}. 

Condition (1) follows from Remark~\ref{bigremark} item (6). Moreover, $\supp (\mu^\pm) \subset \Lambda^\pm$. 
The partially hyperbolic set involved in the cycle is the set $\Lambda$ defined in \eqref{e.thesetL}. As this set contains the sets $\Lambda^\pm$ in \eqref{e.thesetsLpm} and, by Remark~\ref{bigremark} item (6), condition (2) holds taking $\Lambda$.

By Remark \ref{bigremark} item (4), local strong unstable manifolds of points in $\Lambda^+$ are segments of the form $\{x\}\times [0,1]\times \{t\}$. Local stable manifolds of points in $\Lambda^-$ are of the form $[0,1]\times \{y\}\times J$.  Therefore, these invariant manifolds intersect (transversely) in points in the partially hyperbolic set $\Lambda$. In the same way, local strong stable manifolds of points in $\Lambda^-$ are segments of the form $[0,1]\times \{(y,t)\}$ and local unstable manifolds of points in $\Lambda^+$  are on the form $\{x\}\times [0,1]\times J$. Therefore these invariant manifolds transversely intersect in points in $\Lambda$. This proves condition (3).

It remains to check condition (4): there are a $\mu^+$-generic point $p$ and a $\mu^-$-generic point $q$ such that $W^\s(p)\cap W^\u(q)\cap\Lambda\neq \emptyset$. Indeed, by Proposition~\ref{p.skew}, the sets $\Pi(W^\s(\mu^+))$ and $\Pi(W^\u(\mu^-))$ have each full Lebesgue measure in the reference interval $[0,1]$. In particular, there is some $t\in \Pi (W^\s(\mu^+))\cap \Pi(W^\u(\mu^-))$. Thus, there are $(x^+,y^+),(x^-,y^-)$ such that $p=(x^+,y^+,t)$ is $\mu^+$-generic and $q=(x^-,y^-,t)$ is $\mu^-$-generic.
Now it follows from Remark \ref{bigremark} item (4) that $(x^-,y^+,t)\in W^\s(p)\cap W^\u(q)\cap\Lambda$. 

This concludes the proof of the theorem.\qed

\subsection{Proof of Theorem~\ref{t.skew}}\label{secproofThm3}

Throughout this section, assume that $F$ is a skew product with overlapping blenders. Consider a smooth realization $\fF$ of $F$. 

\begin{lemma}\label{lemproc.segment} 
For every pair of measures $\mu^+\in \Leb^+(F)$ and $\mu^-\in \Leb^-(F)$, every measure in $[\mu^-, \mu^+]$  is the limit of ergodic measures supported on periodic orbits of $\Lambda$.
\end{lemma}

\begin{proof}
By Theorem~\ref{tpropt.skew}, we can apply Theorem~\ref{t.cycle} to the measures $\mu^\pm$ which implies that any measure in $[\mu^-,\mu^+]$ is accumulated by a sequence of hyperbolic periodic measures in the weak$\ast$ topology. The only point that is not explicitly stated in Theorem~\ref{t.cycle} is that these periodic orbits are contained in $\Lambda$. It follows from Remark \ref{bigremark} item (1) that every such periodic measure is in the maximal invariant set in $[0,1]^2\times([p^+,q^+]\cup[p^-,q^-])$. As $J$ contains $[p^+,q^+]\cup[p^-,q^-]$ in its interior and, by construction \eqref{e.thesetL}, $\Lambda$ is the maximal invariant set in $[0,1]^2\times J$, the assertion follows.
\end{proof}

We are now ready to provide the

\begin{proof}[Proof of Theorem~\ref{t.skew}]
Let $\mu_0$ be an $F$-invariant measure  in the convex hull of $\Leb^+(F)\cup \Leb^-(F)$.  Consider its Lyapunov exponent $\chi^\c(\mu_0)$ defined in \eqref{defLyapexp} and note that it is just the integral of a continuous function. Assume that $\chi^\c(\mu_0)\leq 0$ (the case where $\chi^\c(\mu_0)>0$ is analogous and hence omitted). Note that, within this convex hull, the measure $\mu_0$ is the limit of measures $\mu$ satisfying $\chi^\c(\mu)<0$. Let us discuss the two possible cases:
\medskip

\noindent\textbf{(1) $\mu_0$ is a convex combination of ergodic measures in $\Leb^-(F)$.}
We refrain from giving full arguments. We just observe that any $\fF$-ergodic measure of $\s$-index two can be weak$\ast$-approximated by periodic ones. Verify that any associated periodic orbit has $\s$-index two and is contained in $\Lambda$ (see arguments in the proof of Lemma \ref{lemproc.segment} above). By Corollary \ref{c.convexblender}, all such periodic points are homoclinically related. Hence, it follows that $\mu_0$ is accumulated by periodic measures supported on orbits in $\Lambda$ (see for example  \cite[Theorem 2]{BocBonGel:18}).
\medskip

\noindent\textbf{(2) $\mu_0$ is a convex combination of ergodic measures in $\Leb^-(F)$ and in $\Leb^+(F)$.}
In this case, the assertion is the direct consequence of the following result.

\begin{proposition}\label{p.convex} 
Consider a measure 
\[
 	\mu
	\eqdef \sum_{i=1}^m\alpha^+_i \mu^+_i+\sum_{j=1}^n\alpha^-_j\mu^-_j,
\]
where  $\mu^-_1,\ldots,\mu^-_n\in \Leb^-(F)$ and $\mu^+_1,\ldots,\mu^+_m \in \Leb^+(F)$ and $\alpha^-_1,\ldots,\alpha^-_n$ and $\alpha^+_1,\ldots,\alpha^+_m$ are positive numbers such that
$
 	\sum_{i=1}^m\alpha^+_i+\sum_{j=1}^n\alpha^-_j=1.
$
Assume that $\chi^\c(\mu)<0$. Then there is a sequence of $\fF$-periodic orbits contained in $\Lambda$ whose corresponding measures converge to $\mu$.
\end{proposition}

\begin{proof}
Note that, 
$$
\chi^\c(\mu)=\sum_{i=1}^m\alpha^+_i \chi^+_i+\sum_{j=1}^n\alpha^-_j\chi^-_j,
\quad
\mbox{where $\chi^\pm_r=\chi^\c(\mu^\pm_r)$.}
$$
Applying Proposition~\ref{p.convexsums} to $\lambda=\chi^\c(\mu)$ taking $A_i^\pm=\alpha_i^\pm$ and $\lambda_i^\pm=\chi_i^\pm$, we get numbers $a_{i,j}\in[0,1]$ and $b_{i,j}\geq 0$ such that:
\begin{itemize}[leftmargin=0.8cm ]
 	\item $\chi_{i,j}\eqdef a_{i,j}\chi_i^+ + (1-a_{i,j})\chi_j^-<0$  for every $1\leq i\leq m$, $1\leq j\leq n$,
 	\item $\sum b_{i,j}=1$,
 	\item $\alpha_i^+= \sum_{j=1}^n b_{i,j} a_{i,j}$ for every $1\leq i\leq n$,
 	\item $\alpha_j^-= \sum_{i=1}^m b_{i,j} (1-a_{i,j})$ for every $1\leq j\leq m$.
\end{itemize}
Consider the measures $\mu_{i,j}\eqdef a_{i,j}\mu^+_i+ (1-a_{i,j})\mu^-_j$.  Then
\begin{itemize}[leftmargin=0.8cm ]
	\item $\chi^\c(\mu_{i,j})=\chi_{i,j}<0$,
	\item $\mu=\sum_{i,j} b_{i,j} \mu_{i,j} 
	=\sum_{i,j} b_{i,j}\big(a_{i,j} \mu^+_i + (1-a_{i,j})\mu^-_j\big).$
\end{itemize}

Note that $\mu_{i,j}\in[\mu^+_i,\mu^-_j]$ and we apply Lemma~\ref{lemproc.segment} to $\mu_{i,j}$. Hence, there is a sequence of hyperbolic periodic points $(p_{i,j,n})_{n\in\bN}\subset\Lambda$ whose corresponding periodic measures $\delta_{i,j,n}$ satisfy
\begin{equation}\label{e.convergencedieracs}
	\delta_{i,j,n} \to \mu_{i,j}
	=a_{i,j}\mu^+_i+ (1-a_{i,j})\mu^-_j.
\end{equation}
As $\chi^\c(\mu)$, for $\mu$ supported on $\Lambda$, is just the integral of the continuous function $a\mapsto\log\, \lVert D\fF|_{E^\c(a)}\rVert$, this number depends continuously on measures supported in $\Lambda$. Thus, 
$$
\chi^\c(\delta_{i,j,n})\to  \chi^\c(\mu_{i,j})=\chi_{i,j}<0.
$$
As a consequence, $\chi^\c(\delta_{i,j,n})<0$ for every $n$ large enough.
 
It follows from Remark \ref{bigremark} item (1), that $\mu_j^-$ is supported in  $[0,1]^2\times [p^-,q^-]$.
As $\delta_{i,j,n}$ ``approaches a positive proportion of $\mu^-_j$", the orbits $p_{i,j,n}$ must have points in $[0,1]^2\times [p^-,q^-]$. Hence, by Corollary~\ref{c.convexblender}, the orbits of the periodic points $p_{i,j,n}$ are pairwise homoclinically related (relative to $\Lambda$).
As a consequence, the convex hull of the collection of measures $\{\delta_{i,j,n}\colon n\in\bN,i,j\}$ is contained in the closure of the periodic measures supported in a homoclinic class (relative to $\Lambda$). 
It follows that for each $n$ there is a sequence of periodic measure $(\delta_{n,k})_k$ supported on $\Lambda$ such that 
 $$
 \delta_{n,k} \to
 \sum_{i,j} b_{i,j}\delta_{i,j,n}.
$$
By \eqref{e.convergencedieracs} and since  $\sum_{i,j}b_{i,j} \mu_{i,j}=\mu$, we obtain a sequence $(k_n)_n$ such that that the periodic measures $(\delta_{n,k_n})_n$ converge  to $\mu$. This ends the proof of Proposition~\ref{p.convex}.
\end{proof}

This proves the theorem.
\end{proof}

\subsection{From skew products to diffeomorphisms: an open question}\label{sec:acip}

In view of the results in this section, the first natural question that arises is the following. 

\begin{question} 
Is there a version of Theorem \ref{t.skew} for a partially hyperbolic diffeomorphism having two blenders of different types of hyperbolicity related by a robust heterodimensional cycle of co-index one? 
\end{question}

To answer this question would require a good understanding of the measures supported on a blender. Certainly, some differentiability hypothesis would be needed, and perhaps some extra assumptions on the ``superposition and intersection properties'' and the type of the blenders. One critical issue here is the existence of ergodic measures satisfying ``absolute continuity-like properties'' similar to the measures in $\Leb(F^\pm)$ (Definition \ref{defLebFpm}). Note that, compared with the Plykin saddle-SRB measures, there are no normally hyperbolic surfaces where the holonomies associated to the strong foliations behave nicely. Thus, a critical issue is to control the holonomies. We now discuss this in more detail providing a technical description of the problem. 

For simplicity, this discussion is done in dimension three. Following \cite[Section~3]{BonDia:12}, the definition of a $\cu$-blender-horseshoe $\Gamma$ involves a superposition region (\cite[Definition 3.1]{BonDia:12}): a $C^1$-open family $\mathcal{D}$ of embeddings of one-dimensional disks intersecting the local stable manifold of the blender $\Gamma$. In fact, the family $\cD$ consists in all the segments tangent to a strong unstable cone field and ``crossing'' the blender. Consider now the ``strips'' $S$ tangent to the center-unstable cone and foliated by disks in $\mathcal{D}$ giving rise to a foliation $\mathcal{F}$. When this foliation is absolutely continuous with respect to Lebesgue, then we say that \emph{$(S, \mathcal{F})$ is an absolutely continuous pair}.

\begin{question}\label{q.generalLY}
Let $\Gamma$ be a $\cu$-blender-horseshoe with superposition region $\mathcal{D}$. Does $\Gamma$ support an ergodic measure $\mu_\Gamma$ such that, for every every absolutely continuous pair $(S, \cF$),
the union of the leaves of $\cF$ intersecting the stable manifold of a generic point for $\mu_\Ga$ has total Lebesgue measure in $S$?
\end{question}

Again, to address this question some high regularity would be necessary. Answering it positively would be a first step to get robust cycles of measures in robust heterodimensional cycles.

\section{Non-rich cycles: examples}\label{appnonrich}
The purpose of this section is to illustrate that without the assumption of richness, heterodimensional cycles involving a pair of measures $\mu^\pm$ do not guarantee any convexity property. We present a variety of results demonstrating different scenarios where either no measure, only one measure, or a subinterval of measures in $(\mu^-, \mu^+)$ are accumulated by ergodic measures. We do not provide an exhaustive study but rather present pertinent examples in simple settings. For that, we consider heterodimensional cycles involving fixed points $p$ and $q$ of different types of hyperbolicity, and examine the cycle of measures corresponding to the Dirac measures $\delta_p$ and $\delta_q$. 

In Section~\ref{ss.nonrich}, we perform a semi-local study considering the dynamics in a neighborhood of the cycle. We present two types of non-rich cycles. In the case of the so-called {\em{non-twisted cycles,}} the only measures supported on a neighborhood of the cycle are the two Dirac measures $\delta_p$ and $\delta_q$, see Proposition~\ref{propB2}. In the case of {\em{twisted cycles,}} there are infinitely many periodic measures supported in that neighborhood. However, they have a very specific feature: there is only one measure in  $(\delta_p, \delta_q)$ which is accumulated, see Proposition~\ref{p.twisted}. Further approximation possibilities are outlined in Remark \ref{remsparsecycle}. In Section~\ref{ss.global}, we present some global dynamics where no measure in $(\mu^-, \mu^+)$ is accumulated by ergodic measures.

\subsection{Non-rich heterodimensional cycles}\label{ss.nonrich}

We consider a diffeomorphism with two fixed points related by a heterodimensional cycle. 
According to the terminology in \cite{BonDiaKir:12}, this cycle may be {\em{twisted}} and {\em{non-twisted}}.
To discuss these two configurations simultaneously, we denote the diffeomorphism by $f_\varepsilon$, where $\varepsilon\in \{+,-\}$. We use the symbol ``$+$" when the cycle is non-twisted, and ``$-$" otherwise.

We assume that $f_\varepsilon$ has two fixed points $p$ and $q$ having linearizing neighborhoods 
$U_p$ and $U_q$, respectively, as follows: there are smooth coordinates $(x,y,z)_j=(x_j,y_j,z_j)$ on $U_j\cup f_\varepsilon(U_j)$, $j=p,q$ such that, in these coordinates,
\begin{itemize}[leftmargin=0.8cm ]
\item[(i)] $f_\varepsilon$ maps as $(x,y,z)_p \mapsto (2x, \frac 14 y, 4z)_p$ on $U_p$ (where $p=(0,0,0)_p$),
\item[(ii)] $f_\varepsilon$ maps as $(x,y,z)_q\mapsto (\frac12 x, \frac 14 y, 4z)_q$ on $U_q$ (where $q=(0,0,0)_q$).
\item[(iii)] there are $n\in\bN$ and small $\delta >0$ such  that
\[
	f^n_\varepsilon (1+x,y,z)_p= (-1+x,\frac1{4^n}y ,4^n z)_q \mbox{ for every }x, y,z \in [-\delta, \delta]
\]	
and $f^i_\varepsilon(1+x,y,z)_p\notin U_p\cup U_q$ for every $i=1,\ldots,n-1$,
\item[(iv)] there is  $m\in\bN$ so that
\[
	f^m_\varepsilon(x,y,-1+z)_q= (\varepsilon x, -1+ \frac 14y, 4z)_p\mbox{ for every } x,y,z \in [-\delta,\delta],  
\]	 
and $f^i_\varepsilon(x,y,-1+z)_q\notin U_p\cup U_q$ for $i=1,\ldots,m-1$.
\end{itemize}

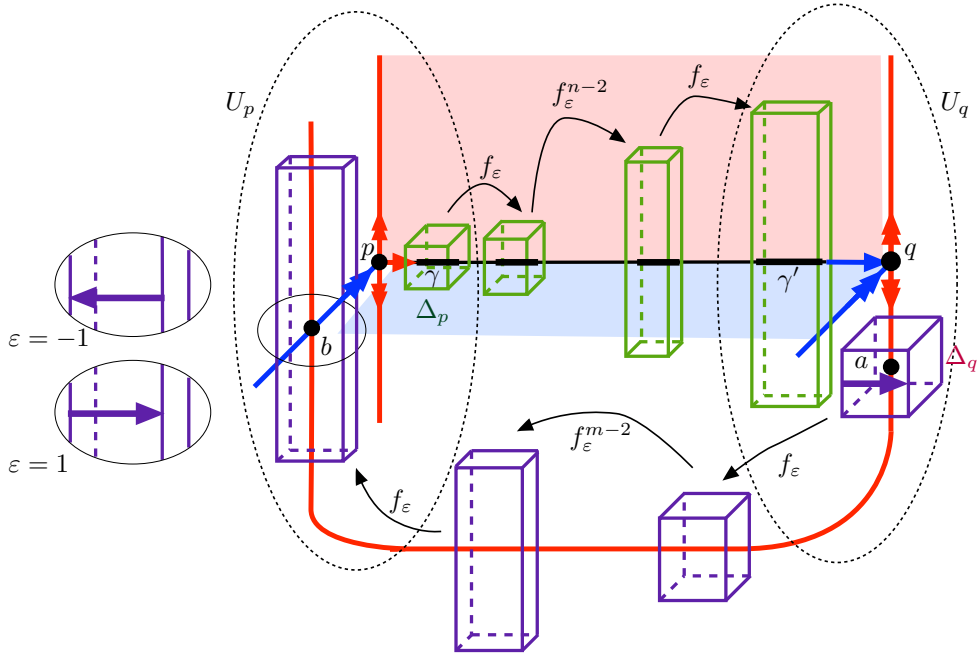
\begin{figure}[h] 
 \begin{overpic}[scale=.62]{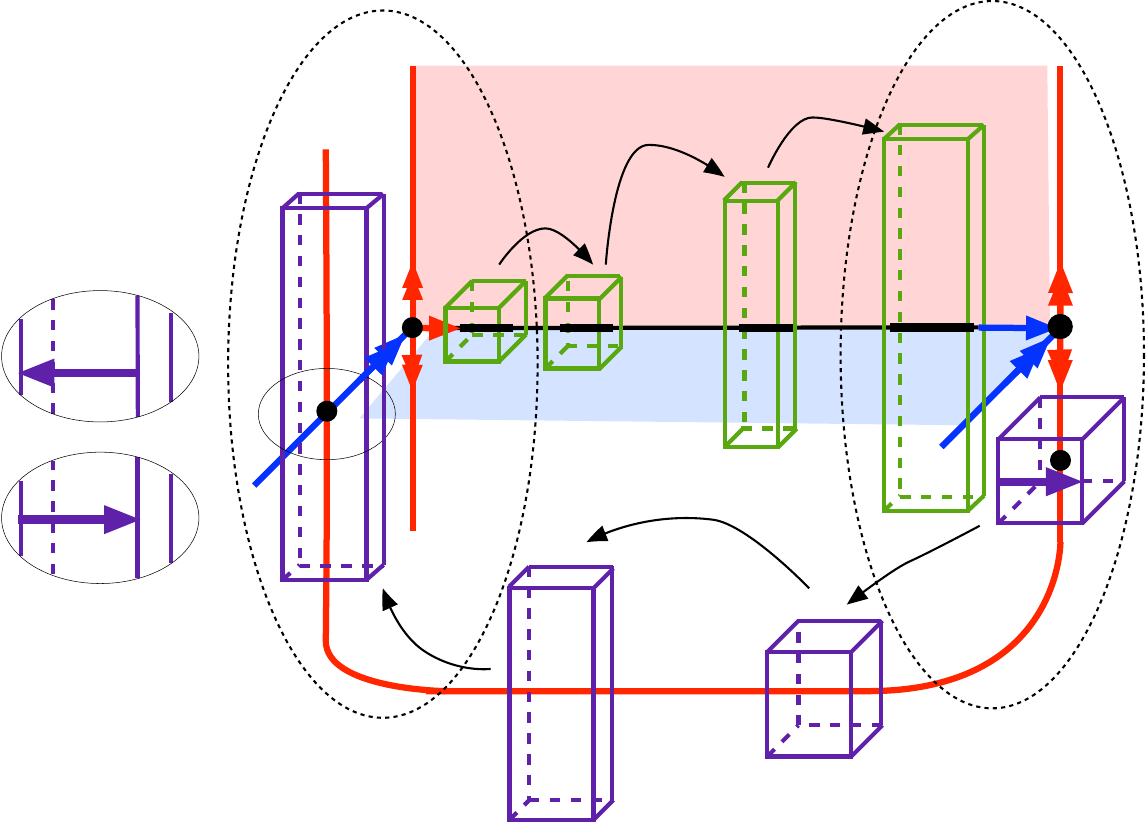}
 	\put(19,60){$U_p$}
 	\put(98,60){$U_q$}
	\put(94,44){$q$}
 	\put(34,44){$p$}
 	\put(-5,34){$\varepsilon=-1$}
 	\put(-5,20){$\varepsilon=1$}
	\put(41,41){$\gamma$}
	\put(80,40.5){$\gamma'$}
	\put(40,37){\textcolor{forestgreen}{$\Delta_p$}}
	\put(98.5,32){\textcolor{purple}{$\Delta_q$}}
	\put(47,53){$f_\varepsilon$}
	\put(55,61){$f_\varepsilon^{n-2}$}
	\put(70,63){$f_\varepsilon$}
	\put(57,23){$f_\varepsilon^{m-2}$}
	\put(37,16){$f_\varepsilon$}
	\put(80,20){$f_\varepsilon$}
	\put(88.5,31.5){$a$}
	\put(29.5,33){$b$}
\end{overpic}
\caption{Twisted and non-twisted heterodimensional cycles}
\label{figtwisted}
\end{figure}

\begin{remark}[Heterodimensional cycles]
Under these hypotheses, the segments 
\[
	\gamma\eqdef [1-\delta,1+\delta]_p\times \{(0,0)_p\}
	\quad\text{ and }\quad
	f^n_\varepsilon (\gamma)=\gamma'\eqdef [1-\delta,1+\delta]_q\times \{(0,0)_q\}
\]	
are contained in $W^\u(p)\pitchfork W^\s(q)$ and the points $a=(0,0,-1)_q$ and $b\eqdef f^m_\varepsilon(a)=(0,-1,0)_p$ are contained in $W^\u(q)\cap W^\s(p)$. Compare Figure \ref{figtwisted}. As the saddles have different indices, this implies that $f_\varepsilon$ has a heterodimensional cycle 
associated to these saddles. In the conditions above, (i) and (ii) determine the local dynamics at the saddle points, (iii) and (iv) describe the transition from  $p$ to $q$ following the transverse intersection $\gamma$, 
and from $q$ to $p$ following the quasi-transverse intersection point $a$, respectively.
\end{remark}

We consider the  \emph{cycle neighborhood} defined by
\[
	W_\varepsilon\eqdef \\
	U_p\cup U_q\cup \bigcup_{i=1}^{n-1} f_\varepsilon^i (\Delta_p) 
		\cup \bigcup_{j=1}^{m-1} f_\varepsilon^j (\Delta_q), 
\]
where 
$$
\Delta_p\eqdef  [1-\delta,1+\delta]_p\times [-\delta,\delta]_p^2 \quad \mbox{and} \quad
\Delta_q\eqdef [-\delta,\delta]_q^2\times [-1-\delta,-1+\delta]_q.
$$
We denote by $\Lambda_\varepsilon$  the maximal invariant set of $f_\varepsilon$ in  $W_\varepsilon$,
\[
	\Lambda_\varepsilon
	\eqdef \bigcap_{n\in\bZ}f_\varepsilon^n(W_\varepsilon).
\]	 

\begin{remark} 
	The restriction of $f_\varepsilon|_{W_\varepsilon}$ has  the partially hyperbolic splitting, 
	$E^{\ss}\oplus E^\c \oplus E^{\uu}$, whose  bundles $E^{\ss}$, $E^\c$, and $E^{\uu}$ are directed by the vector fields $\frac\partial{\partial y}$, $\frac\partial{\partial x}$, and $\frac\partial{\partial z}$, respectively.
\end{remark}

\begin{remark} 
The Dirac measures $\delta_p$ and $\delta_q$ at $p$ and $q$, respectively, are related by a heterodimensional cycle of $f_\varepsilon$. Taking $\Lambda= \Lambda_\varepsilon$ in Definition~\ref{defrichcycmeas}, we observe that this cycle is not rich: conditions (i), (ii), and (iv) hold, but condition (iii) is not satisfied. This is because the strong unstable and stable manifolds of the saddles are not involved in the cycle.
\end{remark}

The next remark states the main difference between twisted and non-twisted cycles.

\begin{remark}
The map $f_1$ preserves the orientation of the center bundle, whereas this is not the case for $f_{-1}$
(the transition $f_{-1}^m$ reverses the orientation). 
\end{remark}

Let us introduce some notation. In what is below, given $w\in U_r$, $r=p,q$, we denote by $(x_r(w), y_r(w), z_r(w))$ its local coordinate in the chart $U_r$. Given a set $A$, denote by
\[
	\cO^-(A)
	\eqdef \bigcup_{n\ge0}f_\varepsilon^{-n}(A)
	\quad\text{ and }\quad
	\cO^+(A)
	\eqdef \bigcup_{n\ge0}f_\varepsilon^n(A)
\]
its \emph{negative} and \emph{positive orbit}, respectively (with respect to $f_\varepsilon$). The \emph{orbit} of $A$ is $\cO(A)\eqdef\cO^-(A)\cup\cO^+(A)$. We also use this notation when $A$ is just a point.

\subsubsection{The non-twisted case}

\begin{proposition} \label{propB2}
	The maximal invariant set of the diffeomorphism $f_1$ in the cycle neighborhood $W_1$ is 
	$$
	\La_1 = \{p\} \cup \{q\} \cup  \cO(\gamma) \cup  \cO(a).
	$$
	As a consequence, the unique $f_1$-invariant ergodic measures supported on $W_1$ are the Dirac measures $\delta_p$ and $\delta_q$.  In particular, there is no measure in the  segment $(\delta_p,\delta_q)$ 
	accumulated by ergodic measures supported on $W_1$.
\end{proposition}

\begin{proof} 
Note that $\{p\}$ and $\{q\}$ are the maximal invariant sets of $f_1$ in $U_p$ and $U_q$, respectively. 
The fact that $\cO(\gamma)\subset \Lambda_1$ is straightforward. Hence, what remains to show is that, besides the already verified points, $\Lambda_1$ only contains the orbit of $a$. 

\begin{remark}\label{r.mapoutside}
Every $w\in U_q$ with $x_q(w)>0$ is mapped outside of $W_1$ by negative iterations of $f_1$,  therefore  $w\notin \Lambda_1$.
Analogously, every $w\in U_p$ with  $x_p(w)<0$
does not belong to $\Lambda_1$.
\end{remark}

\begin{claim}\label{c.heteroorbit}
	Consider $w\not\in \gamma$. If
	\begin{itemize}
	\item
	$w\in U_q\cap\Lambda_1$ and $\cO^+(w)\not\subset U_q$, then $w\in\cO(a)$;
	\item
$w\in U_p\cap\Lambda_1$ and $\cO^-(w)\not\subset U_p$, then $w\in\cO(a)$.
\end{itemize}
\end{claim}

\begin{proof}
We only prove the first item, the proof of the second one is analogous and hence omitted. Consider a point $w$ as in the claim, clearly $w\ne q$. As $\Lambda_1$ is the maximal invariant set in $W_1$, there is $i\in\bN$ such that $f^j(w) \in U_q$ for every $j=0,\dots,i$ and $f^i(w) \in \Delta_q \subset U_q$. 
There are the following three possibilities for $x_q(w)$:
\smallskip

\noindent\textbf{Case $x_q(w)<0$.}
In this case, $x_q(f^i(w))<0$. As the cycle is non-twisted, by Remark~\ref{r.mapoutside} we have that $x_p(f^{i+m}(w))<0$, where $m$ is as in (iv) in the description of $f_1$. By Remark~\ref{r.mapoutside}, $w\not\in \Lambda_1$, hence, this case cannot occur.
\smallskip

\noindent\textbf{Case $x_q(w)>0$.}
Analogously to the previous case, we get $x_p(f^{i+m}(w))>0$ and hence there is $k\in\bN$ such that $x_q(f^{i+m+k}(w))<0$. Thus, we can apply the previous case to the point $w'=f^{i+m+k}(w)$ instead of $w$, getting a contradiction. Hence, this case cannot occur.
\smallskip

\noindent\textbf{Case $x_q(w)=0$.}
In this case, $x_q(f^i(w))=0$ and therefore $x_p(f^{i+m}(w))=0$.\smallskip

Studying analogously the possible values of the other coordinates of points in the orbit of $w$, we obtain the following. We have that $z_p(f^{i+m}(w)) = 0$, otherwise positive iterates of $f^{+m}(w)$ would eventually leave the set $W_1$. This implies that $f^{i+m}(w) \in W^\text{s}_{\text{loc}}(p)$. Similarly, we have that $y_q(f^i(w)) = 0$, otherwise negative iterates of $w$ will go out of $W_1$.
  Thus $f^i(w)\in W^\u_{\rm loc}(q)$ and $f^i(w)$ is the unique point of the intersection $W^\u_{\rm loc}(q)\cap f^{-m}(W^\s_{\rm loc}(p))$, that is, $f^i(w)=a$. This shows that $w\in\cO(a)$ and proves the claim. 
\end{proof}

The next result follows arguing as in the proof of Claim~\ref{c.heteroorbit}, hence its proof is omitted.

\begin{claim}
\label{c.othercase}
	 Let $w\in (U_p\cup U_q)\cap\Lambda_1$. 
	\begin{itemize}
	\item
	If $\cO^-(w)\subset U_p$ and $\cO^+(w)\cap \Delta_p\ne\emptyset$ then $w \in\cO(\gamma)$.
	\item
	If $\cO^+(w)\subset U_q$ and $\cO^-(w)\cap \Delta_q\ne\emptyset$ then $w \in\cO(\gamma)$.
	\end{itemize}
\end{claim}

The proposition follows from Claims~\ref{c.heteroorbit} and \ref{c.othercase}.
\end{proof}

\subsubsection{The twisted case}

\begin{proposition}
\label{p.twisted}
The maximal invariant set of the diffeomorphism $f_{-1}$ in the cycle neighborhood $W_{-1}$ contains periodic points whose corresponding measures converge in the weak$\ast$ topology to the measure $\frac12\delta_p+\frac12\delta_q$.  Moreover, none of the measures in $\{t\delta_p+(1-t)\delta_q \colon t\in(0,\frac12)\cup(\frac12, 1)\}$ is accumulated by ergodic measures.
\end{proposition}

We refrain from providing detailed proof and just give the main ingredients.

\begin{proof}[Skecth of the proof of Proposition~\ref{p.twisted}]   
The orbit of any point $w \in \Lambda_{-1}$, where $w \notin {p, q} \cup {p} \cup \mathcal{O}(\gamma) \cup \mathcal{O}(a)$, has an itinerary as follows: $r_0$ iterates in $U_p$, followed by $n$ iterates in the orbit $\Delta_p$, $s_1$ iterates in $U_q$, and $m$ iterates in the orbit of $\Delta_q$, and then $r_1$ iterates in $U_p$, and so on (here $n$ and $m$ are as in (iii) and (iv) in the description of $f_\varepsilon$). In this way, we obtain a bi-infinite sequence $\ldots, s_{-1}, m, r_0, n, s_1, m, r_1, \ldots$ of natural numbers. These numbers satisfy $s_k = r_k$ for every $k \in \mathbb{Z}$. Indeed, this is a consequence of the following facts:

\begin{itemize}[leftmargin=0.8cm ]
\item $x_p(f_{-1}^m(w))=-x_q(w)$ for every $w\in \Delta_q$
 \item $f_{-1}$ multiplies by $2$ the $x_p$-coordinates in $U_p$ and by $1/2$ the $x_q$-coordinates,
 \item each orbit enters in $U_q$ with the $x_q$-coordinate in the interval $(-1-\delta,-1+\delta)$ and leaves  $U_p$ with the $x_p$-coordinate in $(1-\delta,1+\delta)$.
\end{itemize}

If an ergodic measure $\mu$ is close to the segment $(\delta_p,\delta_q)$, then the itinerary of a generic point has, in average, very large numbers $r_i$ and $s_i$ and the ratio of these numbers determines the point in 
$[\delta_p,\delta_q]$ close to $\mu$.  As the proportion is $1$, $\mu$ is indeed close to  $\frac12\delta_p+\frac12\delta_q$. 

It remains see that $\frac12\delta_p+\frac12\delta_q$ can  be approximated by periodic measures.
Indeed, our precise choice of $f_{-1}$ implies that 
$$
f_{-1}^{k+m+k+n}(1,y_0,z_0)_p=(1,y_0,z_0)_p
$$
if  $y_0,z_0$ are such that $f_{-1}^{k+n}(1,y_0,z_0)_p\in \Delta_q$.  A straightforward calculation shows that, for every sufficiently large enough $k$ there are $y_k,z_k$ such that $f_{-1}^{k+m+k+n}=(1,y_k,z_k)_p$. This provides a periodic point $r_k\in \Lambda_{-1}$. By the above, the sequence of measures
supported on the orbits of the points $r_k$ tends to $\frac12 \delta_p+\frac12\delta_q$ as $s\to\infty$.
\end{proof}

\begin{remark}
	The fact that the only point in $(\delta_p,\delta_q)$ which is accumulated by ergodic measures is its middle point because the center Lyapunov exponent of the limit measure is zero.
	In a more general scenario, but still maintaining the linearized setting of the cycle for simplicity, 
	if $\alpha_p>1$ and $\alpha_q<1$ are the center eigenvalues of the saddles in the cycle, the unique point in $(\delta_p,\delta_q)$ accumulated by ergodic measures would be
\[
	\frac{\log \alpha_p}{\log \alpha_p +|\log \alpha_q|}\delta_p+\frac{|\log \alpha_q|}{\log \alpha_p+|\log \alpha_q|}\delta_p,
\]	
that is, the value for which the averaged center Lyapunov exponent vanishes.
\end{remark}

\begin{remark}[Sparse heterodimensional cycle]\label{remsparsecycle}
For twisted cycles, and still in a semi-local setting, the following phenomena occur in a heterodimensional cycle of measures $\delta_p$ and $\delta_q$. There is a measure $\mu\in(\delta_p,\delta_q)$ such that the only measures in this segment which are accumulated by ergodic ones are the ones in $(\delta_p,\mu)$. Indeed,  in \cite{DiaHorRioSam:09} the ``central exponent of $\mu$ is zero'', while in \cite{DiaGel:12} is nonzero.  
\end{remark}

\subsection{Non-rich cycles: global settings}\label{ss.global}
We now sketch the construction of global dynamics having a heterodimensional cycle associated to the Dirac measures of a pair of saddles, $\mu^-$ and $\mu^+$, such that no measure in $(\mu^-,\mu^+)$ is accumulated by ergodic ones.

We revisit the construction in \cite{BonDia:12b} from a slightly different perspective. As the setting in  \cite{BonDia:12b} is quite elaborate, let us focus on its aspects that are relevant for us. This construction begins with an auxiliary Morse-Smale vector field $X$ defined on the three-sphere $\mathbb{S}^3$. After a surgery, employing an identification quotient map $\Psi$, one obtains a diffeomorphism $F_\Psi$ on $\mathbb{S}^2 \times \mathbb{S}^1$, defined as the time-one map $X_1$ taking the quotient by $\Psi$ (see \cite[Section 3]{BonDia:12b}). There is a fundamental domain $\Delta$ of $F_\Psi$ (independent of $\Psi$), however  the map $F_\Psi$ does depend on $\Psi$. For a fixed $\Psi$, there are two subsets $\EE^s(F_\Psi)$ and $\EE^u(F_\Psi)$ of $\Delta$  whose intersection contains the chain recurrent set of $F_\Psi$ in $\Delta$,  \cite[Corollary 3.5]{BonDia:12b}. When the intersection $\EE^s(F_\Psi)\cap \EE^u(F_\Psi)$ is just a point then  $F_\Psi$ has a heterodimensional cycle associated to a saddle $p_1$ of $\s$-index two and a saddle $q_1$ of $\s$-index one. In such a case, the construction allows to control the dynamics outside the cycle: points ``far" from the cycle tend to an attractor in the future and to a repeller in the past.

In \cite{BonDia:12b} the map $\Psi$ is chosen to get a cycle. This occurs when the intersection $\EE^s(F_\Psi)\cap\EE^u(F_\Psi)$ has nonempty interior. Here we alter this strategy choosing $\Psi$ such that $\EE^s(F_\Psi)\cap\EE^u(F_\Psi)$ consists of a unique point.  By \cite[Lemma 3.6, item 1]{BonDia:12b},  this intersection belongs to the unstable manifold of $p_1$ and to the stable manifold of $q_1$. The global properties of the map $X_1$ imply that these two saddles are involved in a heterodimensional cycle. Moreover, the chain recurrent set of $F_\Psi$ equals the set of the fixed points of $X_1$ (which is finite since we started with a Morse-Smale vector field) together with the heteroclinic orbits relating $p_1$ and $q_1$ giving rise to the cycle. Note that the latter set does not support any invariant measure. Thus, the only ergodic measures are the ones supported on fixed points. As this set is finite, there are no ergodic measures accumulating on $(\delta_{p_1}, \delta_{q_1})$.

Finally, \cite[Sections 4 and 5]{BonDia:12b} considers perturbations of the diffeomorphisms $F_\Psi$. One may deduce that the situation above (heterodimensional cycles where the measures in the interval $(\delta_{p_1},\delta_{q_1})$ cannot be approached by ergodic measures) corresponds to a (local) codimension-one submanifold of the set of diffeomorphisms. 

\appendix
\section{Convex sums}\label{appA}

In this section, we prove the following auxiliary result that we first describe in words.
Given a finite collection of negative and positive numbers having a negative average $\lambda$, each attached with a certain weight, one can rewrite $\lambda$ as an average of averages of \emph{pairs} of numbers, one positive and one negative, such that each such average is negative and that the total weight of each number equals the initial one (compare \eqref{eqconsequence} below). Since we did not find any reference, for completeness, we prove this fact.

\begin{proposition}\label{p.convexsums}
Consider real numbers $\lambda_1^-, \dots,\lambda_m ^-<0$ and $\lambda_1^+, \dots,\lambda_n^+>0$.
Assume that there are nonnegative numbers $\A_1^-,\ldots,\A_ m^-$ and $\A_1^+,\ldots,\A_n^+$ so that
\begin{itemize}[leftmargin=0.8cm ]
  \item[(i)] $\lambda\eqdef  \sum_{i=1}^m  \A_i^-\lambda_i^-+\sum_{j=1}^n\A_j^+\lambda_j^+<0$,
  \item[(ii)] $\sum_{i=1}^m \A_i^-+\sum_{j=1}^n\A_j^+=1$.
 \end{itemize}
Then there are $a_{ij}\in[0,1]$ and $b_{ij}\geq 0$ so that:
\begin{itemize}[leftmargin=0.8cm ]
 \item[(1)] for any $1\leq i\leq m $, $1\leq j\leq n$, 
 \[\lambda_{ij}\eqdef (1-a_{ij})\lambda_i^-+a_{ij} \lambda_j^+<0,\]
 \item[(2)] $\sum_{j=1}^n\sum_{i=1}^m  b_{ij}=1$,
 \item[(3)] one has
 \[
 	\A_i^-= \sum_{j=1}^nb_{ij} (1-a_{ij})\quad\text{ for }1\leq i\leq m,\quad\quad
	\A_j^+= \sum_{i=1}^m  b_{ij} a_{ij}\quad\text{ for }1\leq j\leq n
\]
\end{itemize}
and hence
\begin{equation}\label{eqconsequence}
	\sum_{j=1}^n\sum_{i=1}^m  b_{ij}
	\big((1-a_{ij}) \lambda_i^- + a_{ij}\lambda_j^+ \big)=\lambda.
\end{equation}
\end{proposition}

Note that \eqref{eqconsequence} is an immediate consequence of the assertions (1)--(3).

We first prove this proposition in the case $n=1$.

\begin{lemma}\label{l.1} 
Consider real numbers $\lambda_1^-, \dots,\lambda_m ^-<0$ and $\lambda^+>0$. Assume that there are nonnegative numbers $\A_1^-, \ldots,\A_m ^-$ and $\A^+$ so that
\begin{itemize}[leftmargin=0.8cm ]
  \item[(i)] $\lambda\eqdef  \sum_{i=1}^m  \A_i^-\lambda_i^-+\A^+\lambda^+<0$,
  \item[(ii)] $\sum_{i=1}^m \A_i^-+\A^+ =1$.
 \end{itemize}
Then there are $a_i\in[0,1]$ and $b_i\geq 0$ so that:
\begin{itemize}[leftmargin=0.8cm ]
 \item[(1)] $\lambda_i\eqdef   (1-a_i)\lambda_i^-+a_i\lambda^+<0$ for every  $1\leq i\leq m $,
 \item[(2)] $\sum_{i=1}^m b_i=1$,
 \item[(3)] $\A_i^-= b_i (1-a_i)$ for every $1\leq i\leq m$ and 
 	$\A^+= \sum_{i=1}^m  b_i a_i $.
\end{itemize}
As a consequence,
\begin{equation}\label{eqtriaangle11}
	\sum_{i=1}^m  b_i\lambda_i 
	= \sum_{i=1}^m  b_i \big((1-a_i)\lambda^-_i+a_i\lambda^+\big)
	=\lambda.
\end{equation}
\end{lemma}

\begin{proof}
Note that \eqref{eqtriaangle11} follows immediately from the other assertions of the lemma.

We first consider the following particular case.

\smallskip\noindent\textbf{Case $\lambda^-_i \leq \lambda<0$ for every $i$}.
In this case, the lemma follows taking
\[
	a_i
	\eqdef\frac{\lambda-\lambda^-_i}{\lambda^+-\lambda^-_i}\in (0,1)
	\quad\text{ and }\quad
	b_i\eqdef\frac{c_i}{\sum_{k=1}^m  c_k},
	\quad\text{ where }\quad
	c_i\eqdef \frac{\A^-_i}{1-a_i}>0.
\]	
With this choice we have 
\begin{equation}\label{eqtriaangle}
	\lambda_i
	\eqdef  (1-a_i)\lambda^-_i+a_i\lambda^+
	=\lambda<0,
\end{equation}
which proves item (1).
Moreover, we get
\begin{equation}\label{eqtriaangle2}
	\sum_{i=1}^m  b_i=1, 
	\quad
	\A^-_i= c_i(1-a_i),
	\quad\text{ and }\quad
	b_i(1-a_i)= \frac{\A^-_i}{\sum_{k=1}^m  c_k}.
\end{equation}
This proves item (2).
Summing over $i=1,\ldots,m$, the latter implies
\begin{equation}\label{eqtriaangle3}
 	\frac{\sum_{i=1}^m  \A^-_i}{\sum_{k=1}^m  c_k}  + \sum_{i=1}^m  a_i b_i 
	= \sum_{i=1}^m \big((1-a_i)b_i+a_ib_i\big)
	= \sum_{i=1}^m  b_i 
	=1.
\end{equation}	
Taking in \eqref{eqtriaangle} the average with respect to the weights $\{b_i\}_i$, with \eqref{eqtriaangle2} we get
\[
	\sum_{i=1}^m  b_i\lambda_i
	= \sum_{i=1}^m  b_i (1-a_i)\lambda^-_i+(\sum_{i=1}^m  a_i b_i)\lambda^+ 
	= \lambda\sum_{i=1}^m  b_i
	=\lambda.
\]
Using \eqref{eqtriaangle2}, we can rewrite
\[
	\lambda= \sum_{i=1}^m  \frac{\A^-_i}{\sum_{k=1}^m  c_k}\lambda^-_i
			+(\sum_{i=1}^m  a_i b_i)\lambda^+  
	=  \Big(\frac{\sum_{i=1}^m  \A^-_i}{\sum_{k=1}^m  c_k}\Big)
		\frac{\sum_{i=1}^m  \A^-_i\lambda^-_i}{\sum_{i=1}^m  \A^-_i}
		+ \Big(\sum_{i=1}^m  a_i b_i\Big)\lambda^+.
\]
Note that by \eqref{eqtriaangle3} the latter is a convex sum of $\sum_i \A^-_i\lambda^-_i/\sum_i \A^-_i$ and $\lambda^+$ with value $\lambda$ (whose coefficients $\sum_i \A^-_i/\sum_k c_k$ and $\sum_i a_ib_i$ are uniquely defined).
On the other hand, by hypothesis,
\[
	\lambda
	= \sum_{i=1}^m  \A^-_i\lambda^-_i+ \A^+\lambda^+ 
	= \frac{\sum_{i=1}^m  \A^-_i}{1} \frac {\sum_{i=1}^m  \A^-_i\lambda^-_i}{\sum_{i=1}^m  \A^-_i}
		+ \A^+\lambda^+.
\]
Comparing coefficients, this implies  $\sum_k  c_k=1$ and $\sum_i  a_i b_i=\A^+$. Hence, \eqref{eqtriaangle2} gives $b_i(1-a_i)=A^-_i$. This implies item (3).
This proves the lemma in this case.

\medskip\noindent\textbf{General case.}
Assume now that $\lambda< \lambda^-_{j_0}<0$ for some index $0\leq j_0\leq m $; for notational simplicity we assume $j_0=m$. 

We now argue by induction on $m$. For $m=1$, the assertion of the lemma is straightforward. Let us now suppose that $m>1$ and that the assertion was shown for $m-1$. 

Let us first draw some auxiliary consequences of hypotheses (i) and (ii). By (i),
\begin{eqnarray}
	0>\lambda
	&=& \sum_{i=1}^m  \A_i^-\lambda_i^-+ \A^+\lambda^+ \notag\\
 	&=& \A_m ^-\lambda^-_m  
		+(1-\A^-_m )\Big(\sum_{i=1}^{m-1}  \frac{\A_i^-}{1-\A_m ^-}\lambda_i^-
			+\frac{\A^+}{1-\A_m ^-}\lambda^+ \Big).\label{eqtriaangle7}
\end{eqnarray}
Note that hypothesis (ii) implies
\begin{equation}\label{eqtriaangle6}
	\sum_{i=1}^{m-1}  \frac{\A_i^-}{1-\A_m ^-}+\frac{\A^+}{1-\A_m ^-} 
	= \frac{\sum_{i=1}^{m-1} \A_i^-+\A^+}{1-\A_m ^-}
	=\frac{1-\A_m ^-}{1-\A_m ^-}  
	= 1.
\end{equation}
Let
\begin{equation}\label{eqtriaangle5}
	0> \widetilde \lambda
	\eqdef \frac{\lambda-\A^-_m \lambda^-_m }{1-\A^-_m }
	=	\sum_{i=1}^{m-1}  \frac{\A_i^-}{1-\A_m ^-}\lambda_i^-
		+ \frac{\A^+}{1-\A_m ^-}\lambda^+ .
\end{equation}
Letting 
\begin{equation}\label{eqtriaangle4}
	 \widetilde \A^-_i
	 \eqdef\frac{\A^-_i}{1-\A_m ^-},\quad
	 \widetilde \A^+
	\eqdef\frac{\A^+}{1-\A_m ^-},
\end{equation}
together with \eqref{eqtriaangle5} and \eqref{eqtriaangle6}, we get	
\[
	0
	> \widetilde \lambda
	=\sum_{i=1}^{m-1} \widetilde\A^-_i \lambda^-_i
		+  \widetilde \A^+ \lambda^+ 
	\quad\text{ and }\quad	
	\sum_{i=1}^{m-1} \widetilde \A^-_i+ \widetilde \A^+
	=1.
\]	
Thus, applying the induction hypothesis to the above, the lemma is true for the numbers $\widetilde A_1^-,\ldots,\widetilde A_{m-1}^-$ and $\widetilde A^+$ and hence there are  $\widetilde a_i\in[0,1]$ and $\widetilde b_i\geq 0$ so that:
\begin{equation}\label{consequence0}\begin{split}
	&\text{($\widetilde 1$)}\quad 0> \widetilde \lambda_i\eqdef (1-\widetilde a_i)\lambda_i^-+\widetilde a_i\lambda^+\text{for every }1\leq i\leq m -1,\\[-0.2cm]
	&\text{($\widetilde 2$)}\quad \sum_{i=1}^{m-1}  \widetilde b_i=1,\\[-0.2cm]
	&\text{($\widetilde 3$)}\quad
	\widetilde \A_i^-
	= \widetilde b_i (1-\widetilde a_i)\text{ for every }1\leq i\leq m -1,
	\quad\quad
	\widetilde \A^+= \sum_{i=1}^{m-1}  \widetilde b_i\widetilde a_i
\end{split}\end{equation}
and hence
\[
	\sum_{i=1}^{m-1}  \widetilde b_i\widetilde \lambda_i 
	= \sum_{i=1}^{m-1} \widetilde b_i\big((1-\widetilde a_i)\lambda^-_i+\widetilde a_i\lambda^+\big)
	= \widetilde \lambda.
\]	

From \eqref{eqtriaangle7}, it follows that
\[\begin{split}
	\lambda
	&=\A_m ^-\lambda^-_m  
		+(1-\A^-_m )\Big(\sum_{i=1}^{m-1}  \frac{\A_i^-}{1-\A_m ^-}\lambda_i^-
						+ \frac{\A^+}{1-\A_m ^-}\lambda^+ \Big)\\
	\text{\small{(by \eqref{eqtriaangle4} and \eqref{consequence0})}}\quad					
	&= \A^-_m \lambda^-_m 
		+ \sum_{i=1}^{m-1}  (1-\A^-_m )\widetilde b_i \big((1-\widetilde a_i)\lambda^-_i
										+ \widetilde a_i\lambda^+\big) .
\end{split}\]

Let us now show the assertions (i)--(iii) of the lemma hold taking the following numbers
 \[\begin{split}
	&b_i\eqdef(1-\A^-_m )\widetilde b_i\text{ for every }1\leq i\leq m -1,\quad b_m \eqdef\A^-_m,\\
	&a_i\eqdef\widetilde a_i\text{ for every }1\leq i\leq m -1,\quad a_m \eqdef0.
\end{split}\]
Indeed, with this choice,  one has
\[\begin{split}
	&(1-a_i)\lambda_i^-+a_i\lambda^+
	=(1-\widetilde a_i)\lambda_i^-+\widetilde a_i\lambda^+
	=\widetilde\lambda_i<0
	\quad \text{for $1\le i\le m-1$},\\
	&(1-a_m)\lambda_m^-+a_m\lambda^+
	=\lambda_m^-<0,
\end{split}\]
which proves item (i). Together with \eqref{consequence0} item (2), we get
\[
	\sum_{i=1}^m  b_i
	= \sum_{i=1}^{m-1}  b_i+b_m
	=(1-\A^-_m )\sum_{i=1}^{m-1} \widetilde b_i +\A^-_m
	= (1-\A^-_m )+\A^-_m =1,
\]	
which proves item (ii). Finally, by \eqref{eqtriaangle4} and \eqref{consequence0},
\[\begin{split}
	&A_i^-
	= (1-\A^-_m ) \widetilde \A^-_i
	=   (1-\A^-_m )\widetilde b_i(1-\widetilde a_i)
	= b_i (1-a_i)
	\quad \text{for $1\le i\le m-1$},\\
	&A_m^-
	= b_m =b_m (1-a_m ),
\end{split}\]
which proves the first half of item (iii). Moreover, by \eqref{eqtriaangle4}, we can write
\[\begin{split}
	\A^+&=(1-\A^-_m )\widetilde\A^+
	=(1-\A^-_m )\sum_{i=1}^{m-1} \widetilde b_i\widetilde a_i
	= \sum_{i=1}^{m-1}  b_ia_i+0
	=\sum_{i=1}^m  b_i a_i,
\end{split}\]
which proves the second half of item (iii). 
This finishes the proof by induction.
\end{proof}

The next lemma is, somehow, analogous to  Lemma~\ref{l.1}.

\begin{lemma}\label{l.-1} 
Consider real numbers $\lambda^-<0$ and $\lambda^+_1,\dots, \lambda^+_n>0$.
Assume that there are nonnegative numbers $\A^-$ and $\A^+_1,\ldots,\A^+_n$ so that
\begin{itemize}[leftmargin=0.8cm ]
\item[(i)] $\lambda= \A^-\lambda^-+ \sum_{j=1}^n\A_j^+\lambda_j^+ <0$,
\item[(ii)] $\A^- +\sum_{j=1}^n\A_j^+=1$.
 \end{itemize}
Then there are $a_j\in[0,1]$ and $b_j\geq 0$ so that:
\begin{itemize}[leftmargin=0.8cm ]
 \item[(1)]  $\lambda_j
	\eqdef (1-a_j)\lambda^-+a_j\lambda^+_j<0$ for every  $1\leq j\leq n$,
\item[(2)] $\sum_{j=1}^nb_j=1$,
\item[(3)] $\A_j^+= b_j a_j$  for every $1\leq j\leq n$ and $\A^-= \sum_{j=1}^nb_j (1-a_j)$.
\end{itemize}
As a consequence,
\[
	\sum_{j=1}^nb_j\lambda_j 
	= \sum_{j=1}^nb_j\big((1-a_j)\lambda^-+a_j\lambda^+_j\big)
	=\lambda.
\]
\end{lemma}

\begin{proof}
Note that Lemma~\ref{l.-1} does not follow from Lemma~\ref{l.1} just by exchanging the roles of $\lambda^+_j, \A^+_j, \lambda^-,\A^- $ and $\lambda^-_i, \A^-_i, \lambda^+,\A^+ $ because of the sign $\lambda_j<0$. However, this was not used in the proof of Lemma~\ref{l.1}, except for requiring  $\lambda^-_i\leq \lambda<0$ in order to get  $a_i\in(0,1)$ with $\lambda_i=(1-a_i)\lambda^-_i+a_i\lambda^+<0$, see \eqref{eqtriaangle}.

Under the hypotheses of Lemma~\ref{l.-1}, for any $j$ there is $a_j\in(0,1)$ so that $\lambda_j= (1-a_j)\lambda^-+a_j\lambda^+_j =\lambda\in (\lambda^-,0)$. Now the proof of the lemma is analogous to the first part of the proof of Lemma~\ref{l.1}, after \eqref{eqtriaangle}.
\end{proof}


\begin{proof}[Proof of Proposition~\ref{p.convexsums}]
Let   
\begin{equation}\label{eqtriaangle17}
	\A^+
	\eqdef \sum_{j=1}^n \A_j^+
	\quad\mbox{ and }\quad
	\lambda^+
	\eqdef \frac{\sum_{j=1}^n \A^+_j\lambda_j^+}{\A^+}.
\end{equation}
With this notation, we get 
\[
	\lambda
	= \sum_{i=1}^m \A^-_i \lambda^-_i+\A^+\lambda^+ 
	<0
	\quad\text{ with }\quad
	\sum_{i=1}^m \A^-_i+\A^+ =1.
\]
Thus, we can apply Lemma~\ref{l.1} to get $a_i\in[0,1]$ and $b_i\geq 0$ so that:
\begin{equation}\label{consequence1}\begin{split}
	&\bullet \quad \lambda_i\eqdef  (1-a_i)\lambda_i^-+a_i\lambda^+<0\text{ for every }1\leq i\leq m,\\[-0.2cm]
	&\bullet \quad \sum_{i=1}^m b_i=1,\\[-0.3cm]
	&\bullet \quad \A_i^-= b_i (1-a_i)\text{ for every }1\leq i\leq m\text{ and  }\A^+= \sum_{i=1}^m  b_i a_i,\\[-0.2cm]
	&\bullet \quad \sum_{i=1}^m b_i\lambda_i =\lambda.
\end{split}\end{equation}	

Now, by the definition of $\lambda^+$ in \eqref{eqtriaangle17}, for every $1\leq i\leq m$ we have
\begin{equation}\label{eqtriaangle15}
	0
	>\lambda_i
	= (1-a_i)\lambda^-_i+a_i\lambda^+
	= (1-a_i)\lambda^-_i+a_i\sum_{j=1}^n\frac{ \A^+_j}{\A^+} \lambda^+_j.
\end{equation}
Recalling the definition of $A^+$ in \eqref{eqtriaangle17}, for every $i$
\begin{equation}\label{eqtriaangle18}
	(1-a_i)+\sum_{j=1}^n\frac{a_i \A^+_j}{\A^+} 
	= (1-a_i) +a_i\frac{\sum_{j=1}^n \A^+_j}{\A^+} 
	= 1-a_i+a_i\frac{\A^+}{\A^+}=1.
\end{equation}
Let 
\begin{equation}\label{eqtriaangle19}
	A^+_{ij}
	\eqdef \frac{a_i \A^+_j}{\A^+}\geq 0.
\end{equation}
With this notation and together with \eqref{eqtriaangle15} and \eqref{eqtriaangle18}, for each $1\le i\le m$ we get 
\[
	\lambda_i
	=(1-a_i)\lambda^-_i + \sum_{j=1}^nA^+_{ij}\lambda^+_j 
	<0 
	\quad\mbox{ and }\quad
	(1-a_i)+\sum_{j=1}^nA^+_{ij}=1.
\]	
Thus, for each $1\le i\le m$, we can apply Lemma~\ref{l.-1} to the numbers $\lambda^-_i$ and $\lambda^+_1,\ldots,\lambda^+_n$ and the nonnegative numbers $(1-a_i)$ and $A^+_{i1},\ldots,A^+_{in}$ getting $\alpha_{ij}\in[0,1]$ and $\beta_{ij}\geq 0$ so that
\begin{equation}\label{consequence2}\begin{split}
	&\bullet \quad \lambda_{ij}\eqdef(1-\alpha_{ij})\lambda^-_i+\alpha_{ij}\lambda^+_j <0\text{ for every  }1\leq j\leq n,\\[-0.2cm]
	&\bullet \quad \sum_{j=1}^n \beta_{ij}=1,\\[-0.3cm]
	&\bullet \quad (1-a_i)= \sum_{j=1}^n\beta_{ij} (1-\alpha_{ij})\text{ and }A_{ij}^+= \beta_{ij} \alpha_{ij}\text{ for every }1\leq j\leq n,\\[-0.2cm]
	&\bullet\quad 	\sum_{j=1}^n \beta_{ij}\lambda_{ij} =\lambda_i.
\end{split}\end{equation}	

We claim that taking 
\[
	b_{ij}\eqdef b_i\beta_{ij}\ge0,\quad
	a_{ij}\eqdef\alpha_{ij}\in[0,1]\quad\text{ for }\quad 0\leq i\leq m, 0\leq j\leq n
\]	
 the assertion of the proposition follows. First verify that with \eqref{consequence2}, for every $i,j$,
\[
	\lambda_{ij}
	= (1-\alpha_{ij})\lambda^-_i + \alpha_{ij}\lambda^+_j 
	<0,
\]
proving item (1). Next, by our choice, \eqref{consequence2} and \eqref{consequence1},
\[
	\sum_{j=1}^n\sum_{i=1}^m b_{ij}
	=\sum_{j=1}^n\sum_{i=1}^mb_i\beta_{ij}
	= \sum_{i=1}^mb_i\big(\sum_{j=1}^n\beta_{ij}\big)=\sum_{i=1}^m b_i
	=1,
\]
proving item (2). Further, check that for every $i$, also together with  \eqref{consequence2} and \eqref{consequence1},
\[
	\sum_{j=1}^nb_{ij}(1-a_{ij})
	= \sum_{j=1}^nb_i\beta_{ij}(1-\alpha_{ij})
	=b_i \A^-_i,
\]
proving the first half of item (3). To show the second half, it remains to prove the following.

\begin{claim}
 For every $1\leq j\leq n$, one has $\A_j^+= \sum_{i=1}^m  b_{ij} a_{ij}$. 
\end{claim}

\begin{proof}
Recall that, by \eqref{eqtriaangle19} and \eqref{consequence2}, 
\[
	a_i \A^+_j/\A^+=A^+_{ij}=\beta_{ij}\alpha_{ij}>0
	\quad\text{  and }\quad 
	b_{ij}=b_i\beta_{ij}.
\]	 
Hence, 
\begin{equation}\label{stup}
	b_{ij} \cdot a_{ij}=b_i\beta_{ij}\cdot\alpha_{ij}=b_ia_iA^+_j/A^+.
\end{equation}
One deduces that for every $i,k\in \{1,\dots,m\}$ and  $1\leq j\leq m $ one has:
\[
	\frac{b_{ij}a_{ij}}{b_{kj}a_{kj}}
	=\frac{b_ia_iA^+_j}{b_ka_kA^+_j}
	=\frac{b_ia_i}{b_ka_k}.
\]
From \eqref{stup}, one also deduces
\[
	\frac{\sum_{i=1}^m  b_{ij}a_{ij}}{\sum_{k=1}^m  b_{kj}a_{kj}}
	= \frac{\sum_{i=1}^mb_ia_i}{\sum_{k=1}^mb_ka_k}
	=1.
\]
Hence,  \eqref{stup} together with \eqref{consequence1} implies that for every $1\le j\le n$
\[
	\sum_{i=1}^mb_{ij}a_{ij}
	= \sum_{i=1}^mb_ia_i\frac{A^+_j}{A^+}
	= \frac{A^+_j}{A^+}\sum_{i=1}^mb_ia_i
	= \frac{A^+_j}{A^+}A^+
	= A^+_j,
\]
proving the claim.
\end{proof}

This concludes the proof of Proposition~\ref{p.convexsums}.
\end{proof}

\end{document}